\documentclass[reqno]{amsart}
\usepackage{amsmath, amsthm, amssymb, amscd}

\usepackage[a4paper,width=150mm,top=25mm,bottom=25mm,centering]{geometry}

\usepackage{enumitem}
\usepackage{graphicx} 
\usepackage{xcolor}
\usepackage{comment}
\usepackage{cases}
\usepackage{array}
\usepackage{multirow}

\usepackage{tikz}
\usepackage{pgf}
\usepackage{tikzscale}
\usepackage{tikz-cd}
\usetikzlibrary{decorations.pathreplacing,decorations.markings}
\usetikzlibrary{arrows, knots, external, positioning}

\makeatletter
\tikzset{column sep/.code=\def\pgfmatrixcolumnsep{\pgf@matrix@xscale*(#1)},
  row sep/.code   =\def\pgfmatrixrowsep{\pgf@matrix@yscale*(#1)},
  matrix xscale/.code=%
    \pgfmathsetmacro\pgf@matrix@xscale{\pgf@matrix@xscale*(#1)},
  matrix yscale/.code=%
    \pgfmathsetmacro\pgf@matrix@yscale{\pgf@matrix@yscale*(#1)},
  matrix scale/.style={/tikz/matrix xscale={#1},/tikz/matrix yscale={#1}}}
\def\pgf@matrix@xscale{1}
\def\pgf@matrix@yscale{1}
\makeatother

\usepackage{url}
\usepackage{caption} 

\newtheorem{theorem}{Theorem}[section]
\newtheorem{conjecture}[theorem]{Conjecture}
\newtheorem{lemma}[theorem]{Lemma}
\newtheorem{proposition}[theorem]{Proposition}

\newtheorem{claim}[theorem]{Claim}

\theoremstyle{definition}

\newtheorem{remark}[theorem]{Remark}
\newtheorem{example}[theorem]{Example}

\renewcommand{\Bbb}{\mathbb}

\numberwithin{equation}{section}

\usepackage[bookmarksnumbered=true, colorlinks=true]{hyperref}
\AtBeginDocument{
    \hypersetup{
        citecolor=blue,
        linkcolor=blue,
        urlcolor=blue
    }
}

\begin{document}

\title[The algebraic Montgomery--Yang problem]{The algebraic Montgomery--Yang problem}

\author{Woohyeok Jo}
\address{Department of Mathematical Sciences, Seoul National University, Seoul 08826,  Republic of Korea}
\email{koko1681@snu.ac.kr}

\author{Jongil Park}
\address{Department of Mathematical Sciences and Research Institute of Mathematics, Seoul National University, Seoul 08826, Republic of Korea}
\email{jipark@snu.ac.kr}

\author{Kyungbae Park}
\address{Department of Mathematics, Kangwon National University, Gangwon 24341,  Republic of Korea}
\email{kyungbaepark@kangwon.ac.kr}

\thanks{}

\date{July 7, 2026}

\begin{abstract}
We completely resolve the algebraic Montgomery--Yang problem, a conjecture of Koll\'ar stating that every rational homology projective plane with quotient singularities and a simply-connected smooth locus has at most three singular points. The crux of our proof is a new lattice-theoretic constraint, obtained by combining Donaldson's diagonalization theorem with the distinguished $\operatorname{spin}^c$ structure on the smooth locus whose determinant line bundle is the canonical bundle. Together with the orbifold Bogomolov--Miyaoka--Yau inequality, this constraint rules out all remaining cases in the problem and completes the proof.
\end{abstract}

\maketitle

\section{Introduction}
A normal projective complex surface whose Betti numbers agree with those of the complex projective plane $\mathbb{CP}^2$ is called a \emph{rational homology projective plane}, or simply a \emph{$\mathbb{Q}$-homology $\mathbb{CP}^2$}. Such surfaces have the smallest possible Betti numbers among algebraic surfaces and thus occupy a prominent position in the theory of algebraic surfaces; see \cite{Keum-2018} for a survey. At the same time, as observed by Koll\'ar \cite{Kollar-2008}, these surfaces are deeply connected to questions in low-dimensional topology. For instance, they arise in the study of small fillings of spherical $3$-manifolds and in problems concerning embeddings of rational homology balls into $\mathbb{CP}^2$. For related recent works, see \cite{Jo-Park-Park-2024-lens, Golla-Owens}.

A non-singular $\mathbb{Q}$-homology $\mathbb{CP}^2$ that is not biholomorphic to $\mathbb{CP}^2$ is called a \emph{fake projective plane}. These surfaces have been completely classified \cite{Prasad-Yeung-2007, Prasad-Yeung-2010, Cartwright-Steger-2010}. In the singular setting, one usually restricts attention to surfaces with only quotient singularities. One of the fundamental questions concerning such surfaces is the classification problem in terms of singularity type, namely, determining which collections of quotient singularities can occur on a $\mathbb{Q}$-homology $\mathbb{CP}^2$ \cite{Brenton-1977, Brenton-Drucker-Prins-1981, Brenton-Drucker-Prins-1981-2, Bindschadler-Brenton-1984, Furushima-1986, Miyanishi-Zhang-1988, Zhang-1989, Kojima-1999, Ye-2002, Kojima-2003, Alexeev-Nikulin-2006, Hwang-Keum-Ohashi-2015, Keum-2015, Schutt-2023, Lacini-2024, Jo-Park-Park-2024-02, Belousov-Hwang-2025, Palka-Pelka-2025}. Closely related to this is the question of determining the maximal number of quotient singularities that a $\mathbb{Q}$-homology $\mathbb{CP}^2$ can support \cite{Miyaoka-1984, KeM-1999, Belousov-2008, Belousov-2009, Hwang-Keum-2011-2}. At the heart of both questions lies the so-called algebraic Montgomery--Yang problem, introduced by Koll\'ar \cite{Kollar-2008} as an algebraic analog of the Montgomery--Yang problem for pseudo-free smooth circle actions on the $5$-sphere. The latter problem, originally asked by Montgomery and Yang \cite{Montgomery-Yang-1972} and later formulated by Fintushel and Stern \cite{Fintushel-Stern-1987}, remains widely open; see also \cite[Problem 5.7]{K3-2026}. Hence the algebraic Montgomery--Yang problem belongs to a broader circle of questions at the interface of algebraic geometry and low-dimensional topology. In particular, it is closely related to Koll\'ar's proposed smooth Bogomolov--Miyaoka--Yau (BMY) inequality for differentiable orbifolds \cite{Kollar-2008}, which connects the Montgomery--Yang problem with $H$-cobordism questions for Seifert fibered $3$-manifolds and with the study of algebraic surfaces with quotient singularities. A concrete instance of this connection is the three fibers conjecture, first indicated by Fintushel and Stern and explicitly stated by Lawson \cite{Lawson-1988}; see also \cite[Conjecture 20]{Kollar-2008}. It asserts that a Seifert fibered homology $3$-sphere with more than three exceptional fibers cannot bound a homology $4$-ball.

\begin{conjecture}[Algebraic Montgomery--Yang problem, {\cite[Conjecture 30]{Kollar-2008}}]\label{conj:aMY} Let $S$ be a rational homology projective plane with quotient singularities. If its smooth locus $S^0:= S\setminus\textup{Sing}(S)$ is simply-connected, then $S$ has at most three singular points. 
\end{conjecture}

The main goal of this paper is to prove this conjecture, thereby settling the algebraic Montgomery--Yang problem. It is already known, essentially by the weak version $e_{\textup{orb}}\geq 0$ of the orbifold BMY inequality, that a $\mathbb{Q}$-homology $\mathbb{CP}^2$ with quotient singularities and a simply-connected smooth locus can have at most four singular points. Moreover, Conjecture \ref{conj:aMY} has been verified when $S$ has at least one non-cyclic quotient singularity, is non-rational, or has an anti-ample canonical divisor \cite{Hwang-Keum-2011-1, Hwang-Keum-2013, Hwang-Keum-2014}. Consequently, Conjecture \ref{conj:aMY} reduces to the case where $S$ is rational, has an ample canonical divisor, and admits only cyclic quotient singularities. In our previous work \cite{Jo-Park-Park-2025-01}, we used several gauge-theoretic and smooth $4$-manifold invariants, in a direction quite different from earlier techniques, to narrow down the remaining cases. In this paper, we introduce a new lattice-theoretic constraint, obtained by combining Donaldson's diagonalization theorem with the distinguished $\operatorname{spin}^c$ structure on the smooth locus whose first Chern class is the canonical class. Together with the preceding reductions, this constraint completes the proof of the algebraic Montgomery--Yang problem.

\begin{theorem}[Algebraic Montgomery--Yang problem]\label{main:aMY}
    Let $S$ be a rational homology projective plane with quotient singularities. If its smooth locus $S^0$ is simply-connected, then $S$ has at most three singular points. 
\end{theorem}

Note that the bound in Theorem \ref{main:aMY} is sharp: there are many rational homology projective planes with a simply-connected smooth locus and three quotient singularities, such as weighted projective planes; see Example \ref{ex:weighted_projective_planes}. Thus, the theorem gives an optimal upper bound on the number of singularities of such surfaces.

We first recall the reduction from \cite{Jo-Park-Park-2025-01} and indicate where the new ingredient of the present paper enters. Let $S$ be a $\mathbb{Q}$-homology $\mathbb{CP}^2$ as described in Conjecture \ref{conj:aMY}, and assume that $S$ has precisely four cyclic quotient singularities of types\footnote{The convention for the notation of cyclic quotient singularities in \cite{Jo-Park-Park-2025-01} differs from the one adopted in this paper. Specifically, a cyclic quotient singularity of type $\tfrac{1}{p}(1, q)$ in our notation is referred to as type $(p, p-q)$ in \cite{Jo-Park-Park-2025-01}.} \begin{equation}\label{eq:four_cyclic_singularities}
    \tfrac{1}{p_1}(1,q_1), \quad 
    \tfrac{1}{p_2}(1,q_2), \quad 
    \tfrac{1}{p_3}(1,q_3), \quad 
    \tfrac{1}{p_4}(1,q_4), \qquad 
    p_1 <  p_2 < p_3 < p_4.
\end{equation}
The weak version of the orbifold BMY inequality implies that the orbifold Euler characteristic $e_{\textup{orb}}(S)$ must be nonnegative, and the condition $H_1(S^0;\mathbb{Z})=0$ implies that the orders $p_1,p_2,p_3,p_4$ are pairwise relatively prime. Combining these conditions leaves exactly the following three possibilities: \begin{equation}
\label{eq:three_cases}
(p_1,p_2,p_3,p_4)=\begin{cases}
    (2,3,5,n), & \gcd(n,30)=1,  \\
    (2,3,7,n), & n\in \{11, 13, 17, 19, 23, 25, 29, 31, 37, 41\}, \\
    (2,3,11,13). 
\end{cases}
\end{equation}
Observe that the $(2,3,5,n)$ case is infinite, while the $(2,3,7,n)$ and $(2,3,11,13)$ cases are finite.

Each singular point of $S$ has a neighborhood homeomorphic to the cone over its link, which is a lens space. Excising these cone neighborhoods from $S$ yields a compact, smooth, oriented $4$-manifold $X_S$ with $b_2 = b_2^+ = 1$, whose boundary is the disjoint union of four lens spaces. In \cite{Jo-Park-Park-2025-01}, we applied various tools from the theory of topological and smooth $4$-manifolds—such as the linking form, Donaldson's diagonalization theorem, and Heegaard Floer $d$-invariants—to prove the non-existence of such a $4$-manifold $X_S$. By doing so, we eliminated the $(2,3,7,n)$ and $(2,3,11,13)$ cases \cite[Theorem 1.5]{Jo-Park-Park-2025-01}, as well as the $(2,3,5,n)$ case for $n<2599$ \cite[Theorem 1.6]{Jo-Park-Park-2025-01}.

In the remaining $(2,3,5,n)$ case, the subcase in which the singularity of order $3$ is of type $\tfrac{1}{3}(1,2)$ has already been ruled out by \cite[Lemma 5.3]{Hwang-Keum-2013}. Consequently, the cyclic quotient singularities of such a surface must be of the types
\begin{equation}\label{eq:remaining_cases}
    \tfrac{1}{2}(1,1), \quad \tfrac{1}{3}(1,1), \quad \tfrac{1}{5}(1,q_3), \quad \tfrac{1}{n}(1,q_4),
\end{equation}
for some positive integer $n$ with $\gcd(n, 30)=1$ and some integers $q_3$ and $q_4$ satisfying $0<q_3<5$, $0<q_4<n$, $\gcd(q_3,5)=1$, and $\gcd(q_4,n)=1$. Thus, it remains only to show that no such surface with cyclic quotient singularities of the types listed in \eqref{eq:remaining_cases} can exist; this is precisely the objective of Theorem \ref{thm:aMY} below. 

While the tools in \cite{Jo-Park-Park-2025-01} were quite powerful, they were not sufficient to rule out all remaining cases. In fact, through computer-aided calculations, we applied these obstructions to all singularity types \eqref{eq:remaining_cases} with $n < 50{,}000$, leaving exactly $16$ cases unresolved \cite[Table 3]{Jo-Park-Park-2025-01}; note that $8$ of these were recently ruled out in \cite[Corollary 8.2]{Liu-Liu-2026}. Moreover, we found an infinite family that is not obstructed by any of the conditions considered there \cite[Proposition 4.1]{Jo-Park-Park-2025-01}. The limitation of the previous approach was that it treated $X_S$ merely as a smooth $4$-manifold. In the present paper, we make essential use of the fact that $X_S$ arises from a complex surface: the canonical bundle on the smooth locus determines a distinguished $\operatorname{spin}^c$ structure, and the key new idea is to refine the Donaldson obstruction by imposing compatibility with this $\operatorname{spin}^c$ structure. This leads to the following lattice-theoretic constraint, which is the main technical ingredient in the proof of Theorem \ref{thm:aMY}.

\begin{theorem}\label{thm:main}
    Let $S$ be a rational homology projective plane with $H_1(S^0;\mathbb{Z})=0$ and with $n$ cyclic quotient singularities of types 
\[
    \tfrac{1}{p_1}(1, q_1),\dots,\tfrac{1}{p_n}(1, q_n).
\]
Assume that, for each $i=1,\dots,n$, either $p_i$ is odd or $\sigma(q_i, p_i, +1)\neq \sigma (q_i, p_i, -1)$. Then there exists a codimension-one lattice embedding \begin{equation}\label{eq:codim_one_embedding}
    \iota\colon \bigoplus_{i=1}^n Q_{X(p_i, p_i-q_i)} \longrightarrow -\mathbb{Z}^N=\langle -1\rangle^N,
\end{equation}
where $N=\sum_{i=1}^n b_2(X(p_i,p_i-q_i))+1$, satisfying the following conditions: \begin{enumerate}[label=\textup{(\alph*)}]
    \item \label{item(a)_main} There exists a vector $\mathbf{K}\in -\mathbb{Z}^N$ of the form \[
    \mathbf{K}=\epsilon_1e_1 +\cdots+ \epsilon_{i_0-1}e_{i_0-1}+3\epsilon_{i_0}e_{i_0}+\epsilon_{i_0+1}e_{i_0+1}+\cdots +\epsilon_Ne_N \quad (1\leq i_0\leq N, \, \epsilon_i\in \{\pm 1\})
    \]
    such that for any standard basis vector $v\in \bigoplus_{i=1}^n Q_{X(p_i,p_i-q_i)}$, \[
    \mathbf{K}\cdot \iota(v)=\begin{cases}
        -v^2 & \text{if $v$ is right-end}, \\
        -v^2-2 & \text{otherwise}.
    \end{cases}
    \]
    \item \label{item(b)_main} The orthogonal complement of the image of the embedding $\iota$ is generated by a vector in $-\mathbb{Z}^N$ of square $-p_1\cdots p_n$.   
\end{enumerate} 
\end{theorem}

Here we briefly explain the notation used in the theorem above. Detailed discussions will be given in subsequent sections. \begin{itemize}
    \item  For relatively prime integers $p>q>0$, the function $\sigma$, introduced in \cite{Fukumoto-Furuta-Ue-2001}, is defined by \[
  \sigma(q,p,\epsilon)=\frac{1}{p}\sum_{k=1}^{p-1} \left(\cot \frac{\pi k}{p}\cot \frac{\pi k q}{p} +2\epsilon^k \csc \frac{\pi k }{p}\csc \frac{\pi k q}{p} \right) \quad (\epsilon \in \{\pm 1\}).
\]
For even $p$, we have $\sigma (q,p,+1)=\sigma (q,p,-1)$ precisely when the two spin structures of the lens space $L(p,q)$ have the same Heegaard Floer $d$-invariant (Remark \ref{rmk:sigma_and_d_invariants}). For $(p,q)=(2,1)$, we have $\sigma(1,2,+1)=1\neq -1=\sigma(1,2,-1)$, so Theorem \ref{thm:main} applies to the remaining cases \eqref{eq:remaining_cases} of the algebraic Montgomery--Yang problem.
\item Suppose that the Hirzebruch-Jung continued fraction of $p/q$ is given by \[ \frac{p}{q} = [a_1, \dots, a_s] := a_1 - \frac{1}{a_2 - \dfrac{1}{\cdots - \dfrac{1}{a_s}}} \quad (a_i \geq 2). \]
Then $X(p,q)$ denotes the negative definite plumbed $4$-manifold with $\partial X(p,q)=L(p,q)$, constructed according to the linear graph shown in the following figure. 
\begin{figure}[h]
\centering
\begin{tikzpicture}[scale=1.1]
\draw (-2,0) node[circle, fill, inner sep=1.2pt, black]{};
\draw (-1,0) node[circle, fill, inner sep=1.2pt, black]{};
\draw (1,0) node[circle, fill, inner sep=1.2pt, black]{};
\draw (2,0) node[circle, fill, inner sep=1.2pt, black]{};
\draw (-2,0) node[below]{$-a_1$};
\draw (-1,0) node[below]{$-a_2$};
\draw (1,0) node[below]{$-a_{s-1}$};
\draw (2,0) node[below]{$-a_s$};
\draw (0,0) node{$\cdots$};
\draw (-2,0)--(-1,0) (-1,0)--(-0.5,0) (0.5,0)--(1,0)  (1,0)--(2,0) ;
\end{tikzpicture}
\end{figure}

\noindent Its intersection lattice is denoted by $Q_{X(p,q)}$. The homology classes $v_1,\dots,v_s$ corresponding to the vertices with weights $-a_1,\dots, -a_s$ form a basis for $H_2(X(p,q);\mathbb{Z})$, called the \textit{standard basis} for $Q_{X(p,q)}$. In particular, we refer to $v_s$ as the \textit{right-end basis vector}.
\item For a nonzero integer $m\in \mathbb{Z}$, $\langle m\rangle$ denotes the lattice structure on $\mathbb{Z}$ where a generator $v$ satisfies $v^2=m$. For a positive integer $n$, $-\mathbb{Z}^n:=\langle -1\rangle^n$ denotes the standard diagonal negative definite lattice of rank $n$. We denote the standard basis for $-\mathbb{Z}^n$ by $\{e_1,\dots ,e_n\}$. These vectors satisfy $e_i\cdot e_j=-\delta_{ij}$, where $\delta_{ij}$ is the Kronecker delta. 
\end{itemize} 

We now outline the proof of Theorem \ref{thm:main}. As explained above, we remove cone neighborhoods of the singular points of $S$ to obtain a compact, smooth, oriented $4$-manifold $X_S$ with $b_2(X_S)=b_2^+(X_S)=1$, whose oriented boundary is given by $\partial X_S=\coprod_{i=1}^n -L(p_i,q_i)$, where $-L(p_i,q_i)$ denotes the orientation reversal of $L(p_i,q_i)$. The canonical divisor $K_S$ of $S$ restricts to an integral cohomology class $K_{X_S}\in H^2(X_S;\mathbb{Z})$, and there is a unique $\operatorname{spin}^c$ structure $\mathfrak{s}_{K_{X_S}}$ on $X_S$ whose first Chern class is $K_{X_S}$. 

Next, consider the disjoint union $X':=\coprod_{i=1}^n X(p_i,p_i-q_i)$ of plumbed $4$-manifolds. On each $X(p_i,p_i-q_i)$, we explicitly define a cohomology class $K_{p_i,p_i-q_i}\in H^2(X(p_i,p_i-q_i);\mathbb{Z})$ by specifying its evaluations on the standard basis vectors of $Q_{X(p_i,p_i-q_i)}$. The $\operatorname{spin}^c$ structures on $X(p_i,p_i-q_i)$ with first Chern class $K_{p_i,p_i-q_i}$ together define a $\operatorname{spin}^c$ structure $\mathfrak{s}'$ on $X'$.

Now, we take a specific orientation-reversing diffeomorphism $f$ between the boundaries of $-X_S$ (the orientation reversal of $X_S$) and $X'$: 
\[
    f\colon\partial(-X_S)=\coprod_{i=1}^n L(p_i,q_i) \longrightarrow \coprod_{i=1}^n L(p_i,p_i-q_i) =\partial X'.
\]
Through a detailed analysis of $\operatorname{spin}^c$ structures on lens spaces, we show that these $\operatorname{spin}^c$ structures match along the boundary, i.e., $f^*(\mathfrak{s}'|_{\partial X'})=\mathfrak{s}_{K_{X_S}}|_{\partial(-X_S)}$. This implies the existence of a global $\operatorname{spin}^c$ structure $\mathfrak{s}$, extending $\mathfrak{s}'$ and $\mathfrak{s}_{K_{X_S}}$, on the closed $4$-manifold $W_S:=X'\cup_f (-X_S)$ obtained by gluing along $f$. The closed $4$-manifold $W_S$ has a negative definite intersection form whose rank $b_2(W_S)$ is exactly the number $N$ given in the statement of the theorem. Donaldson's diagonalization theorem \cite{Donaldson-1987} thus implies that the intersection lattice $Q_{W_S}$ is isomorphic to the standard lattice $-\mathbb{Z}^N$. Under this isomorphism, we show that the class $c_1(\mathfrak{s})$, viewed as an element of $H_2(W_S;\mathbb{Z})$ via Poincar\'e duality, corresponds to a vector $\mathbf{K}$ satisfying condition \ref{item(a)_main}. Condition \ref{item(b)_main} follows from elementary considerations.

A notable feature of Theorem \ref{thm:main} is that it converts the existence problem for $S$ into a purely lattice-theoretic one. Applying this constraint to the remaining $(2,3,5,n)$ case listed in \eqref{eq:remaining_cases}, we prove that no compatible lattice embedding can exist. The only additional algebro-geometric input used in this final step is the orbifold BMY inequality (Theorem \ref{thm:oBMY}), which serves to simplify the combinatorial analysis. This gives the following final non-existence theorem, which, together with the preceding reductions, completes the proof of Theorem \ref{main:aMY}.

\begin{theorem}\label{thm:aMY}
    There does not exist a rational homology projective plane $S$ with $H_1(S^0;\mathbb{Z})=0$ whose canonical divisor is ample and that has exactly four cyclic quotient singularities of types 
    \[
        \tfrac{1}{2}(1,1), \quad \tfrac{1}{3}(1,1), \quad \tfrac{1}{5}(1,q_3), \quad \tfrac{1}{n}(1,q_4),
    \]
    where $n$ is a positive integer relatively prime to $30$.
\end{theorem}

Although Conjecture \ref{conj:aMY} and Theorem \ref{main:aMY} are stated under the assumption that the smooth locus is simply-connected, Theorem \ref{thm:aMY}, together with previously known results, shows that in the case of cyclic quotient singularities, the same conclusion holds under the weaker assumption $H_1(S^0;\mathbb{Z})=0$.

\begin{theorem}\label{cor:aMY}
    Let $S$ be a rational homology projective plane with only cyclic quotient singularities. If $H_1(S^0;\mathbb{Z})=0$, then $S$ has at most three singularities. 
\end{theorem}

Note that the restriction to cyclic quotient singularities in Theorem \ref{cor:aMY} is necessary, as the statement fails if non-cyclic singularities are allowed \cite[Theorem 3]{Hwang-Keum-2011-1}.

\begin{remark}
Although our main focus is to apply Theorem \ref{thm:main} to rule out the remaining infinite $(2,3,5,n)$ case in \eqref{eq:three_cases}, the same theorem also immediately excludes the other two cases previously eliminated in \cite{Jo-Park-Park-2025-01}. More precisely, there are exactly $1008$ candidate combinations of singularities in the $(2,3,7,n)$ case and $84$ candidate combinations in the $(2,3,11,13)$ case (see \cite[Section 3.1]{Jo-Park-Park-2025-01}). Among these, the orbifold BMY inequality (Theorem \ref{thm:oBMY}) eliminates all candidates in the $(2,3,11,13)$ case and leaves only the following two combinations in the $(2,3,7,n)$ case:
\[
    \left\{\tfrac{1}{2}(1,1),\, \tfrac{1}{3}(1,1),\, \tfrac{1}{7}(1,5),\, \tfrac{1}{19}(1,16) \right\} \quad \text{and} \quad \left\{ \tfrac{1}{2}(1,1),\, \tfrac{1}{3}(1,2),\, \tfrac{1}{7}(1,1),\, \tfrac{1}{19}(1,17) \right\}.
\]
These two surviving candidates are immediately ruled out by Theorem \ref{thm:main}, since the corresponding codimension-one lattice embedding \eqref{eq:codim_one_embedding} into the standard negative definite lattice cannot exist.
\end{remark}

\begin{remark}
As mentioned above, the $(2,3,5,n)$ case with a singularity of type $\tfrac{1}{3}(1,2)$ has already been ruled out by a different argument. We expect, however, that it can also be excluded by the techniques developed in this paper. Consequently, in the setting of four cyclic quotient singularities, the non-existence of the surfaces considered in Theorem \ref{main:aMY} would be established uniformly, using only the orbifold BMY inequality and the lattice-theoretic constraint of Theorem \ref{thm:main}, without invoking the more involved tools used in \cite{Jo-Park-Park-2025-01}, such as the linking form condition and Heegaard Floer $d$-invariants.
\end{remark}

With the algebraic Montgomery--Yang problem now resolved, the next natural task is to classify the possible singularity types of rational homology projective planes with quotient singularities and a simply-connected smooth locus. For recent progress in this direction, we refer the reader to \cite{Jo-Park-Park-2024-02}. Since Theorem \ref{thm:main} provides an effective tool for obstructing the existence of surfaces with prescribed singularity types, we expect it to play an important role in advancing this classification problem.

\subsection*{Organization of the paper}
In Section \ref{sec:spinc_lattice}, we review basic facts about $\operatorname{spin}^c$ structures, establish our conventions regarding intersection lattices, and explain how the gluing of $\operatorname{spin}^c$ structures can be interpreted in terms of lattice embeddings. In Section \ref{sec:lens_spinc}, we provide a detailed analysis of $\operatorname{spin}^c$ structures on lens spaces and their behavior under certain orientation-reversing diffeomorphisms. The results established in this section will play a crucial role in the proof of Theorem \ref{thm:main}. Section \ref{sec:orbifold_BMY} summarizes the topological properties of rational homology projective planes and derives key properties that will be utilized in the proof of Theorem \ref{thm:main}. We also recall the orbifold BMY inequality (Theorem \ref{thm:oBMY}), which will be applied in the proof of Theorem \ref{thm:aMY}. In Section \ref{sec:main_thm_proof}, we prove the main lattice-theoretic constraint, Theorem \ref{thm:main}, using the results of Sections \ref{sec:spinc_lattice}, \ref{sec:lens_spinc}, and \ref{sec:orbifold_BMY}. Finally, in Section \ref{sec:main_result_proof}, we apply this constraint to prove Theorem \ref{thm:aMY}, thereby completing the resolution of the algebraic Montgomery--Yang problem. 

\subsection*{Acknowledgments} The authors are deeply grateful to Masaaki Ue for kindly explaining the details of the proof of Proposition \ref{prop:canonical_spin^c_structure}. They also express their gratitude to J\'anos Koll\'ar for his invaluable comments and suggestions, especially on the Introduction. Jongil Park was supported by the National Research Foundation of Korea (NRF) grants funded by the Korean government (No.2020R1A5A1016126 and RS-2024-00392067). He is also affiliated with the Research Institute of Mathematics at SNU. Kyungbae Park was supported by NRF grants funded by the Korean government (RS-2025-24523511 and RS-2025-25415913).

\section{Spin$^c$ structures and lattice embeddings}\label{sec:spinc_lattice}

In this section, we recall basic facts about $\operatorname{spin}^c$ structures on $3$- and $4$-manifolds. We also explain how the gluing of $\operatorname{spin}^c$ structures on two compact $4$-manifolds with boundary is reflected in the corresponding intersection forms.

\subsection{Spin$^c$ structures on $3$- and $4$-manifolds}\label{subsec:spin^c_structures} 
Let $X$ be a smooth, compact, oriented $4$-manifold, possibly with boundary. The set $\operatorname{Spin}^c(X)$ of isomorphism classes of $\operatorname{spin}^c$ structures on $X$ is nonempty, and it admits a free, transitive action of $H^2(X;\mathbb{Z})$ on $\operatorname{Spin}^c(X)$ which we denote by \[
(\mathfrak{s},\ell)\longmapsto \mathfrak{s}+\ell
\]
for $\mathfrak{s}\in \operatorname{Spin}^c(X)$ and $\ell \in H^2(X;\mathbb{Z})$. The $H^2(X;\mathbb{Z})$-action endows the set $\operatorname{Spin}^c(X)$ with the structure of an $H^2(X; \mathbb{Z})$-torsor. Consequently, the ``difference'' of two $\operatorname{spin}^c$ structures on $X$ is a well-defined element in $H^2(X; \mathbb{Z})$, even though there is no canonical identification between $\operatorname{Spin}^c(X)$ and $H^2(X; \mathbb{Z})$. Analogous statements hold for any smooth, oriented $3$-manifold $Y$.

There is a natural involution on $\operatorname{Spin}^c(X)$, called \textit{conjugation} and denoted by $\mathfrak{s}\mapsto \overline{\mathfrak{s}}$. In terms of the $H^2(X;\mathbb Z)$-torsor structure, it is given by $\overline{\mathfrak{s}}=\mathfrak{s}-c_1(\mathfrak{s})$, where $c_1(\mathfrak{s})$ is the first Chern class of the determinant line bundle $\det(\mathfrak{s})$ of $\mathfrak{s}$. In particular, we have $\det(\overline{\mathfrak{s}})\cong \det(\mathfrak{s})^{-1}$. The same construction applies to smooth, oriented $3$-manifolds.

A cohomology class $\alpha\in H^2(X;\mathbb{Z})$ is called a \emph{characteristic covector} if it satisfies \[
\langle \alpha ,v\rangle \equiv Q_X(v,v)\pmod{2}
\]
for every $v\in H_2(X;\mathbb{Z})$, where $Q_X\colon H_2(X;\mathbb{Z})\times H_2(X;\mathbb{Z})\to \mathbb{Z}$ denotes the intersection form of $X$. Any integral lift of the second Stiefel-Whitney class $w_2(TX)\in H^2(X;\mathbb{Z}_2)$ is a characteristic covector, and the converse holds provided that $H_1(X;\mathbb{Z})=0$. We denote the set of all characteristic covectors by $\operatorname{Char}(X) \subset H^2(X; \mathbb{Z})$. For a $\operatorname{spin}^c$ structure $\mathfrak{s}\in \operatorname{Spin}^c(X)$, the class $c_1(\mathfrak{s})$ is a characteristic covector, and the assignment $\mathfrak{s}\mapsto c_1(\mathfrak{s})$ defines a bijection $\operatorname{Spin}^c(X)\to \operatorname{Char}(X)$ if $H_1(X;\mathbb{Z})=0$. For any $\mathfrak{s}\in \operatorname{Spin}^c(X)$ and $\ell\in H^2(X;\mathbb{Z})$, we have \begin{equation}\label{eq:spin^c-char_assignment}
    c_1(\mathfrak{s}+\ell) = c_1(\mathfrak{s})+2\ell.
\end{equation}

If $-X$ denotes the orientation reversal of $X$, then there is a canonical identification $\operatorname{Spin}^c(X)\cong \operatorname{Spin}^c(-X)$ given by swapping the $\pm$-spinor bundles. In the $3$-dimensional case, there is a canonical identification $\operatorname{Spin}^c(Y)\cong \operatorname{Spin}^c(-Y)$ given by changing the sign of Clifford multiplication. These identifications are compatible with the actions of $H^2(X; \mathbb{Z})$ and $H^2(Y; \mathbb{Z})$, respectively, and we will implicitly use these identifications throughout this paper.

Now suppose $X$ has a nonempty boundary $\partial X=Y$, where $Y$ is given the boundary orientation. The restriction map $\operatorname{Spin}^c(X)\to \operatorname{Spin}^c(Y)$ is surjective if and only if the map $H^2(X;\mathbb{Z})\to H^2(Y;\mathbb{Z})$ is surjective. This is the case, for example, if $H_1(X;\mathbb{Z})=0$. Note that this restriction is also compatible with conjugation and with the orientation reversal identifications described above.

Furthermore, $\operatorname{spin}^c$ structures can be glued along boundaries in the following sense. Let $X_1, X_2$ be compact, oriented, smooth $4$-manifolds and suppose $f\colon\partial X_2\to \partial X_1$ is an orientation-reversing diffeomorphism. If $\mathfrak{s}_i\in \operatorname{Spin}^c(X_i)$ for $i=1,2$ are $\operatorname{spin}^c$ structures satisfying $f^*(\mathfrak{s}_1|_{\partial X_1}) = \mathfrak{s}_2|_{\partial X_2}$, then there exists a $\operatorname{spin}^c$ structure $\mathfrak{s}\in \operatorname{Spin}^c(W)$ on the closed $4$-manifold $W:=X_1\cup_f X_2$ such that $\mathfrak{s}|_{X_i}=\mathfrak{s}_i$ for $i=1,2$. Strictly speaking, $f^*(\mathfrak{s}_1|_{\partial X_1})$ is a $\operatorname{spin}^c$ structure on $-\partial X_2$, and we are implicitly using the canonical identification $\operatorname{Spin}^c(-\partial X_2)\cong \operatorname{Spin}^c(\partial X_2)$ discussed above.

\subsection{Intersection lattices and spin$^c$-gluing}\label{subsec:spin^c_and_Donaldson}
In this subsection, we explore the relationship between the gluing of $\operatorname{spin}^c$ structures and intersection forms. 

Let us first review the basic notions of lattices. A \emph{lattice} is a pair $(L,Q)$ consisting of a free abelian group $L$ of finite rank and a symmetric bilinear form $Q\colon L\times L\to \mathbb{Q}$. For $v,w\in L$, we will simply write $v\cdot w=Q(v,w)$ and $v^2=Q(v,v)$. We will only consider the case where $Q$ is \emph{nondegenerate}, i.e., for any nonzero $v\in L$, there exists some $w\in L$ such that $v\cdot w\neq 0$. An \emph{embedding} (resp.\ \emph{isomorphism}) $\iota:(L_1,Q_1)\to (L_2,Q_2)$ of lattices is a monomorphism (resp.\ isomorphism) $\iota\colon L_1\to L_2$ of groups that preserves the forms, i.e., $\iota(x)\cdot \iota(y)=x\cdot y$ for all $x,y\in L_1$.

Assume that $(L,Q)$ is an \emph{integral} lattice, i.e., the image of $Q$ lies in $\mathbb{Z}$. Letting $L_{\mathbb{Q}}:=L\otimes \mathbb{Q}$, the form $Q$ naturally extends to a $\mathbb{Q}$-bilinear symmetric form $L_{\mathbb{Q}}\times L_{\mathbb{Q}}\to \mathbb{Q}$, still denoted by $Q$. The subset \[
L^*:=\left\{v\in L_{\mathbb{Q}}:Q(v,w)\in \mathbb{Z} \text{ for all $w\in L$} \right\}
\]
together with the restriction $Q|_{L^*\times L^*}$ is called the \emph{dual lattice} of $L$. We have $L\subset L^*$, with equality if and only if $(L,Q)$ is \emph{unimodular}, i.e., $\det Q=\pm 1$. The dual lattice $L^*$ can be canonically identified with the group $\operatorname{Hom}(L,\mathbb{Z})$; an element $\alpha\in \operatorname{Hom}(L,\mathbb{Z})$ corresponds to the element $v_\alpha\in L^*$ which satisfies $Q(v_\alpha,v)=\alpha(v)$ for all $v\in L$.

For a positive integer $n$, let $\{e_1, \dots, e_n\}$ be the standard basis for $\mathbb{Z}^n$. We denote by $-\mathbb{Z}^n$ the standard diagonal negative definite lattice structure on $\mathbb{Z}^n$, where the pairing is given by $e_i\cdot e_j = -\delta_{ij}$, with $\delta_{ij}$ being the Kronecker delta. We denote by $-\mathbb{Q}^n$ the $\mathbb{Q}$-linear extension of the lattice $-\mathbb{Z}^n$. Also, for any nonzero integer $m\in \mathbb{Z}$, we denote by $\langle m\rangle$ the lattice structure on $\mathbb{Z}$ where a generator $v$ satisfies $v^2=m$. The standard lattice $-\mathbb{Z}^n$ is isomorphic to the direct sum \[
\langle -1\rangle^n =\overbrace{\langle -1\rangle\oplus \cdots \oplus \langle-1\rangle}^n. \]

For a compact, oriented $4$-manifold $X$ whose second homology group $H_2(X;\mathbb{Z})$ is free, we will denote the lattice $(H_2(X;\mathbb{Z}), Q_X)$ simply by $Q_X$, where $Q_X$ is the intersection form of $X$. Donaldson's diagonalization theorem \cite{Donaldson-1987} states that if $X$ is smooth and closed, and $Q_X$ is negative definite, then the lattice $Q_X$ is isomorphic to $-\mathbb{Z}^{b_2(X)}$. 

Suppose $X$ is a compact, oriented $4$-manifold with $H_1(X;\mathbb{Z})=0$ whose boundary $\partial X$ satisfies $H^1(\partial X;\mathbb{Z})=0$ (i.e., each component of $\partial X$ is a rational homology $3$-sphere). The universal coefficient theorem gives an isomorphism $H^2(X;\mathbb{Z})\cong \operatorname{Hom}(H_2(X;\mathbb{Z}),\mathbb{Z})$, and we will regard \begin{equation}\label{eq:second_cohomology_as_dual_lattice}
    H^2(X;\mathbb{Z})\subset H_2(X;\mathbb{Z})\otimes \mathbb{Q}
\end{equation}
as the dual lattice $Q_X^*$ of $Q_X$. If $\{v_1,\dots,v_n\}$ is a basis for $H_2(X;\mathbb{Z})$ with dual basis $\{v_1^*,\dots,v_n^*\}$ for $H^2(X;\mathbb{Z})$ (so that $\langle v_i^*,v_j\rangle = \delta_{ij}$), then the matrix $(v_i^* \cdot v_j^*)$ is the inverse of the matrix $(v_i\cdot v_j)$ representing the lattice $Q_X$. 

With these conventions established, we now translate the gluing of $\operatorname{spin}^c$ structures into a statement about lattice embeddings.

\begin{proposition}\label{prop:spin^c_gluing}
    For $i=1,2$, let $X_i$ be a compact, oriented, smooth $4$-manifold with negative definite intersection form $Q_{X_i}$ such that $H^1(\partial X_i;\mathbb{Z})=0$ and $H_1(X_i;\mathbb{Z})=0$, and let $\mathfrak{s}_i\in \operatorname{Spin}^c(X_i)$. If $f\colon\partial X_2\to \partial X_1$ is an orientation-reversing diffeomorphism such that $f^*(\mathfrak{s}_1|_{\partial X_1})=\mathfrak{s}_2|_{\partial X_2}$, then there is a full-rank lattice embedding \[
    \iota\colon Q_{X_1}\oplus Q_{X_2} \longrightarrow -\mathbb{Z}^N,
    \]
    where $N=b_2(X_1)+b_2(X_2)$, such that its $\mathbb{Q}$-linear extension \[
    \iota_{\mathbb{Q}}\colon (Q_{X_1}\otimes \mathbb{Q})\oplus (Q_{X_2}\otimes \mathbb{Q}) \longrightarrow -\mathbb{Q}^N
    \]  maps $c_1(\mathfrak{s}_1)+c_1(\mathfrak{s}_2) \in Q_{X_1}^*\oplus Q_{X_2}^*$ to $\sum_{i=1}^N a_i e_i\in -\mathbb{Z}^N$ for some odd integers $a_1,\dots,a_N$ satisfying \[
c_1(\mathfrak{s}_1)^2+c_1(\mathfrak{s}_2)^2=-\sum_{i=1}^N a_i^2.
    \]
    Moreover, we have $\iota(Q_{X_1})^\perp =\iota(Q_{X_2})$ and $\iota(Q_{X_2})^\perp=\iota(Q_{X_1})$ in $-\mathbb{Z}^N$. 
\end{proposition}
\begin{proof} Consider the smooth, closed, oriented $4$-manifold $W:=X_1\cup _f X_2$. Since $H^1(\partial X_1;\mathbb{Z})=0$, a Mayer-Vietoris sequence yields a full rank lattice embedding $Q_{X_1}\oplus Q_{X_2}\to Q_W$. It follows that its $\mathbb{Q}$-linear extension \begin{equation}\label{eq:Q-linear_extension}
    (Q_{X_1}\otimes \mathbb{Q})\oplus (Q_{X_2}\otimes \mathbb{Q}) \longrightarrow Q_W\otimes \mathbb{Q}
\end{equation}
is an isomorphism. 

Recall from Section \ref{subsec:spin^c_structures} that there is a $\operatorname{spin}^c$ structure $\mathfrak{s}\in \operatorname{Spin}^c(W)$ on $W$ such that $\mathfrak{s}|_{X_i}=\mathfrak{s}_i$ for $i=1,2$. In particular, we have $c_1(\mathfrak{s})|_{X_i}=c_1(\mathfrak{s}_i)\in H^2(X_i;\mathbb{Z})$ for $i=1,2$. This precisely means that the image of $c_1(\mathfrak{s}_1)+c_1(\mathfrak{s}_2)\in Q_{X_1}^*\oplus Q_{X_2}^*$ under the isomorphism \eqref{eq:Q-linear_extension} is $c_1(\mathfrak{s})$. Since $W$ is closed, we have $H^2(W;\mathbb{Z})=Q_W^*=Q_W\subset Q_W\otimes \mathbb{Q}$, so we may view $c_1(\mathfrak{s})\in Q_W$.  

On the other hand, since $W$ has a negative definite intersection form with $b_2(W)=b_2(X_1)+b_2(X_2)=N$, Donaldson's theorem implies that there is an isomorphism $Q_W\cong -\mathbb{Z}^N$ of lattices. Composing this with our embedding $Q_{X_1}\oplus Q_{X_2}\to Q_W$ gives a lattice embedding \[
    \iota\colon Q_{X_1}\oplus Q_{X_2} \longrightarrow Q_W\cong -\mathbb{Z}^N.
    \] 
By the argument of \cite[Lemma 2.4]{AMP-2025}, the sublattices $\iota(Q_{X_1})$ and $\iota(Q_{X_2})$ are mutually orthogonal complements, meaning $\iota(Q_{X_1})^\perp =\iota(Q_{X_2})$ and $\iota(Q_{X_2})^\perp=\iota(Q_{X_1})$.

Since $c_1(\mathfrak{s})\in \operatorname{Char}(W)$ is a characteristic covector, its image in $-\mathbb{Z}^N$ is of the form $\sum_{i=1}^N a_ie_i$ for some odd integers $a_1,\dots,a_N$ satisfying \[
-\sum_{i=1}^N a_i^2=\left(\sum_{i=1}^N a_ie_i\right)^2=c_1(\mathfrak{s})^2=\left(c_1(\mathfrak{s}_1)+c_1(\mathfrak{s}_2) \right)^2 = c_1(\mathfrak{s}_1)^2+c_1(\mathfrak{s}_2)^2. 
\] The image of $c_1(\mathfrak{s}_1)+c_1(\mathfrak{s}_2)$ under the $\mathbb{Q}$-linear extension  \[
\iota_{\mathbb{Q}}=\iota\otimes \mathbb{Q} \colon (Q_{X_1}\otimes \mathbb{Q})\oplus (Q_{X_2}\otimes \mathbb{Q}) \longrightarrow -\mathbb{Q}^N 
\]
is precisely $\sum_{i=1}^N a_ie_i$. This completes the proof.
\end{proof}

\begin{remark}
Let $W$ be a smooth, closed, oriented $4$-manifold such that $H_2(W;\mathbb{Z})$ is free abelian of rank $N$ (as in the proof of Proposition \ref{prop:spin^c_gluing}). Assume that the intersection form of $W$ is negative definite, so that $Q_W\cong -\mathbb{Z}^N$ by Donaldson's theorem. Let $e_1,\dots,e_N\in H_2(W;\mathbb{Z})$ be a standard basis diagonalizing $Q_W$, and let $e_1^*,\dots, e_N^*\in H^2(W;\mathbb{Z})$ denote the corresponding dual basis vectors. 

There are two natural ways to identify $H^2(W;\mathbb{Z})$ with $H_2(W;\mathbb{Z})$. The first is to use Poincar\'e duality, $\operatorname{PD}\colon H^2(W;\mathbb{Z})\cong H_2(W;\mathbb{Z})$. The second is to identify $H^2(W;\mathbb{Z})$ with $\operatorname{Hom}(H_2(W;\mathbb{Z}),\mathbb{Z})$ via the universal coefficient theorem, and then identify $\operatorname{Hom}(H_2(W;\mathbb{Z}),\mathbb{Z})$ with $H_2(W;\mathbb{Z})$ noting that $Q_W$ is unimodular. Under the second identification, a class $\alpha \in H^2(W;\mathbb{Z})$ corresponds to the unique element $v_\alpha \in H_2(W;\mathbb{Z})$ satisfying $Q_W( v_\alpha,v) = \langle \alpha,v \rangle$ for all $v \in H_2(W;\mathbb{Z})$. However, since $Q_W(e_i,\operatorname{PD}(e_j^*))=\langle e_j^*,e_i\rangle=\delta_{ij}$, the two identifications are the same: both identifications map the dual basis element $e_j^*$ to $-e_j \in H_2(W;\mathbb{Z})$.     
\end{remark}

\section{Spin$^c$ structures on lens spaces}\label{sec:lens_spinc} 
In this section, we provide a detailed analysis of $\operatorname{spin}^c$ structures on lens spaces. The results established in this section will serve as key ingredients in the proof of Theorem \ref{thm:main}.

For relatively prime integers $p>q>0$, we define the lens space $L(p,q)$ as the oriented $3$-manifold obtained as the quotient of $S^3 \subset \mathbb{C}^2$ by the $\mathbb{Z}_p$-action generated by 
 \[
\left(z_1,z_2\right)\longmapsto \left(\zeta z_1, \zeta^q z_2\right),
\]
where $\zeta=e^{2\pi i/p}$. Under this orientation convention, $L(p,q)$ is the result of $-p/q$-surgery on $S^3$ along the unknot, and is precisely the link of a cyclic quotient singularity of type $\tfrac{1}{p}(1,q)$.

Consider the complex line bundle \begin{equation}\label{eq:line_bundle_L_rho}
    \mathcal{L}_{p,q} :=\frac{S^3\times \mathbb{C}}{\mathbb{Z}_p}=\frac{S^3\times \mathbb{C}}{\left(z_1,z_2,v\right)\sim \left(\zeta z_1, \zeta^q z_2, \zeta v\right) }
\end{equation}
over $L(p,q)$. Its first Chern class $c_1(\mathcal{L}_{p,q})$ is a generator of $H^2(L(p,q);\mathbb{Z})\cong \mathbb{Z}_p$. In particular, every complex line bundle over $L(p,q)$ is isomorphic to $\mathcal{L}_{p,q}^{\otimes k}$ for some integer $k$.

Using the standard Heegaard diagram of $L(p,q)$, Ozsv\'ath and Szab\'o provided an identification $\operatorname{Spin}^c(L(p,q))\cong \mathbb{Z}_p$ \cite[Section 4.1]{Ozsvath-Szabo-2003} (note that their orientation convention for $L(p,q)$ is opposite to ours). We denote by $\mathfrak{s}_{p,q}(i) \in \operatorname{Spin}^c(L(p,q))$ the $\operatorname{spin}^c$ structure corresponding to $i\in \mathbb{Z}_p$ under this identification. Recall from Section \ref{subsec:spin^c_structures} that the difference of two $\operatorname{spin}^c$ structures on $L(p,q)$ is a well-defined element in $H^2(L(p,q);\mathbb{Z})$. The following proposition shows that the Ozsv\'ath-Szab\'o identification $\mathbb{Z}_p\to\operatorname{Spin}^c(L(p,q))$ given by $i\mapsto \mathfrak{s}_{p,q}(i)$ is an affine map.

\begin{proposition}[{\cite[Proposition 4.7]{Ue-2009}}]\label{prop:spin^c_affine} For any $i\in \mathbb{Z}_p$ and $k\in \mathbb{Z}$, the difference $\mathfrak{s}_{p,q}(i+k)-\mathfrak{s}_{p,q}(i)$ is represented by $c_1(\mathcal{L}_{p,q}^{\otimes k})=k\cdot c_1(\mathcal{L}_{p,q})$, i.e., \[
\mathfrak{s}_{p,q}(i+k)=\mathfrak{s}_{p,q}(i)+k\cdot c_1(\mathcal{L}_{p,q})\in \operatorname{Spin}^c(L(p,q)),
\]
where $i+k$ is evaluated modulo $p$.  
\end{proposition}

Next, view $D^4$ as the closed unit disk in $\mathbb{C}^2$, and consider the cone \[
C:=CL(p,q)=\frac{D^4}{\mathbb{Z}_p}=\frac{D^4}{\left(z_1,z_2\right)\sim \left(\zeta z_1, \zeta^q z_2\right)},
\]
which is a local model of a cyclic quotient surface singularity of type $\tfrac{1}{p}(1,q)$. Over $D^4$, the canonical line bundle $K_{D^4}\to D^4$ is trivial, because the $2$-form $dz_1\wedge dz_2$ defines a nowhere vanishing section. Also, the pullback of $dz_1\wedge dz_2$ is given by \[
\zeta^*(dz_1\wedge dz_2)=\zeta^{q+1}(dz_1\wedge dz_2),
\]
where we regard $\zeta$ as a map $D^4\to D^4$ defined by $\zeta \left(z_1,z_2\right)= \left(\zeta z_1, \zeta^q z_2\right)$. This shows that the canonical line orbibundle $K_{C}$ over the cone $C$ is given by the quotient \begin{equation}\label{eq:orbibundle_over_cone}
    K_{C}\cong  \frac{D^4\times \mathbb{C}}{\mathbb{Z}_p} =\frac{D^4\times \mathbb{C}}{(z_1,z_2,v)\sim (\zeta z_1,\zeta^q z_2, \zeta^{-(q+1)}v )},
\end{equation}
noting that $K_{C}$ is defined to be the quotient of $K_{D^4}$ under the identification \[
 ( z_1,z_2,v) \sim ( \zeta(z_1,z_2), (\zeta^*)^{-1}v)
\]
for $(z_1,z_2)\in D^4$ and $v\in K_{D^4}|_{(z_1,z_2)}$. In particular, the restriction $K_C|_{L(p,q)}$ of $K_C$ to $L(p,q)$ is isomorphic to $\mathcal{L}_{p,q}^{\otimes -(q+1)}$. 

Now recall that $C$ carries a canonical $\operatorname{spin}^c$ structure $\mathfrak{s}_{\operatorname{can}}$ induced by its complex structure. Its $+$-spinor bundle is given by $\underline{\mathbb{C}} \oplus K_C^{-1}$, where $\underline{\mathbb{C}}$ denotes the trivial complex line bundle. The determinant line bundle of this canonical $\operatorname{spin}^c$ structure is the anti-canonical line bundle $K_C^{-1}$. Note that the conjugate $\overline{\mathfrak{s}}_{\mathrm{can}}$ is obtained by twisting $\mathfrak{s}_{\mathrm{can}}$ by $K_C$, and its determinant line bundle is precisely $K_C$.

\begin{remark}\label{rmk:orbifold_spin^c}
In the discussion above, and in the proof of Proposition \ref{prop:canonical_spin^c_structure} below, $\operatorname{spin}^c$ structures on the cone $C$ are understood in the orbifold sense. If one prefers to work without orbifold $\operatorname{spin}^c$ structures, one may instead work on the smooth manifold $C-\{0\}$. Indeed, the restriction map $\operatorname{Spin}^c_{\mathrm{orb}}(C)\to \operatorname{Spin}^c(C-\{0\})$ is a bijection \cite[Proposition 2]{Fukumoto-2000}.    
\end{remark}

The following proposition, communicated to us by Ue, describes the restriction of $\overline{\mathfrak{s}}_{\mathrm{can}}$ to $L(p,q)$ in terms of the Ozsv\'ath--Szab\'o identification. We include a sketch of the proof for completeness.

\begin{proposition}[{\cite{Ue-personal}}]\label{prop:canonical_spin^c_structure} 
    If $p$ is odd, then 
    \[
        \overline{\mathfrak{s}}_{\mathrm{can}}|_{L(p,q)}=\mathfrak{s}_{p,q}(p-1)\in \operatorname{Spin}^c(L(p,q)).
    \]
    If $p$ is even and $\sigma(q,p,+1)\neq \sigma(q,p,-1)$, then the same conclusion holds, where 
    \begin{equation}\label{eq:def_of_sigma}
        \sigma(q,p,\epsilon)=\frac{1}{p}\sum_{k=1}^{p-1} \left(\cot \frac{\pi k}{p}\cot \frac{\pi k q}{p} +2\epsilon^k \csc \frac{\pi k }{p}\csc \frac{\pi k q}{p} \right).
    \end{equation}
\end{proposition}

\begin{remark}\label{rmk:sigma_and_d_invariants} For even $p$, we have $\sigma(q,p,+1)=\sigma (q,p,-1)$ if and only if the Heegaard Floer $d$-invariants for the two spin structures on $L(p,q)$ are the same \cite[Proposition 4.8]{Ue-2009}. For example, we have $\sigma(1,2,+1)=1 \neq -1 =\sigma(1,2,-1)$, so Proposition \ref{prop:canonical_spin^c_structure} implies that $\overline{\mathfrak{s}}_{\mathrm{can}}|_{L(2,1)}=\mathfrak{s}_{2,1}(1) \in \operatorname{Spin}^c(L(2,1))$.  
\end{remark}

\begin{proof}[Proof of Proposition \ref{prop:canonical_spin^c_structure}] The $+$-spinor bundle for any $\operatorname{spin}^c$ structure $\mathfrak{s}$ on $C$ is of the form $(\underline{\mathbb{C}}\oplus K_C^{-1})\otimes L$ for some line orbibundle $L$ over $C$. The spinor bundle for the restriction $\mathfrak{s}|_{L(p,q)}$ is the restriction of the $+$-spinor bundle for $\mathfrak{s}$, so it is given by \[
(\underline{\mathbb{C}}\oplus K_C^{-1})|_{L(p,q)}\otimes \mathcal{L}_{p,q}^{\otimes k} \cong \mathcal{L}_{p,q}^{\otimes k}\oplus \mathcal{L}_{p,q}^{\otimes (q+1+k)},
\] 
where $k\in \mathbb{Z}$ is an integer (uniquely determined modulo $p$) satisfying $L|_{L(p,q)}\cong \mathcal{L}_{p,q}^{\otimes k}$. This shows that $\mathfrak{s}|_{L(p,q)}$ comes from a spin structure on $L(p,q)$ if and only if $2k\equiv -(q+1) \pmod p$. Therefore, spin structures on $L(p,q)$ correspond to integers $k$ (modulo $p$) satisfying $2k\equiv -(q+1) \pmod p$. On the other hand, taking $L=K_C$ shows that the $\operatorname{spin}^c$ structure $\mathfrak{s}_{p,q}^K:=\overline{\mathfrak{s}}_{\mathrm{can}}|_{L(p,q)}$ corresponds to $k\equiv -(q+1) \pmod{p}$.

Now consider the case where $p$ is odd. In this case, exactly one of $(q-1)/2$ and $(p+q-1)/2$ is an integer. Also, $L(p,q)$ has a unique spin structure, and the induced $\operatorname{spin}^c$ structure is either $\mathfrak{s}_{p,q}((q-1)/2)$ or $\mathfrak{s}_{p,q}((p+q-1)/2)$, depending on which of $(q-1)/2$ and $(p+q-1)/2$ is an integer. If the spin structure corresponds to $\mathfrak{s}_{p,q}((q-1)/2)$, i.e., if $q$ is odd, then the integer $k$ of the preceding paragraph must satisfy $k\equiv -(q+1)/2\pmod p$, which gives \begin{align*}
    \mathfrak{s}_{p,q}^K-\mathfrak{s}_{p,q}((q-1)/2) & = \left(\mathfrak{s}_{\mathrm{can}}|_{L(p,q)}+ c_1\left(\mathcal{L}_{p,q}^{\otimes -(q+1)}\right)\right) -\left(\mathfrak{s}_{\mathrm{can}}|_{L(p,q)}+ c_1\left( \mathcal{L}_{p,q}^{\otimes -(q+1)/2}\right)\right)   
    \\&= -\frac{q+1}{2}\cdot c_1\left(\mathcal{L}_{p,q}\right).
\end{align*}
Hence $\mathfrak{s}_{p,q}^K=\mathfrak{s}_{p,q}(p-1)$ by Proposition \ref{prop:spin^c_affine}. When the spin structure on $L(p,q)$ corresponds to $\mathfrak{s}_{p,q}((p+q-1)/2)$, i.e., when $q$ is even, we have $k\equiv -(p+q+1)/2\pmod p$, and we get the same result.

Next, assume that $p$ is even (so that $q$ is odd). In this case, $L(p,q)$ has two spin structures, corresponding to $\mathfrak{s}_{p,q}((q-1)/2)$ and $\mathfrak{s}_{p,q}((p+q-1)/2)$. If these correspond to $k\equiv -(q+1)/2\pmod p$ and $k\equiv -(p+q+1)/2\pmod p$, respectively, then we have $\mathfrak{s}_{p,q}^K=\mathfrak{s}_{p,q}(p-1)$ as before. If $\mathfrak{s}_{p,q}((q-1)/2)$ and $\mathfrak{s}_{p,q}((p+q-1)/2)$ correspond to  $k\equiv -(p+q+1)/2 \pmod p$ and $k\equiv -(q+1)/2\pmod p$, respectively, then we can show that $\sigma(q,p,+1)=\sigma(q,p,-1)$ by combining the results of \cite{Fukumoto-Furuta-Ue-2001} and \cite{Ue-2009}.
\end{proof} 

For relatively prime integers $p > q > 0$, if the Hirzebruch-Jung continued fraction of $p/q$ is given by \[ \frac{p}{q} = [a_1, \dots, a_s] = a_1 - \frac{1}{a_2 - \dfrac{1}{\cdots - \dfrac{1}{a_s}}} \quad (a_i \geq 2), \]
then the lens space $L(p,q)$ is the boundary of the plumbed $4$-manifold $X(p,q)$ constructed according to the linear plumbing graph with weights $-a_1,\dots,-a_s$. The manifold $X(p,q)$ is smooth, compact, oriented, simply-connected, and its intersection form $Q_{X(p,q)}$ is negative definite. Let $v_1,\dots,v_s\in H_2(X(p,q);\mathbb{Z})$ be the homology classes of the spheres corresponding to the vertices with weights $-a_1,\dots, -a_s$, respectively. These classes form a basis for $H_2(X(p,q);\mathbb{Z})$, called the \emph{standard basis} for $Q_{X(p,q)}$. Their intersection pairings are given by  \begin{equation}\label{eq:intersection_form_of_X(p,q)}
   v_i\cdot v_j=Q_{X(p,q)}(v_i,v_j)=\begin{cases}
    -a_i & \text{if } i=j, \\
    1 &  \text{if } |i-j|=1, \\
    0 & \text{if }  |i-j|\geq 2.
\end{cases} 
\end{equation}
Let $v_1^*,\dots,v_s^*\in H^2(X(p,q);\mathbb{Z})$ be the dual basis satisfying $\langle v_i^*,v_j\rangle =\delta_{ij}$. We define \begin{equation}\label{eq:def_of_K_p,q}
    K_{p,q}:=2v_s^*+\sum_{i=1}^s (a_i-2)v_i^* \in H^2(X(p,q);\mathbb{Z}).
\end{equation}
The additional term $2v_s^*$ is chosen so that the corresponding $\operatorname{spin}^c$ structure on $X(p,q)$ restricts to $\mathfrak{s}_{p,q}(0)$ on the boundary. 
\begin{proposition}\label{prop:K_p,q}
    The class $K_{p,q}$ is a characteristic covector in $\operatorname{Char}(X(p,q))\subset H^2(X(p,q);\mathbb{Z})$ with \[
     K_{p,q}^2= -s-2\left(1-\frac{1}{p}\right)-12s(q,p)=-b_2(X(p,q))-2\left(1-\frac{1}{p}\right)-12s(q,p),
    \]
    where $s(q,p)$ denotes the Dedekind sum \begin{equation}\label{eq:Dedekind}
            s(q,p)=\frac{1}{4p}\sum_{k=1}^{p-1} \cot \frac{k\pi}{p} \cot \frac{qk\pi}{p}.
    \end{equation}
    Moreover, if either $p$ is odd or $\sigma(q,p,+1)\neq \sigma(q,p,-1)$, then \[
    \mathfrak{s}_{K_{p,q}}|_{L(p,q)}=\mathfrak{s}_{p,q}(0)\in \operatorname{Spin}^c(L(p,q)),
    \]
    where $\mathfrak{s}_{K_{p,q}}\in \operatorname{Spin}^c(X(p,q))$ is the $\operatorname{spin}^c$ structure on $X(p,q)$ with $c_1(\mathfrak{s}_{K_{p,q}})=K_{p,q}$.
\end{proposition}
\begin{proof}
    Let $K_0:=\sum_{i=1}^s (a_i-2)v_i^*$ so that $K_{p,q}=K_0+2v_s^*$. To show that $K_{p,q}\in \operatorname{Char}(X(p,q))$, it suffices to show that $K_0\in \operatorname{Char}(X(p,q))$. Indeed, for any homology class $v=\sum_{i=1}^s b_iv_i\in H_2(X(p,q);\mathbb{Z})$, we have \begin{align*}
        \langle K_0,v\rangle = \sum_{i=1}^s (a_i-2)b_i \equiv \sum_{i=1}^s a_ib_i \equiv \sum_{i=1}^s -a_ib_i^2 =\sum_{i=1}^s b_i^2 v_i^2 \equiv \left( \sum_{i=1}^s b_iv_i\right)^2 =v^2 \pmod{2}.
    \end{align*}

    Next, let us compute $K_{p,q}^2$. First, we have \begin{equation}\label{eq:K_0^2}
         K_0^2=-s+2\left(1-\frac{1}{p}\right)-12s(q,p);
    \end{equation}
    see \cite[p.304, 7.1]{Nemethi-Nicolaescu-2002}. 

    \begin{claim}\label{claim}
        We have $(v_s^*)^2+K_0\cdot v_s^*=-\left(1-\dfrac{1}{p}\right)$. 
    \end{claim}

    The proof of Claim \ref{claim} will be given below. Combining \eqref{eq:K_0^2} with Claim \ref{claim} gives \begin{align*}
        K_{p,q}^2 &= (K_0+2v_s^*)^2= K_0^2+4K_0\cdot v_s^*+4(v_s^*)^2 =-s-2\left(1-\frac{1}{p}\right)-12s(q,p).
    \end{align*}

    Finally, let us prove the last statement. Let $\mathfrak{s}_{K_0}, \mathfrak{s}_{K_{p,q}} \in \operatorname{Spin}^c(X(p,q))$ denote the $\operatorname{spin}^c$ structures on $X(p,q)$ whose first Chern classes are $K_0$, $K_{p,q}$, respectively. Since $c_1(\mathfrak{s}_{K_{p,q}})=c_1(\mathfrak{s}_{K_0})+2v_s^*$, we must have $\mathfrak{s}_{K_{p,q}}=\mathfrak{s}_{K_0}+v_s^*$; see \eqref{eq:spin^c-char_assignment}. On the other hand, it is proved in \cite[Section 5]{Ue-2009} that the restriction $\mathfrak{s}_{K_0}|_{L(p,q)}$ is precisely $\overline{\mathfrak{s}}_{\mathrm{can}}|_{L(p,q)}$, and that the image of $v_s^*$ under the map $H^2(X(p,q);\mathbb{Z})\to H^2(L(p,q);\mathbb{Z})$ is $c_1(\mathcal{L}_{p,q})$. It follows from Propositions \ref{prop:spin^c_affine} and \ref{prop:canonical_spin^c_structure}  that \[
    \mathfrak{s}_{K_{p,q}}|_{L(p,q)}=\left( \mathfrak{s}_{K_0}+v_s^*\right)|_{L(p,q)}=\mathfrak{s}_{K_0}|_{L(p,q)}+v_s^*|_{L(p,q)}=\mathfrak{s}_{p,q}(p-1)+c_1(\mathcal{L}_{p,q})=\mathfrak{s}_{p,q}(0).
    \]
\end{proof}

\begin{proof}[Proof of Claim \ref{claim}]
   Write $v_s^*$ as a $\mathbb{Q}$-linear combination of $v_1,\dots,v_s$ as $v_s^*=\sum_{i=1}^s x_iv_i$; recall our convention \eqref{eq:second_cohomology_as_dual_lattice} that $H^2(X(p,q);\mathbb{Z})\subset H_2(X(p,q);\mathbb{Z})\otimes \mathbb{Q}$. For each $i=1,\dots,s$, let $D_i>0$ be the numerator of the continued fraction $[a_1,\dots,a_i]$, and let $D_0=1$ and $D_{-1}=0$. Then it is easy to see that \[
   D_i=a_iD_{i-1}-D_{i-2} \quad (i=1,\dots,s).
   \]
   Let us show that \begin{equation}\label{eq:claim}
       x_i=D_{i-1}x_1 \quad (i=0,1,\dots,s),
   \end{equation}
   where we put $x_0=0$. We proceed by induction. The cases $i=0$ and $i=1$ are obvious, so assume $1<i\leq s$. By \eqref{eq:intersection_form_of_X(p,q)}, we have \[
   0=v_s^*\cdot v_{i-1}=x_{i-2}-a_{i-1}x_{i-1}+x_i.
   \]
   Thus, the induction hypothesis implies \[
   x_i=a_{i-1}x_{i-1}-x_{i-2}=(a_{i-1}D_{i-2}-D_{i-3})x_1= D_{i-1}x_1,
   \]
   proving \eqref{eq:claim}. Next, from \[
   1=v_s^*\cdot v_s=x_{s-1}-a_sx_s=(D_{s-2}-a_sD_{s-1})x_1=-D_sx_1=-px_1,
   \]
   we get $x_1=-1/p$, and therefore \[
   v_s^*\cdot v_s^*= x_s=D_{s-1}x_1=-\frac{D_{s-1}}{p}. 
   \]
   Also, we have \begin{align*}
       K_0\cdot v_s^*& = \sum_{i=1}^s (a_i-2)x_i =-\sum_{i=1}^s (a_i-2) \frac{D_{i-1}}{p}=-\frac{1}{p}\sum_{i=1}^s (a_iD_{i-1}-2D_{i-1}) \\
       &= -\frac{1}{p}\sum_{i=1}^s (D_i+D_{i-2}-2D_{i-1}) = -\frac{1}{p}(D_s-D_{s-1}-1)=-\frac{p-D_{s-1}-1}{p}.
   \end{align*}
   Now the result follows immediately.   
\end{proof}

\begin{remark}\label{rmk:K_p,q_in_terms_of_v_i} We can write the class $K_{p,q}$ as a $\mathbb{Q}$-linear combination $K_{p,q} = \sum_{i=1}^s y_i(p,q) v_i$ of the standard basis vectors $v_1,\dots,v_s$, where the coefficients are given by: \begin{equation}\label{eq:y_i}
    \begin{pmatrix}
    y_1(p,q) \\ y_2(p,q) \\ y_3(p,q) \\ \vdots \\ y_{s-1}(p,q) \\ y_s(p,q) 
\end{pmatrix} = \begin{pmatrix}
    -a_1 & 1 & 0 & \cdots & 0 & 0  \\
    1 & -a_2 & 1 & \cdots & 0 & 0 \\
    0 & 1 & -a_3 & \cdots & 0 &0 \\
    \vdots & \vdots & \vdots & \ddots  & \vdots & \vdots  \\
    0 & 0 &0 & \cdots & -a_{s-1} & 1 \\
    0 & 0 & 0 & \cdots & 1 & -a_s
\end{pmatrix}^{-1}\begin{pmatrix}
    a_1-2 \\ a_2-2 \\ a_3-2 \\ 
    \vdots \\ a_{s-1}-2 \\ a_s
\end{pmatrix}.
\end{equation} 
\end{remark}

\vspace{2mm}
Finally, we analyze the behavior of $\operatorname{spin}^c$ structures on lens spaces under orientation-reversing diffeomorphisms. 
\begin{proposition}\label{prop:orientation-reversing}
    For relatively prime integers $p>q>0$, consider the orientation-reversing diffeomorphism $f=f_{p,q}\colon L(p,q)\to L(p,p-q)$ defined by $f[z_1,z_2]=[\bar{z}_1,z_2]$, where $\bar{z}_1$ denotes the complex conjugate of $z_1$. If either $p$ is odd or $\sigma(q,p,+1)\neq \sigma(q,p,-1)$, then \[
    f^*\mathfrak{s}_{p,p-q}(i)=\mathfrak{s}_{p,q}(p-1-i)\in \operatorname{Spin}^c(L(p,q))
    \]
    for all $i\in \mathbb{Z}_p$. 
\end{proposition}
\begin{proof}
Recall the complex line bundle \[
\mathcal{L}_{p,p-q}= \frac{S^3\times \mathbb{C}}{(z_1,z_2,v)\sim (\zeta z_1,\zeta^{-q}z_2,\zeta v)}
\]
over $L(p,p-q)$ defined in \eqref{eq:line_bundle_L_rho}, where $\zeta=e^{2\pi i/p}$. Regarding the pullback bundle $f^*\mathcal{L}_{p,p-q}$ as a subspace of $L(p,q)\times \mathcal{L}_{p,p-q}$, the map \begin{align*}
    f^*\mathcal{L}_{p,p-q} & \longrightarrow \frac{S^3\times \mathbb{C}}{(z_1,z_2,v)\sim (\zeta z_1,\zeta^q z_2,\zeta^{-1}v)}, \\
    ([z_1,z_2],[\bar{z}_1,z_2,v]) & \longmapsto [z_1,z_2,v],
\end{align*}
is a well-defined bundle isomorphism over $L(p,q)$, which shows that $f^*\mathcal{L}_{p,p-q}\cong \mathcal{L}_{p,q}^{-1}$. 

Next, we consider the case where $p$ is odd. If $q$ is even, then $\mathfrak{s}_{p,p-q}((p-q-1)/2)$ and $\mathfrak{s}_{p,q}((p+q-1)/2)$ are the unique spin structures of $L(p,p-q)$ and $L(p,q)$, respectively, so we must have $f^*\mathfrak{s}_{p,p-q}((p-q-1)/2)=\mathfrak{s}_{p,q}((p+q-1)/2)$. Using Proposition \ref{prop:spin^c_affine}, it follows that for any $i\in \mathbb{Z}_p$, \begin{align*}
    f^*\mathfrak{s}_{p,p-q}(i) & = f^*\left(\mathfrak{s}_{p,p-q}((p-q-1)/2)+\left(i-\tfrac{p-q-1}{2}\right) \cdot c_1(\mathcal{L}_{p,p-q})  \right) \\
    & = f^*\mathfrak{s}_{p,p-q}((p-q-1)/2) +\left(i-\tfrac{p-q-1}{2}\right) \cdot c_1(f^*\mathcal{L}_{p,p-q}) \\
    &= \mathfrak{s}_{p,q}((p+q-1)/2) -\left(i-\tfrac{p-q-1}{2}\right) \cdot c_1(\mathcal{L}_{p,q}) = \mathfrak{s}_{p,q}(p-1-i). \end{align*}
If $q$ is odd, then $\mathfrak{s}_{p,p-q}(p-(q+1)/2)$ and $\mathfrak{s}_{p,q}((q-1)/2)$  are the unique spin structures of $L(p,p-q)$ and $L(p,q)$, respectively, and the result can be obtained similarly. 

Now suppose $p$ is even. Then $\mathfrak{s}_{p,q}((q-1)/2)$ and $\mathfrak{s}_{p,q}((p+q-1)/2)$ are the spin structures of $L(p,q)$, and $\mathfrak{s}_{p,p-q}(p-(q+1)/2)$ and $\mathfrak{s}_{p,p-q}((p-q-1)/2)$ are the spin structures of $L(p,p-q)$. In particular, $f^*\mathfrak{s}_{p,p-q}((p-q-1)/2)$ is either $\mathfrak{s}_{p,q}((q-1)/2)$ or $\mathfrak{s}_{p,q}((p+q-1)/2)$. Assuming $\sigma(q,p,+1)\neq \sigma (q,p,-1)$, the $d$-invariants for $\mathfrak{s}_{p,q}((q-1)/2)$ and $\mathfrak{s}_{p,q}((p+q-1)/2)$ are distinct, and therefore we have \begin{align*}
    d(L(p,q),f^*\mathfrak{s}_{p,p-q}((p-q-1)/2))&=-d(L(p,p-q), \mathfrak{s}_{p,p-q}((p-q-1)/2)) \\
    & =d(L(p,q),\mathfrak{s}_{p,q}((p+q-1)/2))  \\
    & \neq d(L(p,q),\mathfrak{s}_{p,q}((q-1)/2)),
\end{align*}
where the second equality follows from a closed formula for $d$-invariants of lens spaces \cite[p.136]{Ue-2009}. It follows that \[
f^*\mathfrak{s}_{p,p-q}((p-q-1)/2)=\mathfrak{s}_{p,q}((p+q-1)/2),
\]
and now we may obtain the result as above.
\end{proof}

\begin{remark}
If we consider the orientation-reversing diffeomorphism $g \colon L(p,q)\to L(p,p-q)$ defined by $g[z_1,z_2]=[z_1,\bar{z}_2]$, then we have $g^*\mathcal{L}_{p,p-q} \cong \mathcal{L}_{p,q}$. A similar argument as above shows that 
\[
g^* \mathfrak{s}_{p,p-q}(i)=\mathfrak{s}_{p,q}(i+q)\in \operatorname{Spin}^c(L(p,q))
\]
under the same assumptions as in Proposition \ref{prop:orientation-reversing}.     
\end{remark}

\section{Rational homology projective planes with quotient singularities}\label{sec:orbifold_BMY}

A normal projective complex surface $S$ whose Betti numbers $b_i(S)$ agree with those of the complex projective plane $\mathbb{CP}^2$ is called a \emph{rational homology projective plane}, or simply a \emph{$\mathbb{Q}$-homology $\mathbb{CP}^2$}. A (necessarily isolated) singular point of $S$ is called a \emph{cyclic quotient singularity} if its germ is analytically isomorphic to $(\mathbb{C}^2/G,0)$, where $G$ is a finite cyclic subgroup of $\operatorname{GL}(2,\mathbb{C})$ acting freely on $\mathbb{C}^2-\{0\}$. In this case, there exist relatively prime positive integers $p>q>0$ such that $G$ is conjugate to the cyclic group $C_{p,q}\subset \operatorname{GL}(2,\mathbb{C})$ of order $p$, generated by the matrix \[
\begin{pmatrix}
    e^{2\pi i /p} & 0 \\ 0 & e^{2\pi i q/p}
\end{pmatrix}.
\]
We then say that the cyclic singularity is of type $\tfrac{1}{p}(1,q)$. The integer $p$ is uniquely determined by the singularity, and $q$ is uniquely determined up to taking its multiplicative inverse modulo $p$. Observe that the link of a cyclic singularity of type $\tfrac{1}{p}(1,q)$ is the lens space $L(p,q)$, and the singularity has a neighborhood analytically isomorphic to the cone $CL(p,q)=D^4/C_{p,q}$, which was discussed in Section \ref{sec:lens_spinc}.

Suppose $S$ is a $\mathbb{Q}$-homology $\mathbb{CP}^2$, whose singularities are all cyclic. We write $S^0$ for the \emph{smooth locus} of $S$, that is, the complement of its singular points. Let $X_S$ denote the smooth, compact $4$-manifold obtained from $S$ by removing cone neighborhoods of the singular points. We will call $X_S$ the \emph{associated $4$-manifold} of $S$. We give $X_S$ the orientation induced from the canonical orientation of $S$. We summarize the topological properties of the $4$-manifold $X_S$ in the following lemma. Proofs can be found in \cite{Jo-Park-Park-2024-02,Jo-Park-Park-2025-01}.
\begin{lemma}\label{lem:topology_of_QHCP2}
    Let $S$ be a rational homology projective plane with exactly $n$ cyclic quotient singularities of types \[
    \tfrac{1}{p_1}(1,q_1),\dots,\tfrac{1}{p_n}(1,q_n).
    \]
    Then the associated $4$-manifold $X_S$ is homotopy equivalent to $S^0$, which has a positive definite intersection form and $b_2(X_S)=1$. Its oriented boundary is given by \[
    \partial X_S=-L(p_1,q_1)\amalg
     \cdots \amalg -L(p_n,q_n),
    \]
    where $-L(p,q)$ denotes the orientation reversal of $L(p,q)$. Furthermore, if $H_1(S^0;\mathbb{Z})=0$, then the $p_i$'s are pairwise relatively prime, and the intersection form $Q_{X_S}$ of $X_S$ is isomorphic to the lattice $\langle p_1\dots p_n \rangle$.   
\end{lemma}

Let $S$ be a $\mathbb{Q}$-homology $\mathbb{CP}^2$ with exactly $n$ cyclic quotient singularities of types \[
    \tfrac{1}{p_1}(1,q_1),\dots,\tfrac{1}{p_n}(1,q_n).
    \]
We denote by $K_S$ the canonical divisor of the surface $S$. By abuse of notation, we will denote both the canonical line orbibundle over $S$, and its first orbifold Chern class in $H^2(S;\mathbb{Q})$ also by $K_S$. Expand $p_i/q_i$ ($i=1,\dots,n$) into its Hirzebruch-Jung continued fraction: \[
\frac{p_i}{q_i}=\left[b_{i,1},\dots,b_{i,\ell_i}\right]=b_{i,1}-\frac{1}{b_{i,2}-\dfrac{1}{\cdots -\dfrac{1}{b_{i,\ell_i}}}}, \quad b_{i,j}\geq 2.
\]
An application of the adjunction formula shows that \begin{equation}\label{eq:K_square}
    K_S^2=9-3L-2n+\sum_{i=1}^n \left(\frac{q_i+q_i'+2}{p_i}+\sum_{j=1}^{\ell_i} b_{i,j} \right),
\end{equation}
where $L=\sum_{i=1}^n \ell_i$, and $q_i'$ is the unique integer such that $0<q_i'<p_i$ and $q_iq_i'\equiv 1\pmod {p_i}$ \cite[Section 3]{Hwang-Keum-2011-2}. On the other hand, it is known that the Dedekind sum $s(q_i,p_i)$ satisfies the following relation \cite[Lemma 2.3]{Holzapfel-1988}: \[
s(q_i,p_i)= \frac{1}{12} \left(\frac{q_i+q_i'}{p_i}-3\ell_i+\sum_{j=1}^{\ell_i} b_{i,j} \right).
\]
Thus, we can re-express the formula \eqref{eq:K_square} as \begin{equation}\label{eq:K_square_2}
    K_S^2=9+12\sum_{i=1}^n s(q_i,p_i) -2\sum_{i=1}^n \left( 1-\frac{1}{p_i}\right).
\end{equation}

\begin{remark}
Recall that the orbifold Euler characteristic of a normal complex surface $S$ with quotient singularities is given by 
\begin{equation}\label{eq:orbifold_euler_characteristic}
    e_{\textup{orb}}(S)=e(S)-\sum_{p\in \operatorname{Sing}(S)} \left(1-\frac{1}{|G_p|}\right),
\end{equation}
where $e(S)$ is the ordinary Euler characteristic of $S$ and $G_p$ denotes the local fundamental group of a singular point $p\in \operatorname{Sing}(S)$ \cite{Satake-1957,Kollar-2008}. In particular, if $S$ is as in Lemma \ref{lem:topology_of_QHCP2}, then we have \begin{equation*}
    e_{\textup{orb}}(S)=3-\sum_{i=1}^n\left(1- \frac{1}{p_i}\right).
\end{equation*}
Combining this with \eqref{eq:K_square_2}, we have the following computation: \begin{align*}
\frac{1}{3}p_1^{\textup{orb}}(S) & =\frac{1}{3}\left( c_1^{\textup{orb}}(S)^2 -2c_2^{\textup{orb}}(S)\right) = \frac{1}{3}\left( K_S^2-2e_{\textup{orb}}(S) \right) \\ &= 1+4\sum_{i=1}^n s(q_i,p_i) =\sigma(S)+4\sum_{i=1}^n s(q_i,p_i),
\end{align*}
where $\sigma(S)=1$ is the signature of the intersection form of $S$. This formula is precisely the content of Kawasaki's orbifold signature theorem \cite{Kawasaki-1978} applied to $S$. Conversely, the formula \eqref{eq:K_square_2} can also be derived from the orbifold signature theorem.     
\end{remark}

The class $K_S\in H^2(S;\mathbb{Q})$ restricts to an integral class  $K_{X_S}:=K_S|_{X_S}\in H^2(X_S;\mathbb{Z})$, since the orbibundle $K_S\to S$ is a genuine line bundle over $X_S$. 
\begin{proposition}\label{prop:K_X_S} Let $S$ be a rational homology projective plane with $H_1(S^0;\mathbb{Z})=0$, having exactly $n$ cyclic quotient singularities of types \[
    \tfrac{1}{p_1}(1,q_1),\dots,\tfrac{1}{p_n}(1,q_n).
    \]
Then the class $K_{X_S}$ is a characteristic covector in $\operatorname{Char}(X_S)\subset H^2(X_S;\mathbb{Z})$ with \[
K_{X_S}^2= 9+12\sum_{i=1}^n s(q_i,p_i) -2\sum_{i=1}^n \left( 1-\frac{1}{p_i}\right).
\]
Moreover, if $\mathfrak{s}_{K_{X_S}}\in \operatorname{Spin}^c(X_S)$ is the $\operatorname{spin}^c$ structure on $X_S$ with $c_1(\mathfrak{s}_{K_{X_S}})=K_{X_S}$, then its restriction to each boundary component $-L(p_i,q_i)$ is given by  \begin{equation}\label{eq:restriction_of_K_X_S}
    \mathfrak{s}_{K_{X_S}}|_{-L(p_i,q_i)}=\mathfrak{s}_{p_i,q_i}(p_i-1)\in \operatorname{Spin}^c(-L(p_i,q_i))\cong \operatorname{Spin}^c(L(p_i,q_i)),
\end{equation}
assuming either $p_i$ is odd or $\sigma(q_i,p_i,+1)\neq \sigma(q_i,p_i,-1)$. Furthermore, $p_1\cdots p_n K_{X_S}^2$ is a nonzero square number, which is even if and only if the product $p_1\cdots p_n$ is even. 
\end{proposition}
\begin{proof}
    We have $K_{X_S}\in \operatorname{Char}(X_S)$ since $K_{X_S}=-c_1(TX_S)$ is an integral lift of $w_2(TX_S)$. The formula for $K_{X_S}^2$ directly follows from \eqref{eq:K_square_2}, noting that $K_S^2=K_{X_S}^2$. 
    
    Next, let $\mathfrak{s}$ denote the canonical orbifold $\operatorname{spin}^c$ structure on $S$ induced by the complex structure. Recall that the singular point of $S$ of type $\tfrac{1}{p_i}(1,q_i)$ has a neighborhood that is analytically isomorphic to the cone $CL(p_i,q_i)=D^4/\mathbb{Z}_{p_i}$, so that we have a decomposition $S=(\bigcup_{i=1}^n CL(p_i,q_i)) \cup_\partial X_S$. Consider the commutative diagram \[
    \begin{tikzcd}[matrix scale=0.5]
    & \operatorname{Spin}^c(X_S) \ar[r] & \operatorname{Spin}^c(-L(p_i,q_i)) \ar[dd, leftrightarrow, "\cong"] \\
        \operatorname{Spin}^c_{\mathrm{orb}}(S) \ar[ru] \ar[rd] \\&  \operatorname{Spin}^c_{\mathrm{orb}}(CL(p_i,q_i)) \ar[r] & \operatorname{Spin}^c(L(p_i,q_i)),
    \end{tikzcd}
    \]
    where the right vertical map is the canonical identification described in Section \ref{subsec:spin^c_structures}, and all the other maps are restriction maps. The restriction $\mathfrak{s}|_{CL(p_i,q_i)}$ is the canonical $\operatorname{spin}^c$ structure $\mathfrak{s}_{\mathrm{can}}$ on the cone, and we have $\overline{\mathfrak{s}}_{\mathrm{can}}|_{L(p_i,q_i)}=\mathfrak{s}_{p_i,q_i}(p_i-1)$ by Proposition \ref{prop:canonical_spin^c_structure}. On the other hand, since $\det (\overline{\mathfrak{s}})=K_S$, we have $c_1(\overline{\mathfrak{s}}|_{X_S})=K_S|_{X_S}=K_{X_S}$. Thus, we must have $\overline{\mathfrak{s}}|_{X_S}=\mathfrak{s}_{K_{X_S}}$ since the map $c_1\colon\operatorname{Spin}^c(X_S)\to \operatorname{Char}(X_S)$ is a bijection. Now \eqref{eq:restriction_of_K_X_S} follows by applying the commutative diagram to $\overline{\mathfrak{s}}\in \operatorname{Spin}^c_{\mathrm{orb}}(S)$.
    
    To prove the final statement, let $v\in H_2(X_S;\mathbb{Z})\cong \mathbb{Z}$ be a generator. We have $v^2=p_1\cdots p_n$ by Lemma \ref{lem:topology_of_QHCP2}, so its dual generator $v^*\in H^2(X_S;\mathbb{Z})$ satisfies $(v^*)^2=1/(p_1\cdots p_n)$. Let $m\in \mathbb{Z}$ be the integer such that $K_{X_S}=mv^*\in H^2(X_S;\mathbb{Z})$. Under the assumption $H_1(S^0;\mathbb{Z})=0$, $m$ must be nonzero since either $K_S$ or $-K_S$ is ample. It follows that \[
    p_1\cdots p_n K_{X_S}^2=p_1\cdots p_n \cdot \frac{m^2}{p_1\cdots p_n}=m^2
    \]
    is a nonzero square number. Finally, since $K_{X_S}\in \operatorname{Char}(X_S)$, $m$ is even if and only if the product $p_1\cdots p_n$ is even. This completes the proof.
\end{proof}

We end this section by stating the following orbifold BMY inequality, which will play a crucial role in the proof of Theorem \ref{thm:aMY}. 

\begin{theorem}[Orbifold BMY inequality, {\cite{Sakai-1980,Miyaoka-1984,KoNS-1989,Megyesi-1999}}]\label{thm:oBMY} 
  Let $S$ be a normal projective surface with quotient singularities, whose canonical divisor $K_S$ is nef. Then \begin{equation*} 
            K_S^2 \leq 3e_{\textup{orb}}(S),
       \end{equation*}
 where $e_{\textup{orb}}(S)$ denotes the orbifold Euler characteristic \eqref{eq:orbifold_euler_characteristic} of $S$. 
\end{theorem}

\section{Proof of Theorem \ref{thm:main}}\label{sec:main_thm_proof}

In this section, we give a proof of Theorem \ref{thm:main}. For relatively prime integers $p>q>0$, we recall the following notations established in the previous sections:
\begin{itemize}
    \item $\mathfrak{s}_{p,q}(i) \in \operatorname{Spin}^c(L(p,q))$ denotes the $\operatorname{spin}^c$ structure on $L(p,q)$ corresponding to $i\in \mathbb{Z}_p$ under the Ozsv\'ath-Szab\'o identification $\operatorname{Spin}^c(L(p,q))\cong \mathbb{Z}_p$. 
    \item $X(p,q)$ denotes the negative definite plumbed $4$-manifold with $\partial X(p,q)=L(p,q)$.
    \item $K_{p,q}\in H^2(X(p,q);\mathbb{Z})$ denotes the characteristic covector defined in \eqref{eq:def_of_K_p,q}. 
    \item $\mathfrak{s}_{K_{p,q}}\in \operatorname{Spin}^c(X(p,q))$ denotes the $\operatorname{spin}^c$ structure on $X(p,q)$ satisfying $c_1(\mathfrak{s}_{K_{p,q}})=K_{p,q}$. 
    \item $f_{p,q}\colon L(p,q)\to L(p,p-q)$ denotes the orientation-reversing diffeomorphism defined by $f_{p,q}[z_1,z_2]=[\bar{z}_1,z_2]$. 
\end{itemize} 

Let $S$ be a $\mathbb{Q}$-homology $\mathbb{CP}^2$ with exactly $n$ cyclic quotient singularities of types 
\[
    \tfrac{1}{p_1}(1,q_1),\dots,\tfrac{1}{p_n}(1,q_n),
\]
and assume that its smooth locus $S^0$ satisfies $H_1(S^0;\mathbb{Z})=0$. We also recall the following notations regarding $S$: 
\begin{itemize}
    \item $X_S$ denotes the associated $4$-manifold obtained from $S$ by removing cone neighborhoods of the singular points. 
    \item $K_{X_S}\in H^2(X_S;\mathbb{Z})$ denotes the restriction to $X_S$ of the first orbifold Chern class $K_S\in H^2(S;\mathbb{Q})$ of the canonical line orbibundle over $S$. 
    \item $\mathfrak{s}_{K_{X_S}}\in \operatorname{Spin}^c(X_S)$ denotes the $\operatorname{spin}^c$ structure on $X_S$ satisfying $c_1(\mathfrak{s}_{K_{X_S}})=K_{X_S}$. We use the same symbol $\mathfrak{s}_{K_{X_S}}$ for the corresponding $\operatorname{spin}^c$ structure on the orientation reversal $-X_S$ of $X_S$.
\end{itemize}
Let us assume that for each $i=1,\dots,n$, either $p_i$ is odd or $\sigma(q_i,p_i,+1)\neq \sigma(q_i,p_i,-1)$, where $\sigma$ is the function given in \eqref{eq:def_of_sigma}. Recall from Lemma \ref{lem:topology_of_QHCP2} that \[
\partial X_S=\coprod_{i=1}^n -L(p_i,q_i).
\]
Thus, the boundary of the orientation reversal $-X_S$ of $X_S$ is given by \[
\partial (-X_S)=-\partial X_S=\coprod_{i=1}^n L(p_i,q_i).
\]
On the other hand, the boundary of the disjoint union $X':=\coprod_{i=1}^n X(p_i,p_i-q_i)$ is given by \[
\partial X'=\coprod_{i=1}^n \partial X(p_i,p_i-q_i) = \coprod_{i=1}^n L(p_i,p_i-q_i).
\] 
Let $f\colon\partial(-X_S)\to \partial X'$ be the orientation-reversing diffeomorphism given by \[
f=\coprod_{i=1}^n f_{p_i,q_i}\colon \coprod_{i=1}^n L(p_i,q_i) \longrightarrow \coprod_{i=1}^n L(p_i,p_i-q_i). 
\]
Let $\mathfrak{s}'=\coprod_{i=1}^n \mathfrak{s}_{K_{p_i,p_i-q_i}}\in \operatorname{Spin}^c(X')$ denote the $\operatorname{spin}^c$ structure on $X'$ such that \[
\mathfrak{s}'|_{X(p_i,p_i-q_i)}=\mathfrak{s}_{K_{p_i,p_i-q_i}} \quad (i=1,\dots,n).
\] 
\begin{claim}\label{claim:spin^c}
    We have $f^*(\mathfrak{s}'|_{\partial X'})=\mathfrak{s}_{K_{X_S}}|_{\partial(-X_S)}$. 
\end{claim}
\begin{proof}
    By considering each boundary component separately, it suffices to show that \[
    f^* (\mathfrak{s}'|_{L(p_i,p_i-q_i)}) = \mathfrak{s}_{K_{X_S}}|_{L(p_i,q_i)} \quad (i=1,\dots,n).
    \]
   Recall that we assume either $p_i$ is odd or $\sigma(q_i,p_i,+1)\neq \sigma(q_i,p_i,-1)$. Also note that $\sigma(p_i-q_i,p_i,\pm 1)=-\sigma(q_i,p_i,\mp 1)$). Thus we have \begin{align*}
       f^* (\mathfrak{s}'|_{L(p_i,p_i-q_i)})  & = f^*(\mathfrak{s}_{K_{p_i,p_i-q_i}}|_{L(p_i,p_i-q_i)}) & (\text{by the definition of $\mathfrak{s}'$}) \\
       & = f^*(\mathfrak{s}_{p_i,p_i-q_i}(0)) & (\text{by Proposition \ref{prop:K_p,q}}) \\
       & = f_{p_i,q_i}^*(\mathfrak{s}_{p_i,p_i-q_i}(0)) & (\text{by the definition of $f$}) \\
       & = \mathfrak{s}_{p_i,q_i}(p_i-1) & (\text{by Proposition \ref{prop:orientation-reversing}}) \\
       & = \mathfrak{s}_{K_{X_S}}|_{L(p_i,q_i)} & (\text{by Proposition \ref{prop:K_X_S}}). 
   \end{align*}
   This completes the proof of the claim. 
\end{proof}

Claim \ref{claim:spin^c} implies the existence of a global $\operatorname{spin}^c$ structure $\mathfrak{s}$ on the closed $4$-manifold $W_S:=X'\cup_f (-X_S)$ such that $\mathfrak{s}|_{X'}=\mathfrak{s}'$ and $\mathfrak{s}|_{-X_S}=\mathfrak{s}_{K_{X_S}}$. Applying Proposition \ref{prop:spin^c_gluing}, we obtain a lattice embedding  \begin{equation}\label{eq:embedding_iota}
    \iota\colon Q_{X'}\oplus Q_{-X_S} \longrightarrow -\mathbb{Z}^N,
\end{equation}
    where \[
    N=b_2(W_S)=b_2(X')+b_2(-X_S)=\sum_{i=1}^n b_2(X(p_i,p_i-q_i)) +1,
    \] 
such that its $\mathbb{Q}$-linear extension $\iota_{\mathbb{Q}}$ maps $c_1(\mathfrak{s}')+c_1(\mathfrak{s}_{K_{X_S}})\in Q_{X'}^*\oplus Q_{-X_S}^*$ to a vector $\mathbf{K}:=\sum_{i=1}^N a_i e_i\in -\mathbb{Z}^N$ for some odd integers $a_1,\dots,a_N$ satisfying \[
   c_1(\mathfrak{s}')^2+c_1(\mathfrak{s}_{K_{X_S}})^2= -\sum_{i=1}^N a_i^2.
    \]
Note that $\mathbf{K}$ is precisely the image of $c_1(\mathfrak{s})$ under $Q_{W_S}^*=Q_{W_S}\cong -\mathbb{Z}^N$. 

By Proposition \ref{prop:K_p,q}, we have \begin{align*}
    c_1(\mathfrak{s}')^2 &= \sum_{i=1}^n c_1(\mathfrak{s}_{K_{p_i,p_i-q_i}})^2 =\sum_{i=1}^n K_{p_i,p_i-q_i}^2 \\
     &=\sum_{i=1}^n \left( -b_2(X(p_i,p_i-q_i))-2\left(1-\frac{1}{p_i}\right)-12s(p_i-q_i,p_i)\right) \\
    & =\sum_{i=1}^n \left(  -b_2(X(p_i,p_i-q_i))-2\left(1-\frac{1}{p_i}\right)+12s(q_i,p_i)\right),
\end{align*}
where the last equality uses the property $s(p-q,p)=-s(q,p)$ of Dedekind sums, which easily follows from the formula \eqref{eq:Dedekind}. On the other hand, Proposition \ref{prop:K_X_S} shows that \[
c_1(\mathfrak{s}_{K_{X_S}})^2=-K_{X_S}^2=-\left(9+12\sum_{i=1}^n s(q_i,p_i) -2\sum_{i=1}^n \left( 1-\frac{1}{p_i}\right)\right),
\]
where the minus sign arises since we are considering the orientation reversal $-X_S$ of $X_S$. It follows that the odd integers $a_1,\dots,a_N$ satisfy \begin{align*}
    &-\sum_{i=1}^N a_i^2 =c_1(\mathfrak{s}')^2+c_1(\mathfrak{s}_{K_{X_S}})^2 \\
    & = \sum_{i=1}^n \left(  -b_2(X(p_i,p_i-q_i))-2\left(1-\frac{1}{p_i}\right)+12s(q_i,p_i)\right)-\left(9+12\sum_{i=1}^n s(q_i,p_i) -2\sum_{i=1}^n \left( 1-\frac{1}{p_i}\right)\right) \\
    & = -9-\sum_{i=1}^n b_2(X(p_i,p_i-q_i))=-8-N,
\end{align*}
i.e., \[
-\mathbf{K}^2=\sum_{i=1}^N a_i^2 =N+8.
\]
Since each $a_i$ is an odd integer, this equality holds if and only if exactly one $a_i$ is $\pm 3$ and the remaining $a_i$'s are all $\pm 1$. That is, $\mathbf{K}$ is of the form \[
\mathbf{K}=\epsilon_1e_1+\cdots+3\epsilon_{i_0}e_{i_0}+\cdots+\epsilon_Ne_N \quad (1\leq i_0\leq N,\, \epsilon_i\in \{\pm 1\}). 
\]

Now consider the restriction of the embedding \eqref{eq:embedding_iota} to $Q_{X'}=\bigoplus_{i=1}^n Q_{X(p_i,p_i-q_i)}$, which is a codimension-one lattice embedding into $-\mathbb{Z}^N$: \[
\iota\colon \bigoplus_{i=1}^n Q_{X(p_i,p_i-q_i)} \longrightarrow -\mathbb{Z}^N.
\]
Let us verify properties \ref{item(a)_main} and \ref{item(b)_main} of Theorem \ref{thm:main}. For \ref{item(a)_main}, recall from our construction that $\mathbf{K}=\iota_{\mathbb{Q}}(c_1(\mathfrak{s}')+K_{X_S})$ and that $c_1(\mathfrak{s'})=\sum_{i=1}^n c_1(\mathfrak{s}_{K_{p_i,p_i-q_i}})=\sum_{i=1}^n K_{p_i,p_i-q_i}$. Therefore, for a standard basis vector $v\in \bigoplus_{i=1}^n Q_{X(p_i,p_i-q_i)}$, the condition \[
\mathbf{K}\cdot \iota(v)=\begin{cases}
    -v^2 & \text{if $v$ is right-end} \\
    -v^2-2 & \text{otherwise}
\end{cases}
\]
immediately follows from the definition \eqref{eq:def_of_K_p,q} of $K_{p,q}$. For \ref{item(b)_main}, recall from Proposition \ref{prop:spin^c_gluing} that our embedding \eqref{eq:embedding_iota} satisfies $\iota(Q_{-X_S})=\iota(Q_{X'})^\perp$. Since $Q_{-X_S}\cong \langle -p_1\cdots p_n\rangle$ by Lemma \ref{lem:topology_of_QHCP2}, we see that the orthogonal complement $\iota(Q_{X'})^\perp$ is generated by a vector of square $-p_1\cdots p_n$. This completes the proof of Theorem \ref{thm:main}. 

\begin{remark}
The statement of Theorem \ref{thm:main} is formulated in a convenient way for the proof of Theorem \ref{thm:aMY}. However, when applying this theorem to specific examples, it is more practical to use the following tractable form.

For each $i=1,\dots,n$, suppose that the Hirzebruch-Jung continued fraction of $p_i/(p_i-q_i)$ is given by \[
\frac{p_i}{p_i-q_i}=\left[a_{i,1},\dots, a_{i,s_i} \right],
\]
and let $v_{i,1},\dots,v_{i,s_i}\in H_2(X(p_i,p_i-q_i);\mathbb{Z})$ denote the corresponding standard basis vectors. We have seen in Remark \ref{rmk:K_p,q_in_terms_of_v_i} that \[
K_{p_i,p_i-q_i}=\sum_{j=1}^{s_i} y_j(p_i,p_i-q_i)\cdot v_{i,j} \in H_2(X(p_i,p_i-q_i);\mathbb{Z})\otimes \mathbb{Q},
\]
where the coefficients $y_j(p_i,p_i-q_i)$ are determined by \eqref{eq:y_i}. It follows that \[
c_1(\mathfrak{s}')=\sum_{i=1}^n c_1(\mathfrak{s}_{K_{p_i,p_i-q_i}})=\sum_{i=1}^n K_{p_i,p_i-q_i}=\sum_{i=1}^n \sum_{j=1}^{s_i} y_j(p_i,p_i-q_i)\cdot v_{i,j} \in H_2(X';\mathbb{Z})\otimes \mathbb{Q},
\]
noting that $H_2(X';\mathbb{Z}) =\bigoplus_{i=1}^n H_2(X(p_i,p_i-q_i);\mathbb{Z})$. 

Next, let $v\in H_2(-X_S;\mathbb{Z})\cong \mathbb{Z}$ be a generator. By Lemma \ref{lem:topology_of_QHCP2}, it satisfies $v^2=-p_1\cdots p_n$. Let $v^*\in H^2(-X_S;\mathbb{Z})$ denote its dual generator. Viewing $H^2(-X_S;\mathbb{Z})\subset H_2(-X_S;\mathbb{Z})\otimes \mathbb{Q}$ as in \eqref{eq:second_cohomology_as_dual_lattice}, we have $v^*=-1/(p_1\cdots p_n) \cdot v$ (since $v^*\cdot v=1$). Moreover, the proof of Proposition \ref{prop:K_X_S} shows that $K_{X_S}= \pm \sqrt{p_1\cdots p_n K_S^2}\cdot v^*$, and by changing the sign of $v$ if necessary, we may assume that \[
K_{X_S}=  \sqrt{p_1\cdots p_n K_S^2}\cdot v^* = -\frac{ \sqrt{p_1\cdots p_n K_S^2}}{p_1\cdots p_n}\cdot v \in H_2(-X_S;\mathbb{Z})\otimes \mathbb{Q}.
\]
Note that we can view \eqref{eq:embedding_iota} as a lattice embedding \[
\iota\colon \bigoplus_{i=1}^n Q_{X(p_i,p_i-q_i)}\oplus \langle -p_1\cdots p_n\rangle \longrightarrow -\mathbb{Z}^N.
\]
In summary, we have shown that there must exist such a lattice embedding $\iota$ such that \[
\sum_{i=1}^n \sum_{j=1}^{s_i} y_j(p_i,p_i-q_i)\cdot \iota (v_{i,j} ) -\frac{ \sqrt{p_1\cdots p_n K_S^2}}{p_1\cdots p_n}\cdot  \iota(v)  =\epsilon_1e_1+\cdots+3\epsilon_{i_0}e_{i_0}+\cdots+\epsilon_Ne_N
\]
with $1\leq i_0\leq N$ and $\epsilon_i\in \{\pm 1\}$. By applying an automorphism of $-\mathbb{Z}^N$, we can actually assume that \[
\sum_{i=1}^n \sum_{j=1}^{s_i} y_j(p_i,p_i-q_i)\cdot \iota (v_{i,j} ) -\frac{ \sqrt{p_1\cdots p_n K_S^2}}{p_1\cdots p_n}\cdot  \iota(v)  =3e_1+e_2+\cdots+e_N.
\]
We summarize this discussion in the following theorem. 
\end{remark}

\begin{theorem}
     Let $S$ be a rational homology projective plane with $H_1(S^0;\mathbb{Z})=0$, having $n$ cyclic quotient singularities of types 
\[
    \tfrac{1}{p_1}(1,q_1),\dots,\tfrac{1}{p_n}(1,q_n).
\]
Assume that, for each $i=1,\dots,n$, either $p_i$ is odd or $\sigma(q_i,p_i,+1)\neq \sigma (q_i,p_i,-1)$. Then there exists a full-rank lattice embedding \[
\iota\colon \bigoplus_{i=1}^n Q_{X(p_i,p_i-q_i)}\oplus \langle -p_1\cdots p_n\rangle \longrightarrow -\mathbb{Z}^N,
\]
where $N=\sum_{i=1}^n b_2(X(p_i,p_i-q_i))+1$, such that     \[
\sum_{i=1}^n \sum_{j=1}^{s_i} y_j(p_i,p_i-q_i)\cdot \iota (v_{i,j} ) -\frac{ \sqrt{p_1\cdots p_n K_S^2}}{p_1\cdots p_n}\cdot  \iota(v)  =3e_1+e_2+\cdots +e_N.
\] 
Here, $\{v_{i,1},\dots ,v_{i,s_i}\}$ is the standard basis for $Q_{X(p_i,p_i-q_i)}$ with $s_i=b_2(X(p_i,p_i-q_i))$, and $v$ is a generator of the lattice $\langle -p_1\cdots p_n\rangle $. 
\end{theorem}

Note that the image $\iota(v)$ must be a primitive vector in $-\mathbb{Z}^N$. The following example serves as a sanity check for this theorem.

\begin{example}[Weighted projective planes]\label{ex:weighted_projective_planes} For pairwise relatively prime integers $1\leq a\leq b\leq c$, the \emph{weighted projective plane} $\mathbb{CP}(a,b,c)$ is defined as the quotient of $\mathbb{C}^3-\{0\}$ under the action of $\mathbb{C}^*=\mathbb{C}-\{0\}$ given by \[
t\cdot (z_1,z_2,z_3) =\left(t^a z_1, t^bz_2,t^cz_3 \right) \quad (t\in \mathbb{C}^*, \, (z_1,z_2,z_3)\in \mathbb{C}^3-\{0\}). 
\]
The open subset $\{[z_1,z_2,z_3]\in \mathbb{CP}(a,b,c):z_1\neq 0\}$ is isomorphic to the quotient space \[
\frac{\mathbb{C}^2}{\mathbb{Z}_a} =\frac{\mathbb{C}^2}{(z,w)\sim (e^{2\pi i b/a}z, e^{2\pi i c/a}w)}
\]
via the correspondence $[1,z,w]\mapsto [z,w]$. Thus, if $a>1$, then the point $[1,0,0]\in \mathbb{CP}(a,b,c)$ is a cyclic quotient singularity of type $\tfrac{1}{a}(1,b'c)$, where $b'$ is the multiplicative inverse of $b$ modulo $a$, and $b'c$ is evaluated modulo $a$. (If $a=1$, then $[1,0,0]$ is not a singular point.) Similarly, the points $[0,1,0]$ and $[0,0,1]$ are cyclic quotient singularities (assuming $b,c>1$) of types $\tfrac{1}{b}(1,c'a)$ and $\tfrac{1}{c}(1,a'b)$, respectively, where $cc'\equiv 1 \pmod{b}$ and $aa' \equiv 1 \pmod{c}$. Thus, $\mathbb{CP}(a,b,c)$ has at most three cyclic quotient singularities.

Observe that $\mathbb{CP}(a,b,c)$ can also be defined as the quotient of $S^5\subset \mathbb{C}^3$ under the $S^1$-action given by the same formula as above. This action is free on the complement of the orbits over the singular points, which implies that there is an $S^1$-bundle over the smooth locus of $\mathbb{CP}(a,b,c)$ whose total space is the complement of the union of (at most three) circles in $S^5$. Using the corresponding homotopy exact sequence, we see that the smooth locus of $\mathbb{CP}(a,b,c)$ is simply-connected. Then the homology groups of the smooth locus can be computed using the corresponding Leray-Serre spectral sequence, from which it easily follows that $\mathbb{CP}(a,b,c)$ is a $\mathbb{Q}$-homology $\mathbb{CP}^2$. It is also well-known that  \[
K_{\mathbb{CP}(a,b,c)}^2 = \frac{(a+b+c)^2}{abc}.
\] 

As an example, consider $S=\mathbb{CP}(2,3,5)$, which has three cyclic quotient singularities of types $\tfrac{1}{2}(1,1)$, $\tfrac{1}{3}(1,1)$, and $\tfrac{1}{5}(1,4)$. Let us verify that there is a lattice embedding as in Theorem \ref{thm:main}. Consider the standard basis $\{v_{1,1},v_{2,1},v_{2,2},v_{3,1}\}$ of $Q_{X(2,1)}\oplus Q_{X(3,2)}\oplus Q_{X(5,1)}$, and let $v$ be a generator of the lattice $\langle -2\cdot 3\cdot 5\rangle = \langle -30\rangle$. We have \[
y_1(2,1) = -1, \quad y_1(3,2)=-\frac{2}{3}, \quad y_2(3,2)=-\frac{4}{3}, \quad y_1(5,1)=-1, \quad K_S^2=\frac{10^2}{30}.
\]
Therefore, we need to show that there is a lattice embedding \[
\iota\colon Q_{X(2,1)}\oplus Q_{X(3,2)}\oplus Q_{X(5,1)} \oplus  \langle -30\rangle \longrightarrow -\mathbb{Z}^5
\]
such that \[
-\iota(v_{1,1})-\frac{2}{3}\iota(v_{2,1})-\frac{4}{3}\iota(v_{2,2})-\iota(v_{3,1})-\frac{10}{30}\iota(v)=3e_1+e_2+e_3+e_4+e_5.
\]
Such a lattice embedding $\iota$ is depicted in Figure \ref{fig:lattice_embedding_example}.  
\end{example} 

\begin{figure}[!th]
\centering
\begin{tikzpicture}[scale=1]
\draw (0,0) node[circle, fill, inner sep=1.2pt, black]{};
\draw (-1,-1.4) node[circle, fill, inner sep=1.2pt, black]{};
\draw (1,-1.4) node[circle, fill, inner sep=1.2pt, black]{};
\draw (0,-2.8) node[circle, fill, inner sep=1.2pt, black]{};
\draw (6,-1.4) node{$*$};

\draw (0,0) node[above]{$-2$};
\draw (-1,-1.4) node[above]{$-2$};
\draw (1,-1.4) node[above]{$-2$};
\draw (0,-2.8) node[above]{$-5$};
\draw (6,-1.4) node[above]{$-30$};

\draw (-1,-1.4)--(1,-1.4) ;

\draw (0,0) node[below]{$\iota(v_{1,1})=-e_1+e_4$};
\draw (-1.7,-1.4) node[below]{$\iota(v_{2,1})=-e_2+e_3$};
\draw (1.7,-1.4) node[below]{$\iota(v_{2,2})=-e_3-e_5$};
\draw (0,-2.8) node[below]{$\iota(v_{3,1})=-e_1-e_2-e_3-e_4+e_5$};
\draw (6.5,-1.4) node[below]{$\iota(v)=-3e_1+2e_2+2e_3-3e_4-2e_5$};
\end{tikzpicture}
\caption{A lattice embedding $\iota\colon Q_{X(2,1)}\oplus Q_{X(3,2)}\oplus Q_{X(5,1)} \oplus  \langle -30\rangle \longrightarrow -\mathbb{Z}^5$}
\label{fig:lattice_embedding_example}
\end{figure}

\section{Proof of Theorem \ref{thm:aMY}}\label{sec:main_result_proof}
In this section, we prove the main result of the paper, Theorem \ref{thm:aMY}, by applying the lattice-theoretic constraint established in Theorem \ref{thm:main}. Suppose, for contradiction, that there exists a rational homology projective plane $S$ with $H_1(S^0;\mathbb{Z})=0$ and ample canonical divisor $K_S$, having exactly four cyclic quotient singularities of types 
\[
    \tfrac{1}{2}(1,1), \quad \tfrac{1}{3}(1,1), \quad \tfrac{1}{5}(1,q_3), \quad \tfrac{1}{n}(1,q_4),
\]
where $\gcd(n,30)=1$ and $\gcd(5,q_3)=\gcd(n,q_4)=1$. We may assume that $q_3\in \{1,3,4\}$, since the cyclic quotient singularities of types $\tfrac{1}{5}(1,2)$ and $\tfrac{1}{5}(1,3)$ are isomorphic. Theorem \ref{thm:main} then implies that there exists a codimension-one lattice embedding 
\begin{equation}\label{eq:main_embedding}
    \iota\colon Q_{X(2,1)}\oplus Q_{X(3,2)}\oplus Q_{X(5,5-q_3)}\oplus Q_{X(n,n-q_4)} \longrightarrow -\mathbb{Z}^N,
\end{equation}
where 
\[
    N=b_2(X(2,1))+b_2(X(3,2))+b_2(X(5,5-q_3))+b_2(X(n,n-q_4))+1,
\] 
satisfying the following conditions: 
\begin{enumerate}[label=\textup{(\alph*)}]
    \item There exists a vector $\mathbf{K}\in -\mathbb{Z}^N$ of the form \begin{equation}\label{eq:vector_K}
        \mathbf{K}=\epsilon_1e_1+\cdots+3\epsilon_{i_0}e_{i_0}+\cdots +\epsilon_Ne_N \quad (1\leq i_0\leq N, \, \epsilon_i\in \{\pm 1\})
    \end{equation}
    such that for any standard basis vector $v$ of  $Q_{X(2,1)}\oplus Q_{X(3,2)}\oplus Q_{X(5,5-q_3)}\oplus Q_{X(n,n-q_4)}$,  \begin{equation}\label{eq:property_of_K}
         \mathbf{K}\cdot \iota(v)=\begin{cases}
        -v^2 & \text{if $v$ is right-end}, \\
        -v^2-2 & \text{otherwise}.
    \end{cases}
    \end{equation}
    \item \label{item:orthogonal_complement} The orthogonal complement of the image of $\iota$ is generated by a vector in $-\mathbb{Z}^N$ of square $-2\cdot 3\cdot 5\cdot n=-30n$. 
\end{enumerate} 
Throughout this section, $\iota$ will denote a lattice embedding \eqref{eq:main_embedding} satisfying these conditions. Our goal is to derive a contradiction, showing that such an embedding cannot exist; the remainder of the proof is purely combinatorial.

\subsection{Setup and basic identities}
Note that we have the following continued fraction expansions: \[
\frac{2}{1}=[2], \quad \frac{3}{2}=[2,2], \quad \frac{5}{5-q_3}=\begin{cases}
  [2,2,2,2] & \text{if } q_3=1, \\
  [3,2] & \text{if } q_3=3, \\
  [5] & \text{if } q_3=4.
\end{cases}
\]
Let $\{u_1\}$ and $\{u_2,u_3\}$ denote the standard bases for $Q_{X(2,1)}$ and $Q_{X(3,2)}$, respectively. For $Q_{X(5,5-q_3)}$, we denote its standard basis as follows: 
\[
\begin{cases}
    \{u_4,u_5,u_6,u_7\} & \text{if } q_3=1, \\
    \{u_4,u_5\} & \text{if } q_3=3, \\
    \{u_4\} & \text{if } q_3=4.
\end{cases}
\]
We begin by classifying the possible images of these vectors in $-\mathbb{Z}^N$, which will guide the case-by-case analysis later in this section.

\begin{lemma}\label{lem:6.1}
   Up to an automorphism of $-\mathbb{Z}^N$, we have 
   \[
        \iota(u_1)=e_1-e_2,\quad  \iota(u_2)=e_3-e_4,\quad  \iota(u_3)=e_4-e_5,
   \] and 
   \begin{align}
       &(\iota(u_4),\iota(u_5),\iota(u_6),\iota(u_7))=(e_6-e_7,e_7-e_8,e_8-e_9,e_9-e_{10}) & \text{if $q_3=1$}, \nonumber \\
       & (\iota(u_4),\iota(u_5))=(e_1+e_2-e_6,e_6-e_7) \quad \text{or} \quad (e_6+e_7-e_8,e_8-e_9) & \text{if } q_3=3, \nonumber \\
       & \iota(u_4)=\begin{cases}
            e_1+e_2+e_3+e_4+e_5 \\
            e_1+e_2-e_3-e_4-e_5 \\
            e_3+e_4+e_5+e_6+e_7 \\
            -e_3-e_4-e_5+e_6+e_7 \\
            2e_6+e_7 \\
            e_1+e_2+e_6+e_7+e_8 \\
            e_6+e_7+e_8+e_9+e_{10} 
        \end{cases} & \text{if } q_3=4.  \label{eq:iota(u_4)_for_q_3=4}
   \end{align} 
\end{lemma}
\begin{proof}
    The possible images of the standard basis vectors are obtained by a finite enumeration of vectors with prescribed squares and pairings, in the same spirit as the proof of \cite[Theorems 3.1 and 3.2]{Jo-Park-Park-2025-01}. This elementary computation gives precisely the list above.
\end{proof}

Next, suppose that the continued fraction expansion of $n/(n-q_4)$ is given by \[
\frac{n}{n-q_4}=[a_1,\dots,a_s].
\]
We write $\{v_1,\dots,v_s\}$ for the standard basis of $Q_{X(n,n-q_4)}$. For a lattice embedding $\iota$ as in \eqref{eq:main_embedding}, note that \begin{equation}\label{eq:N_and_s}
     N=\begin{cases}
          s+8 & \text{if } q_3=1, \\
          s+6 & \text{if } q_3=3, \\
          s+5 & \text{if } q_3=4.
      \end{cases}
\end{equation}
Throughout this section, we use the following notation. 
\begin{itemize}
    \item We write $\mathbf{K}=K_1e_1+\cdots+K_Ne_N\in -\mathbb{Z}^N$ for the vector \eqref{eq:vector_K} satisfying \eqref{eq:property_of_K}. Thus, $K_1,\dots,K_N\in \{\pm1,\pm3\}$, and exactly one of them is equal to $\pm 3$. 
    \item We define the sum vectors $\mathbf{S}_1, \dots, \mathbf{S}_4\in -\mathbb{Z}^N$ corresponding to the four plumbing summands as follows: $\mathbf{S}_1=\iota(u_1)$, $\mathbf{S}_2=\iota(u_2)+\iota(u_3)$, \[
    \mathbf{S}_3=\begin{cases}
        \iota(u_4)+\iota(u_5)+\iota(u_6)+\iota(u_7) & \text{if } q_3=1, \\
        \iota(u_4)+\iota(u_5) & \text{if } q_3=3, \\
        \iota(u_4) & \text{if } q_3=4,
    \end{cases}
    \]
    and $\mathbf{S}_4=\iota(v_1)+\cdots+\iota(v_s)$. Note that the vectors $\mathbf{S}_1,\mathbf{S}_2,\mathbf{S}_3,\mathbf{S}_4$ are mutually orthogonal. We also define the total sum vector 
    \[
        \mathbf{T}=T_1e_1+\cdots +T_Ne_N:=\mathbf{S}_1+\mathbf{S}_2+\mathbf{S}_3+\mathbf{S}_4.
    \]
    \item We define the partial sum vectors for the fourth plumbing summand $Q_{X(n,n-q_4)}$ by
    \[
        \mathbf{P}_k= \iota(v_1)+\cdots +\iota(v_k) \quad (1\leq k\leq s).
    \]
    Note that $\mathbf{P}_s=\mathbf{S}_4$.
    \item We define 
    \[
        \mathbb{I}_4 = \sum_{i=1}^s (a_i-3) = \sum_{i=1}^s a_i -3s.
    \]
    Similarly, we put $\mathbb{I}_1=-1$, $\mathbb{I}_2=-2$, and 
    \[
        \mathbb{I}_3=
        \begin{cases}
            -4 & \text{if } q_3=1, \\
            -1 & \text{if } q_3=3, \\
            2 & \text{if } q_3=4.
        \end{cases}
    \]
    Finally, we put $\mathbb{I}_{\operatorname{total}}=\mathbb{I}_1+\mathbb{I}_2+\mathbb{I}_3+\mathbb{I}_4$. Note that \begin{equation}\label{eq:I_total}
         \mathbb{I}_{\operatorname{total}}=\begin{cases}
        \mathbb{I}_4-7 & \text{if $q_3=1$}, \\
        \mathbb{I}_4-4 & \text{if $q_3=3$}, \\
        \mathbb{I}_4-1 & \text{if $q_3=4$}.
    \end{cases}
    \end{equation}
\end{itemize}
With this notation established, we record several fundamental algebraic identities satisfied by these vectors.

\begin{lemma}\label{lem:6.2}
    We have the following identities: 
    \begin{enumerate}[label=\textup{(\alph*)}]
        \item $\mathbf{S}_i^2+\mathbf{S}_i\cdot \mathbf{K}=0$ for $i=1,2,3,4$. 
        \item $\mathbf{T}^2=-(\mathbb{I}_{\operatorname{total}}+N+7)$. 
        \item $\mathbf{T}^2+\mathbf{T}\cdot \mathbf{K}=0$. 
        \item $\mathbf{P}_k^2+\mathbf{P}_k\cdot \mathbf{K}=-2$ for $k=1,\dots,s-1$. 
        \item $(\mathbf{P}_\ell-\mathbf{P}_k)^2+(\mathbf{P}_\ell-\mathbf{P}_k)\cdot \mathbf{K}=-2$ for $1\leq k<\ell <s$.
        \item $\mathbf{P}_k\cdot (\mathbf{S}_4+\mathbf{K})=-1$ for $k=1,\dots,s-1$. 
    \end{enumerate}
\end{lemma}

\begin{proof}
    We first prove part \textup{(a)} in a uniform way. For each $i=1,2,3,4$, write the continued fraction expansion corresponding to the $i$-th plumbing summand as $[a_{i,1},\dots,a_{i,s_i}]$, and let $\{v_{i,1},\dots,v_{i,s_i}\}$ denote its standard basis. Thus
    \begin{align*}
        \mathbf{S}_i^2&= (v_{i,1}+\cdots+v_{i,s_i})^2 = \sum_{j=1}^{s_i} v_{i,j}^2+2\sum_{j=1}^{s_i-1} v_{i,j}\cdot v_{i,j+1} \\
        &= -\sum_{j=1}^{s_i} a_{i,j}+2(s_i-1) = -\sum_{j=1}^{s_i}(a_{i,j}-3)-s_i-2 = -(\mathbb{I}_i+s_i+2).
    \end{align*}
    On the other hand, by \eqref{eq:property_of_K},
    \begin{align*}
        \mathbf{S}_i\cdot \mathbf{K} &= \sum_{j=1}^{s_i} \mathbf{K}\cdot \iota(v_{i, j}) = \sum_{j=1}^{s_i-1} (-v_{i,j}^2-2)-v_{s_i}^2  \\
              &  =\sum_{j=1}^{s_i} a_{i, j} -2({s_i}-1)=\sum_{j=1}^{s_i} (a_{i,j}-3)+{s_i}+2 = \mathbb{I}_i+{s_i}+2.
    \end{align*}
    Therefore $\mathbf{S}_i^2+\mathbf{S}_i\cdot\mathbf{K}=0$ for each $i=1,2,3,4$, proving part \textup{(a)}.
    
    Next, since the vectors $\mathbf{S}_1,\dots,\mathbf{S}_4$ are mutually orthogonal, the above computations give
    \[
        \mathbf{T}^2=\mathbf{S}_1^2+\mathbf{S}_2^2+\mathbf{S}_3^2+\mathbf{S}_4^2=
        \begin{cases}
           -(\mathbb{I}_4+s+8) & \text{if $q_3=1$}, \\
           -(\mathbb{I}_4+s+9) & \text{if $q_3=3$}, \\
           -(\mathbb{I}_4+s+11) & \text{if $q_3=4$}.
        \end{cases}
    \]
    Part \textup{(b)} now follows from \eqref{eq:N_and_s} and \eqref{eq:I_total}. Moreover, using again the mutual orthogonality of the $\mathbf{S}_i$'s and part \textup{(a)}, we obtain part \textup{(c)}.
      
    Next, note that for $0\leq r<\ell<s$, we have
    \begin{align*}
        \iota(v_{r+1}+\cdots+v_\ell)^2+\iota(v_{r+1}+\cdots+v_\ell)\cdot \mathbf{K} & = (v_{r+1}+\cdots+v_\ell)^2 +\sum_{i=r+1}^\ell \mathbf{K}\cdot \iota(v_i) \\
        & = \sum_{i=r+1}^\ell v_i^2 +2(\ell-r-1)+\sum_{i=r+1}^\ell (-v_i^2-2)= -2,
    \end{align*}
    Taking $r=0$ proves part \textup{(d)}, and taking $1\leq r<\ell<s$ proves part \textup{(e)}. Finally, for part \textup{(f)}, note that
    \[
        \mathbf{P}_k\cdot \mathbf{S}_4=\iota(v_1+\cdots +v_k)\cdot \iota(v_1+\cdots +v_s) = \mathbf{P}_k^2+1,
    \]
    Therefore, by part \textup{(d)},
    \[
        \mathbf{P}_k\cdot(\mathbf{S}_4+\mathbf{K}) = \mathbf{P}_k^2+\mathbf{P}_k\cdot\mathbf{K}+1 = -1,
    \]
    which proves part \textup{(f)}.
\end{proof}

Next, we apply the orbifold BMY inequality (Theorem \ref{thm:oBMY}) to obtain a constraint on the quantity $\mathbb{I}_{\operatorname{total}}$.
\begin{lemma}\label{lem:6.3}
    We have
    \[
        \mathbb{I}_{\operatorname{total}}=
        \begin{cases}
            -2 \text{ or } -1 & \text{if } q_3=1 \text{ or } 3, \\
            -1 \text{ or } 0 & \text{if } q_3=4.
        \end{cases}
    \]
\end{lemma}
\begin{proof}
    By \eqref{eq:I_total}, it suffices to show that \[
    \mathbb{I}_4=\begin{cases}
        5\text{ or }6  & \text{if } q_3=1, \\
        2\text{ or }3 & \text{if } q_3=3, \\
        0\text{ or }1 & \text{if } q_3=4.
    \end{cases}
    \]
        Suppose that the continued fraction expansion of $n/q_4$ is given by \[
        \frac{n}{q_4}=[b_1,\dots,b_\ell]. 
        \]
    Note that we have the identity $\mathbb{I}_4+\sum_{j=1}^\ell (b_j-3)=-2$ \cite[Lemma 2.6]{Lisca-2007}, so the problem reduces to showing that \[
    \sum_{j=1}^\ell (b_j-3)=\begin{cases}
        -7\text{ or }{-8}  & \text{if $q_3=1$}, \\
        -4\text{ or }{-5} & \text{if $q_3=3$}, \\
        -2\text{ or }{-3} & \text{if $q_3=4$}.
    \end{cases}
    \]
    Since our surface $S$ has four singularities of types $\tfrac{1}{2}(1,1)$, $\tfrac{1}{3}(1,1)$, $\tfrac{1}{5}(1,q_3)$, and $\tfrac{1}{n}(1,q_4)$, the formula \eqref{eq:K_square} gives that \[
    K_S^2 =\begin{cases}
        \displaystyle\sum_{j=1}^\ell (b_j-3)+\frac{q_4+q_4'+2}{n}+\frac{92}{15} & \text{if $q_3=1$}, \\
         \displaystyle\sum_{j=1}^\ell (b_j-3)+\frac{q_4+q_4'+2}{n}+\frac{56}{15} & \text{if $q_3=3$}, \\
          \displaystyle\sum_{j=1}^\ell (b_j-3)+\frac{q_4+q_4'+2}{n}+\frac{4}{3} & \text{if $q_3=4$},
    \end{cases}
    \]
    where $q_4'$ denotes the unique integer satisfying $0<q_4'<n$ and $q_4q_4'\equiv 1 \pmod{n}$. On the other hand, the orbifold Euler characteristic of $S$ is given by \[
    e_{\operatorname{orb}}(S)=3-\left(1-\frac{1}{2}\right)-\left(1-\frac{1}{3}\right)-\left(1-\frac{1}{5}\right)-\left(1-\frac{1}{n}\right)=\frac{1}{30}+\frac{1}{n}.
    \]
    Since $K_S$ is ample, we have $K_S^2>0$. Therefore, the orbifold BMY inequality implies that 
    \[
    0<K_S^2\leq 3e_{\operatorname{orb}}(S)=\frac{1}{10}+\frac{3}{n}.
    \]
    Finally, observe that 
    \[
        \frac{4}{n}\leq \frac{q_4+q_4'+2}{n} \leq 2
    \]
    since $1\leq q_4,q_4'\leq n-1$. Combining this with the above bounds on $K_S^2$ and using the fact that $\sum_{j=1}^\ell(b_j-3)$ is an integer, we obtain the desired possibilities for $\sum_{j=1}^\ell(b_j-3)$, and hence for $\mathbb{I}_4$.
\end{proof}

The next lemma shows that $\mathbb{I}_{\operatorname{total}}$ determines the possible forms of the total sum vector $\mathbf{T}$. Note that the standard basis vectors $e_i$ appearing in the lemma are not necessarily the same as those in Lemma \ref{lem:6.1}. We explain this point immediately after the lemma.
\begin{lemma}\label{lem:6.4}
    After applying an automorphism of $-\mathbb{Z}^N$, we may assume that 
    \[
        \mathbf{K}=3e_1+e_2+\cdots+e_N,
    \]
    and that the total sum vector $\mathbf{T}$ is one of the following:
    \[
    \mathbf{T} = \begin{cases}
        \begin{array}{l}
            -3e_1-e_2-\cdots-e_{N-1} 
        \end{array}& \text{if } \mathbb{I}_{\operatorname{total}}=0,\\[1ex]
        \begin{array}{l}
            -3e_1-e_2-\cdots-e_{N-2} \\
            \text{or } -2e_1-2e_2-e_3-\cdots-e_N
        \end{array}& \text{if } \mathbb{I}_{\operatorname{total}}=-1, \\[3ex]
        \begin{array}{l}
            -3e_1-e_2-\cdots-e_{N-3} \\
            \text{or } -2e_1-2e_2-e_3-\cdots-e_{N-1}
        \end{array} & \text{if } \mathbb{I}_{\operatorname{total}}=-2.
    \end{cases}
    \]
\end{lemma}
\begin{proof}
    Since $\mathbf{K}$ is of the form \eqref{eq:vector_K}, after applying an automorphism of $-\mathbb{Z}^N$, we may assume $\mathbf{K}=3e_1+e_2+\cdots+e_N$. By Lemma \ref{lem:6.2}(b), we have 
    \begin{equation}\label{eq:T_square}
        \sum_{i=1}^N T_i^2 = -\mathbf{T}^2= \mathbb{I}_{\operatorname{total}}+N+7=
        \begin{cases}
        N+7 & \text{if $\mathbb{I}_{\operatorname{total}}=0$}, \\
        N+6 & \text{if $\mathbb{I}_{\operatorname{total}}=-1$}, \\
        N+5 & \text{if $\mathbb{I}_{\operatorname{total}}=-2$}.
        \end{cases}
    \end{equation}
    On the other hand, Lemma \ref{lem:6.2}\textup{(c)} gives
    \begin{align*}
       0=-(\mathbf{T}^2+\mathbf{T}\cdot \mathbf{K})=\sum_{i=1}^N T_i^2 +3T_1+\sum_{i=2}^N T_i=(T_1^2+3T_1)+\sum_{i=2}^N (T_i^2+T_i).
    \end{align*}
    Since $T_i^2+T_i\geq 0$ for every integer $T_i$, we must have $T_1^2+3T_1\leq 0$, and hence $T_1\in\{0,-1,-2,-3\}$. 
    
    If $T_1=0$, then $T_i\in \{0,-1\}$ for all $i\geq 2$. Hence, $\sum_{i=1}^N T_i^2 \leq N-1$, contradicting \eqref{eq:T_square}. If $T_1=-1$, then $\sum_{i=2}^N (T_i^2+T_i) = 2$. Thus exactly one $T_k$ with $k \geq 2$ is either $1$ or $-2$, while $T_i \in \{0,-1\}$ for all other $i \geq 2$. It follows that $\sum_{i=1}^N T_i^2 \leq  N+3$ again contradicting \eqref{eq:T_square}.
   
    Suppose next that $T_1=-2$. Then exactly one $T_k$ with $k \geq 2$ is either $1$ or $-2$, while $T_i \in \{0,-1\}$ for all other $i \geq 2$. By applying an automorphism swapping $e_2$ and $e_k$ which preserves $\mathbf{K}$, we may assume that $T_2\in \{-2,1\}$. If $T_2=1$, then $\sum_{i=1}^N T_i^2\leq N+3$, which is impossible by \eqref{eq:T_square}. Hence $T_2=-2$. In this case, $\sum_{i=1}^N T_i^2\leq N+6$, so the case $\mathbb{I}_{\operatorname{total}}=0$ cannot occur. If $\mathbb{I}_{\operatorname{total}}=-1$, then $T_3=\cdots=T_N=-1$, and hence $\mathbf{T}=-2e_1-2e_2-e_3-\cdots-e_N$. If $\mathbb{I}_{\operatorname{total}}=-2$, then exactly one $T_j$ with $j\geq 3$ is $0$, and $T_i=-1$ for all other $i\geq 3$. By swapping $e_j$ and $e_N$, we may assume that $T_N=0$, and hence $\mathbf{T}=-2e_1-2e_2-e_3-\cdots-e_{N-1}$. 

    Finally, suppose that $T_1=-3$. Then $T_i\in \{0,-1\}$ for all $i\geq 2$. If $\mathbb{I}_{\operatorname{total}}=0$, then $\sum_{i=1}^N T_i^2=N+7$ implies that exactly one $T_k$ with $k\geq 2$ is $0$ and $T_i=-1$ for all other $i\geq 2$. We may assume $k=N$, and hence $\mathbf{T}=-3e_1-e_2-\cdots-e_{N-1}$. If $\mathbb{I}_{\operatorname{total}}=-1$, then exactly two $T_j, T_k$ with $2\leq j<k$ are $0$ and $T_i=-1$ for all other $i\geq 2$. We may assume $(j,k)=(N-1,N)$, and hence $\mathbf{T}=-3e_1-e_2-\cdots-e_{N-2}$. If $\mathbb{I}_{\operatorname{total}}=-2$, then a similar argument gives $\mathbf{T}=-3e_1-e_2-\cdots-e_{N-3}$. 
\end{proof}

Once the images of the initial plumbing vectors have been normalized as in Lemma \ref{lem:6.1}, we cannot in general also put $\mathbf{K}$ in the normal form $\mathbf{K}=3e_1+e_2+\cdots+e_N$. Therefore, we will use the following reformulation of Lemma \ref{lem:6.4} in the case-by-case arguments.

\begin{remark}\label{rmk:6.5}
Observe that any automorphism of $-\mathbb{Z}^N$ is a composition of permutations of the basis vectors and sign changes $e_i\mapsto -e_i$. Consequently, Lemma \ref{lem:6.4} can be restated as follows. Suppose that the vector $\mathbf{K}$ satisfying \eqref{eq:property_of_K} is given by 
\[
    \mathbf{K}=\epsilon_1e_1+\cdots+3\epsilon_{i_0}e_{i_0}+\cdots +\epsilon_Ne_N \quad (1\leq i_0\leq N, \, \epsilon_i\in \{\pm 1\}).
\] 
Then the total sum vector $\mathbf{T}=\sum_{i=1}^NT_ie_i$ satisfies one of the following conditions: 
\begin{itemize}
    \item If $\mathbb{I}_{\operatorname{total}}=0$, then $T_{i_0}=-3\epsilon_{i_0}$, exactly one index $k\neq i_0$ satisfies $T_k=0$, and $T_i=-\epsilon_i$ for all remaining indices $i$.

    \item If $\mathbb{I}_{\operatorname{total}}=-1$, then $T_{i_0}\in\{-2\epsilon_{i_0},-3\epsilon_{i_0}\}$. 
    \begin{itemize} 
        \item if $T_{i_0}=-2\epsilon_{i_0}$, then exactly one index $k\neq i_0$ satisfies $T_k=-2\epsilon_k$, and $T_i=-\epsilon_i$ for all remaining indices $i$; 
        \item if $T_{i_0}=-3\epsilon_{i_0}$, then exactly two distinct indices $j,k\neq i_0$ satisfy $T_j=T_k=0$, and $T_i=-\epsilon_i$ for all remaining indices $i$. 
    \end{itemize}

    \item If $\mathbb{I}_{\operatorname{total}}=-2$, then $T_{i_0}\in\{-2\epsilon_{i_0},-3\epsilon_{i_0}\}$.
    \begin{itemize} 
        \item if $T_{i_0}=-2\epsilon_{i_0}$, then exactly one index $j\neq i_0$ satisfies $T_j=0$, exactly one index $k\notin\{i_0,j\}$ satisfies $T_k=-2\epsilon_k$, and $T_i=-\epsilon_i$ for all remaining indices $i$; \item if $T_{i_0}=-3\epsilon_{i_0}$, then exactly three distinct indices $j,k,\ell\neq i_0$ satisfy $T_j=T_k=T_\ell=0$, and $T_i=-\epsilon_i$ for all remaining indices $i$. 
    \end{itemize}
\end{itemize}
\end{remark}

\subsection{Case-by-case analysis}

We now proceed with a case-by-case analysis to prove that the embedding \[
    \iota\colon Q_{X(2,1)}\oplus Q_{X(3,2)}\oplus Q_{X(5,5-q_3)}\oplus Q_{X(n,n-q_4)} \longrightarrow -\mathbb{Z}^N
\] 
given in \eqref{eq:main_embedding} cannot exist. We divide our argument into three main cases based on the value of $q_3 \in \{1, 3, 4\}$. According to the classification of Lemma \ref{lem:6.1}, these are further subdivided, yielding ten subcases; one for $q_3=1$, two for $q_3=3$, and seven for $q_3=4$. 

We briefly describe the common strategy. Once the images of the first three plumbing summands are fixed, the condition \eqref{eq:property_of_K}, which specifies the intersection numbers between the vector $\mathbf{K}$ and the images of the standard basis vectors, leaves only finitely many possibilities for the initial coordinates of $\mathbf{K}$. For each such possibility, Lemma \ref{lem:6.4} gives strong restrictions on the corresponding coordinates of the total sum vector $\mathbf{T}$. In many cases these restrictions are already incompatible with the coordinate relations forced by the fixed images. In the remaining cases, we use the partial sum vectors $\mathbf{P}_k$ of the last plumbing summand. By applying Lemma \ref{lem:6.2} to constrain the possible coordinate patterns of $\mathbf{P}_k$, we either derive a rank contradiction or deduce that the orthogonal complement of the image of the embedding \eqref{eq:main_embedding} is generated by a vector whose square is not $-30n$, contradicting condition \ref{item:orthogonal_complement}. Table \ref{table:obstructions} summarizes the main obstruction used in each case.

\begin{table}[h!]
\centering
\small
\setlength{\extrarowheight}{3pt}
\resizebox{\textwidth}{!}{
\begin{tabular}{c|c|c|c}
$q_3$ &  Fixed image(s) from Lemma \ref{lem:6.1} & Possible $\mathbf{K}$ & Main obstruction \\ \hline \hline

$1$ & $e_6-e_7,e_7-e_8,e_8-e_9,e_9-e_{10}$ & $5$ choices & No compatible $\mathbf{T}$ satisfies Lemma \ref{lem:6.4}. \\ \hline

\multirow{3}{*}{3} & $e_1+e_2-e_6,e_6-e_7$ & unique & No compatible $\mathbf{T}$ satisfies Lemma \ref{lem:6.4}.  \\ \cline{2-4}

&\multirow{2}{*}{$e_6+e_7-e_8, e_8-e_9$} & \eqref{eq:K_case2-6} & No compatible $\mathbf{T}$ satisfies Lemma \ref{lem:6.4}.\\ \cline{3-4}

&& \eqref{eq:K_case2-1}-\eqref{eq:K_case2-5} & $\mathbf{P}_k$; rank or orthogonal-complement contradiction.\\ \hline

\multirow{8}{*}{4} & $e_1+e_2+e_3+e_4+e_5$ & - & Orthogonal complement condition.  \\ \cline{2-4}

 & $e_1+e_2-e_3-e_4-e_5$ & - & Orthogonal complement condition.\\ \cline{2-4}

 & $e_3+e_4+e_5+e_6+e_7$ & unique & No compatible $\mathbf{T}$ satisfies Lemma \ref{lem:6.4}.\\ \cline{2-4}
 
 & $-e_3-e_4-e_5+e_6+e_7$ & unique & No compatible $\mathbf{T}$ satisfies Lemma \ref{lem:6.4}. \\ \cline{2-4}
 
 & $2e_6+e_7$ & $2$ choices & $\mathbf{S}_4$; orthogonal complement condition. \\ \cline{2-4}
 
 & $e_1+e_2+e_6+e_7+e_8$ & $2$ choices & $\mathbf{S}_4$; orthogonal complement condition. \\ \cline{2-4}
 
& \multirow{2}{*}{\centering $e_6+e_7+e_8+e_9+e_{10}$}
& \eqref{case:4-7-5}
& No compatible $\mathbf{T}$ satisfies Lemma \ref{lem:6.4}. \\ \cline{3-4}

&
& \eqref{case:4-7-1}--\eqref{case:4-7-4}
& $\mathbf{P}_k$; rank or orthogonal-complement contradiction.
\end{tabular}}
\caption{Summary of the case-by-case analysis and the main obstruction used in each case.}
\label{table:obstructions}
\end{table}

\subsubsection{\bf{The case $q_3=1$}}
By Lemma \ref{lem:6.1}, we may assume that $\iota(u_1)=e_1-e_2$, $\iota(u_2)=e_3-e_4$, $\iota(u_3)=e_4-e_5$, and 
\[
    (\iota(u_4),\iota(u_5),\iota(u_6),\iota(u_7))=(e_6-e_7,e_7-e_8,e_8-e_9,e_9-e_{10}).
\]
Consider the vector $\mathbf{K}=\sum_{i=1}^N K_ie_i$ of the form \eqref{eq:vector_K} satisfying \eqref{eq:property_of_K}. Then \eqref{eq:property_of_K} gives 
\[
    K_1=K_2-2, \quad K_3=K_4=K_5-2, \quad K_6=K_7=K_8=K_9=K_{10}-2.
\]
It follows that $(K_1,\dots,K_{10})$ is one of the following:
\begin{numcases}{(K_1,\dots,K_{10})=}
    (-1,+1,-1,-1,+1,-1,-1,-1,-1,+1), & \label{eq:K_case1-1} \\
    (-1,+1,-1,-1,+1,+1,+1,+1,+1,+3), & \label{eq:K_case1-2} \\
    (-1,+1, +1,+1,+3, -1,-1,-1,-1,+1), & \label{eq:K_case1-3} \\
    (-3,-1,-1,-1,+1,-1,-1,-1,-1,+1), & \label{eq:K_case1-4} \\
    (+1,+3,-1,-1,+1,-1,-1,-1,-1,+1). & \label{eq:K_case1-5}
\end{numcases}

Next, write $\mathbf{T}=\sum_{i=1}^NT_ie_i$. Then
\[
    \mathbf{T}=\mathbf{S}_1+\mathbf{S}_2+\mathbf{S}_3+\mathbf{S}_4=(e_1-e_2+e_3-e_5+e_6-e_{10})+\mathbf{S}_4.
\]
Since $\mathbf{S}_4$ is orthogonal to $\iota(u_1),\dots,\iota(u_7)$, computing the intersection numbers $\mathbf{T}\cdot \iota(u_j)$ for $j=1,\dots,7$ gives 
\[
    T_1-T_2=2, \quad T_3-T_4=T_4-T_5=1, \quad T_6-T_7=1, \quad T_7=T_8=T_9, \quad T_9-T_{10}=1. 
\]
By Lemma \ref{lem:6.3} we also have $\mathbb{I}_{\operatorname{total}}\in\{-2,-1\}$ . 

If $(K_3,K_4,K_5)=(-1,-1,+1)$, as in the cases \eqref{eq:K_case1-1}, \eqref{eq:K_case1-2}, \eqref{eq:K_case1-4}, and \eqref{eq:K_case1-5}, then the relation between $\mathbf{K}$ and $\mathbf{T}$ given in Lemma \ref{lem:6.4} (see also Remark \ref{rmk:6.5}), together with the condition $T_3-T_4=T_4-T_5=1$, implies that $(T_3,T_4,T_5)$ is either $(+1,0,-1)$ or $(+2,+1,0)$. If $(K_3,K_4,K_5)=(+1,+1,+3)$, as in the case \eqref{eq:K_case1-3}, then the only possibility for $(T_3,T_4,T_5)$ is $(0,-1,-2)$. Similarly, if $(K_6,\dots,K_{10})=(-1,-1,-1,-1,+1)$, as in the cases \eqref{eq:K_case1-1}, \eqref{eq:K_case1-3}, \eqref{eq:K_case1-4}, and \eqref{eq:K_case1-5}, then $(T_6,\dots,T_{10})$ must be either $(+1,0,0,0,-1)$ or $(+2,+1,+1,+1,0)$. If $(K_6,\dots,K_{10})=(+1,+1,+1,+1,+3)$, as in the case \eqref{eq:K_case1-2}, then $(T_6,\dots,T_{10})=(0,-1,-1,-1,-2)$. 

Hence, if $(K_1,\dots,K_{10})$ is given by \eqref{eq:K_case1-1}, \eqref{eq:K_case1-4}, or \eqref{eq:K_case1-5}, then $(T_3,\dots,T_{10})$ is one of the following: \[
(T_3,\dots,T_{10})=\begin{cases}
    (+1,0,-1,+1,0,0,0,-1), \\
    (+1,0,-1,+2,+1,+1,+1,0),\\
    (+2,+1,0,+1,0,0,0,-1), \\
    (+2,+1,0,+2,+1,+1,+1,0). 
\end{cases}
\]
The first case is impossible since the vector $\mathbf{T}$ can have at most three zero coefficients. The other cases are also impossible since if $T_i=\pm 2$ for some $i$, then $\mathbf{T}$ can have at most one zero coefficient. If $(K_1,\dots,K_{10})$ is given by \eqref{eq:K_case1-2}, then $(T_3,\dots,T_{10})$ is either \[
(+1,0,-1, 0,-1,-1,-1,-2) \quad \text{or} \quad (+2,+1,0,0,-1,-1,-1,-2),
\]
but both cases are impossible by similar reasoning. If $(K_1,\dots,K_{10})$ is given by \eqref{eq:K_case1-3}, then $(T_3,\dots,T_{10})$ is either \[
(0,-1,-2,+1,0,0,0,-1) \quad \text{or} \quad (0,-1,-2,+2,+1,+1,+1,0),
\]
which are again impossible. This completes the proof that the case $q_3=1$ cannot occur.

\subsubsection{\bf{The case $q_3=3$}}\label{sec:q_3=3}
By Lemma \ref{lem:6.1}, we may assume that $\iota(u_1)=e_1-e_2$, $\iota(u_2)=e_3-e_4$, $\iota(u_3)=e_4-e_5$, and \[
(\iota(u_4),\iota(u_5))=(e_1+e_2-e_6,e_6-e_7) \quad \text{or} \quad (e_6+e_7-e_8,e_8-e_9).
\]

\vspace{1mm}
\textbf{\hypertarget{Case_3-1}{Case 1}:} $(\iota(u_4),\iota(u_5))=(e_1+e_2-e_6,e_6-e_7)$.\\ The vector $\mathbf{K}=\sum_{i=1}^N K_ie_i$ of the form \eqref{eq:vector_K} satisfies \eqref{eq:property_of_K}, and it follows immediately that \[
(K_1,\dots,K_7)=(-1,+1,-1,-1,+1,+1,+3).
\]
On the other hand, the vector $\mathbf{T}=\sum_{i=1}^NT_ie_i$ satisfies 
\[
    \mathbf{T}=\mathbf{S}_1+\mathbf{S}_2+\mathbf{S}_3+\mathbf{S}_4=(2e_1+e_3-e_5-e_7)+\mathbf{S}_4.
\]
Since $\mathbf{S}_4$ is orthogonal to the vectors $\iota(u_1),\dots,\iota(u_5)$, computing the intersection numbers $\mathbf{T}\cdot \iota(u_j)$ for $j=1,\dots,5$ gives 
\[
    T_1-T_2=2, \quad T_3-T_4=T_4-T_5=1, \quad T_1+T_2-T_6=2, \quad T_6-T_7=1.
\]
Since $K_7=+3$, Lemma \ref{lem:6.4} shows that $T_7\in \{-2,-3\}$. If $T_7=-2$, then $T_6-T_7=1$ implies $T_6=-1$. The conditions $T_1-T_2=2$ and $T_1+T_2-T_6=2$ then imply $(T_1,T_2)=(\frac{3}{2},-\frac{1}{2})$, which is impossible since the $T_i$ are integers. If $T_7=-3$, then Lemma \ref{lem:6.4} gives $T_1\in \{0,+1\}$ and $T_2\in \{-1,0\}$. Hence the condition $T_1-T_2=2$ implies $(T_1,T_2)=(+1,-1)$. Then the condition $T_1+T_2-T_6=2$ implies $T_6=-2$, which is impossible since the vector $\mathbf{T}$ cannot have coefficients $\pm 2$ and $\pm3$ simultaneously by Lemma \ref{lem:6.4}. We conclude that \textbf{\hyperlink{Case_3-1}{Case 1}} cannot occur.

\vspace{1mm}
\textbf{\hypertarget{Case_3-2}{Case 2}:} $(\iota(u_4),\iota(u_5))=(e_6+e_7-e_8,e_8-e_9)$.\\ This is the most involved case in this section. Note that interchanging $e_6$ and $e_7$ leaves the vectors $\iota(u_1),\dots,\iota(u_5)$ unchanged. Up to interchanging $e_6$ and $e_7$, one checks that $(K_1,\dots,K_9)$ is one of the following: \begin{numcases}{(K_1,\dots,K_9)=}
 (-1,+1,+1,+1,+3,-1,-1,-1,+1), & \label{eq:K_case2-6}\\
 (-3,-1,-1,-1,+1,-1,-1,-1,+1), & \label{eq:K_case2-1} \\ 
 (+1,+3,-1,-1,+1,-1,-1,-1,+1), & \label{eq:K_case2-2} \\
 (-1,+1,-1,-1,+1,-3,+1,-1,+1), & \label{eq:K_case2-3} \\
 (-1,+1,-1,-1,+1,-1,+1,+1,+3), & \label{eq:K_case2-4} \\
 (-1,+1,-1,-1,+1,-1,-1,-1,+1). & \label{eq:K_case2-5} 
\end{numcases}
Moreover, by applying, if necessary, an automorphism of $-\mathbb{Z}^N$ that changes the signs of some of the vectors $e_i$ with $i>9$, we may assume that $K_{10}=\cdots=K_N=+1$, except in the case \eqref{eq:K_case2-5}. In the case \eqref{eq:K_case2-5}, we assume instead that $K_{10}=+3$ and $K_{11}=\cdots=K_N=+1$. 

On the other hand, the vector $\mathbf{T}=\sum_{i=1}^N T_ie_i$ satisfies 
\[
\mathbf{T}=(e_1-e_2+e_3-e_5+e_6+e_7-e_9)+\mathbf{S}_4,
\]
and computing $\mathbf{T}\cdot \iota(u_j)$ for $j=1,\dots,5$ gives 
\[
T_1-T_2=2, \quad T_3-T_4=T_4-T_5=1, \quad T_6+T_7-T_8=2, \quad T_8-T_9=1.
\]

We now consider the cases \eqref{eq:K_case2-6}--\eqref{eq:K_case2-5} for $\mathbf{K}$ separately. 

\vspace{1mm}
\textbf{\hypertarget{Case_3-2-0}{Case 2-1} \eqref{eq:K_case2-6}:} $(K_1,\dots,K_9)=(-1,+1,+1,+1,+3,-1,-1,-1,+1)$. 
This case is easily eliminated. Indeed, since $K_5=+3$, Lemma \ref{lem:6.4} shows that $T_5\in \{-2,-3\}$. If $T_5=-2$, then $T_3-T_4=T_4-T_5=1$ implies $(T_3,T_4,T_5)=(0,-1,-2)$. Thus Lemma \ref{lem:6.4} shows that $\mathbb{I}_{\operatorname{total}}=-2$, and that $T_i\neq 0$ for $i\neq 3$. Now $(K_8,K_9)=(-1,+1)$ implies $T_8\in \{+1,+2\}$ and $T_9\in \{-1,-2\}$, so the condition $T_8-T_9=1$ cannot be satisfied. If $T_5=-3$, then $T_3-T_4=T_4-T_5=1$ implies $(T_3,T_4,T_5)=(-1,-2,-3)$, which is impossible by Lemma \ref{lem:6.4}. 

\vspace{1mm}
\textbf{\hypertarget{Case_3-2-1}{Case 2-2} \eqref{eq:K_case2-1}:} $(K_1,\dots,K_9)=(-3,-1,-1,-1,+1,-1,-1,-1,+1)$.
Since $K_1=-3$, we have $T_1\in \{+2,+3\}$. If $T_1=+2$, then $T_1-T_2=2$ implies $T_2=0$. Thus, we must have $\mathbb{I}_{\operatorname{total}}=-2$ and $T_i\neq 0$ for $i\neq 2$. Then $(K_4,K_5)=(-1,+1)$ implies $T_4\in \{+1,+2\}$ and $T_5\in \{-1,-2\}$, so the condition $T_4-T_5=1$ cannot be satisfied. Therefore, we must have $T_1=+3$, and it is easily seen that \[
(T_1,\dots,T_9)=(+3,+1,+1,0,-1,+1,+1,0,-1).
\]

Now, if $\mathbb{I}_{\operatorname{total}}=-1$, then $T_{10}=\cdots=T_N=-1$ by Lemma \ref{lem:6.4}. Thus, $\mathbf{T}+\mathbf{K}=-e_4-e_8$, and therefore \[
\mathbf{S}_4+\mathbf{K}=\mathbf{T}+\mathbf{K}-(\mathbf{S}_1+\mathbf{S}_2+\mathbf{S}_3)=-e_1+e_2-e_3-e_4+e_5-e_6-e_7-e_8+e_9.
\]
On the other hand, since the vector $\mathbf{P}_k=\sum_{i=1}^k \iota(v_i)$ $(k=1,\dots,s)$ is orthogonal to the vectors $\iota(u_1),\dots,\iota(u_5)$, it must be of the form \[
\mathbf{P}_k=\alpha_k(e_1+e_2)+\beta_k(e_3+e_4+e_5)+\gamma_k(e_6+e_8+e_9)+\delta_k(e_7+e_8+e_9)+\sum_{i=10}^N d_{k,i}e_i.
\]
Let us compute $C_k:=(\alpha_k,\beta_k,\gamma_k,\delta_k)$ for each $k$. For $k=s$, we have \[
\mathbf{P}_s=\mathbf{S}_4=\mathbf{T}-(\mathbf{S}_1+\mathbf{S}_2+\mathbf{S}_3)=2(e_1+e_2)-(e_{10}+\cdots+e_{N}),
\]
so $C_s=(2,0,0,0)$. For $k<s$, Lemma \ref{lem:6.2}(f) implies \begin{equation}\label{eq:case2-1}
    -1=\mathbf{P}_k\cdot (\mathbf{S}_4+\mathbf{K}) = \beta_k+\gamma_k+\delta_k.
\end{equation}
Also, Lemma \ref{lem:6.2}(d) implies \begin{align*}
    2& =-(\mathbf{P}_k^2+\mathbf{P}_k\cdot \mathbf{K})\\ & = (2\alpha_k^2-4\alpha_k)+(3\beta_k^2-\beta_k) +(3\gamma_k^2+3\delta_k^2+4\gamma_k\delta_k-\gamma_k-\delta_k) +\sum_{i=10}^N (d_{k,i}^2+d_{k,i}).
\end{align*}
Since $d_{k,i}^2+d_{k,i}\geq 0$, it follows that \begin{equation}\label{eq:case2-1-2}
    (2\alpha_k^2-4\alpha_k)+(3\beta_k^2-\beta_k) +(3\gamma_k^2+3\delta_k^2+4\gamma_k\delta_k-\gamma_k-\delta_k)  \leq 2.
\end{equation}
By \eqref{eq:case2-1} and \eqref{eq:case2-1-2}, we must have \[
C_k=(\alpha_k,\beta_k,\gamma_k,\delta_k)\in \{(1,-1,0,0), (1,0,-1,0), (1,0,0,-1)\} \quad (k=1,\dots,s-1).
\]
We claim that if the set $C=\{C_1,\dots,C_{s-1}\}$ contains $(1,-1,0,0)$, then it cannot contain $(1,0,-1,0)$ or $(1,0,0,-1)$. If this were not the case, there would exist $k,\ell<s$ such that $C_k=(1,-1,0,0)$ and $C_\ell\in \{(1,0,-1,0),(1,0,0,-1)\}$. Assume that $C_\ell=(1,0,-1,0)$ and $k<\ell$; the other cases are similar. Then $C_\ell-C_k=(0,1,-1,0)$, so \[
\mathbf{P}_\ell-\mathbf{P}_k=(e_3+e_4+e_5)-(e_6+e_8+e_9)+\sum_{i=10}^N d_ie_i
\]
where $d_i=d_{\ell,i}-d_{k,i}$. By Lemma \ref{lem:6.2}(e), it follows that \begin{align*}
    2 & = -\left((\mathbf{P}_\ell-\mathbf{P}_k)^2+(\mathbf{P}_\ell-\mathbf{P}_k)\cdot \mathbf{K} \right)  = 6+\sum_{i=10}^N (d_i^2+d_i)\geq 6,
\end{align*}
which is a contradiction. This proves our claim. 

Now, if $C$ contains $(1,-1,0,0)$, then $C_k\notin \{(1,0,-1,0),(1,0,0,-1)\}$ for all $k$, so $\gamma_k=\delta_k=0$ for all $k=1,\dots,s$. Observing that  $\iota(v_k)=\mathbf{P}_k-\mathbf{P}_{k-1}$ (with the convention that $\mathbf{P}_0=0$), it follows that \[
\iota(u_1),\iota(u_2),\iota(u_3),\iota(v_1),\dots,\iota(v_s)\in \langle e_1,\dots,e_5,e_{10},\dots,e_N\rangle  \cong -\mathbb{Z}^{N-4}=-\mathbb{Z}^{s+2},
\]
i.e., the image $\iota(Q_{X(2,1)}\oplus Q_{X(3,2)}\oplus Q_{X(n,n-q_4)} )$ is a sublattice of $\langle e_1,\dots,e_5,e_{10},\dots,e_N\rangle$, but this is impossible since $\operatorname{rank}(Q_{X(2,1)}\oplus Q_{X(3,2)}\oplus Q_{X(n,n-q_4)} )=s+3$. 

If $C$ does not contain $(1,-1,0,0)$, then $\beta_k=0$ for all $k$, from which it follows that \[
\iota(u_1),\iota(u_4),\iota(u_5),\iota(v_1),\dots,\iota(v_s)\in \langle e_1,e_2,e_6,\dots,e_N\rangle \cong -\mathbb{Z}^{N-3}=-\mathbb{Z}^{s+3},
\]
i.e., the image $\iota(Q_{X(2,1)}\oplus Q_{X(5,2)}\oplus Q_{X(n,n-q_4)})$ is a full-rank sublattice of $\langle e_1,e_2,e_6,\dots,e_N\rangle$. Thus, any vector $\sum_{i=1}^N c_ie_i\in -\mathbb{Z}^N$ orthogonal to $\iota(Q_{X(2,1)}\oplus Q_{X(5,2)}\oplus Q_{X(n,n-q_4)})$ must satisfy $c_1=c_2=c_6=\cdots=c_N=0$. Consequently, the orthogonal complement of the image of the embedding \eqref{eq:main_embedding} is generated by the vector $e_3+e_4+e_5$, which has square $-3\neq -30 n$, contradicting condition \ref{item:orthogonal_complement}. This completes the proof for \textbf{\hyperlink{Case_3-2-1}{Case 2-2}} in the case $\mathbb{I}_{\operatorname{total}}=-1$. 

The case $\mathbb{I}_{\operatorname{total}}=-2$ is similar. In this case, we may assume that $T_{10}=\cdots=T_{N-1}=-1$ and $T_N=0$. For each $k=1,\dots,s$, we write $\mathbf{P}_k$ as \[
\mathbf{P}_k=\alpha_k(e_1+e_2)+\beta_k(e_3+e_4+e_5)+\gamma_k(e_6+e_8+e_9)+\delta_k(e_7+e_8+e_9)+\sum_{i=10}^{N-1} d_{k,i}e_i+\eta_k e_N,
\]
and consider $C_k:=(\alpha_k,\beta_k,\gamma_k,\delta_k,\eta_k)$. For $k=s$, we have \[
\mathbf{P}_s=\mathbf{S}_4=2(e_1+e_2)-(e_{10}+\cdots+e_{N-1}),
\]
so $C_s=(2,0,0,0,0)$. For $k<s$, the condition $\mathbf{P}_k\cdot (\mathbf{S}_4+\mathbf{K})=-1$ of Lemma \ref{lem:6.2}(f) gives \begin{equation}\label{eq:case:2-1-2(1)}
    \beta_k+\gamma_k+\delta_k-\eta_k=-1,
\end{equation}
and the condition $\mathbf{P}_k^2+\mathbf{P}_k\cdot \mathbf{K}=-2$ of Lemma \ref{lem:6.2}(d) gives \begin{equation}\label{eq:case:2-1-2(2)}
   (2\alpha_k^2-4\alpha_k)+(3\beta_k^2-\beta_k) +(3\gamma_k^2+3\delta_k^2+4\gamma_k\delta_k-\gamma_k-\delta_k) +(\eta_k^2+\eta_k)  \leq 2.
\end{equation}
Combining \eqref{eq:case:2-1-2(1)} and \eqref{eq:case:2-1-2(2)} gives \[
C_k \in \left\{ 
\begin{aligned}
    &(0,0,0,0,1), \, (1,0,0,0,1), \, (2,0,0,0,1), \, (1,-1,0,0,0), \\
    &(1,0,-1,0,0), \, (1,0,-1,1,1), \, (1,0,0,-1,0), \, (1,0,1,-1,1)
\end{aligned} 
\right\} \quad (k=1,\dots,s-1).
\]
Then by using the condition $(\mathbf{P}_\ell-\mathbf{P}_k)^2+(\mathbf{P}_\ell-\mathbf{P}_k)\cdot \mathbf{K}=-2$ of Lemma \ref{lem:6.2}(e) as above, one checks that if the set $\{C_1,\dots,C_{s-1}\}$ contains $(1,-1,0,0,0)$, then it cannot contain any of $(1, 0, -1, 0, 0)$, $(1, 0, -1, 1, 1)$, $(1, 0, 0, -1, 0)$, $(1, 0, 1, -1, 1)$. This shows that we have either $\beta_k=0$ for all $k=1,\dots s$, or $\gamma_k=\delta_k=0$ for all $k=1,\dots,s$. Now we can obtain a contradiction exactly as in the case $\mathbb{I}_{\operatorname{total}}=-1$. This completes the proof for \textbf{\hyperlink{Case_3-2-1}{Case 2-2}}.

\vspace{2mm}
The remaining cases \eqref{eq:K_case2-2}-\eqref{eq:K_case2-5} are similar to \textbf{\hyperlink{Case_3-2-1}{Case 2-2}} \eqref{eq:K_case2-1}, so we summarize the arguments briefly.

\vspace{1mm}
\textbf{\hypertarget{Case_3-2-2}{Case 2-3} \eqref{eq:K_case2-2}:} $(K_1,\dots,K_9)=(+1,+3,-1,-1,+1,-1,-1,-1,+1)$.
In this case, we have \[
(T_1,\dots,T_9)=(-1,-3,+1,0,-1,+1,+1,0,-1).
\]
If $\mathbb{I}_{\operatorname{total}}=-1$, then for $k=1,\dots,s$, write \begin{equation}\label{eq:P_k_for_I_total=-1}
\mathbf{P}_k=\alpha_k(e_1+e_2)+\beta_k(e_3+e_4+e_5)+\gamma_k(e_6+e_8+e_9)+\delta_k(e_7+e_8+e_9)+\sum_{i=10}^N d_{k,i}e_i,
\end{equation}
and let $C_k=(\alpha_k,\beta_k,\gamma_k,\delta_k)$. We have $C_s=(-2,0,0,0)$. For $k<s$, the conditions $\mathbf{P}_k\cdot (\mathbf{S}_4+\mathbf{K})=-1$ and $\mathbf{P}_k^2+\mathbf{P}_k\cdot \mathbf{K}=-2$ imply \[
C_k\in \{(-1,-1,0,0), (-1,0,-1,0), (-1,0,0,-1)\} \quad (k=1,\dots,s-1).
\]
From here, we proceed exactly as in \textbf{\hyperlink{Case_3-2-1}{Case 2-2}} to show that if the set $\{C_1,\dots,C_{s-1}\}$ contains $(-1,-1,0,0)$, then it cannot contain $(-1,0,-1,0)$ or $(-1,0,0,-1)$, thereby obtaining a contradiction.

If $\mathbb{I}_{\operatorname{total}}=-2$, we may assume that $T_{10}=\cdots=T_{N-1}=-1$ and $T_N=0$. For $k=1,\dots,s$, write \begin{equation}\label{eq:P_k_for_I_total=-2}
  \mathbf{P}_k=\alpha_k(e_1+e_2)+\beta_k(e_3+e_4+e_5)+\gamma_k(e_6+e_8+e_9)+\delta_k(e_7+e_8+e_9)+\sum_{i=10}^{N-1} d_{k,i}e_i+\eta_k e_N,  
\end{equation}
and let $C_k=(\alpha_k,\beta_k,\gamma_k,\delta_k,\eta_k)$. In this case, we have $C_s=(-2,0,0,0,0)$ and \[
C_k \in \left\{ 
\begin{aligned}
    &(0,0,0,0,1), \, (-1,0,0,0,1), \, (-2,0,0,0,1), \, (-1,-1,0,0,0), \\
    &(-1,0,-1,0,0), \, (-1,0,-1,1,1), \, (-1,0,0,-1,0), \, (-1,0,1,-1,1)
\end{aligned} 
\right\} \quad (k=1,\dots,s-1).
\]
Again, we proceed as in \textbf{\hyperlink{Case_3-2-1}{Case 2-2}} to show that if the set $\{C_1,\dots,C_{s-1}\}$ contains $(-1,-1,0,0,0)$, then it cannot contain any of $(-1,0,-1,0,0)$, $(-1,0,-1,1,1)$, $(-1,0,0,-1,0)$, $(-1,0,1,-1,1)$, thereby obtaining a contradiction.

\vspace{1mm}
\textbf{\hypertarget{Case_3-2-3}{Case 2-4} \eqref{eq:K_case2-3}:} $(K_1,\dots,K_9)=(-1,+1,-1,-1,+1,-3,+1,-1,+1)$.
In this case, we have \begin{numcases}
    {(T_1,\dots,T_9)=} (+1,-1,+1,0,-1,+3,-1,0,-1), \label{case:3-2-3(1)} \\
    (+1,-1,+1,0,-1,+3,0,+1,0). \label{case:3-2-3(2)}
\end{numcases}
Consider first the case \eqref{case:3-2-3(1)}. If $\mathbb{I}_{\operatorname{total}}=-1$, write $\mathbf{P}_k$ as \eqref{eq:P_k_for_I_total=-1} for $k=1,\dots,s$, and let $C_k=(\alpha_k,\beta_k,\gamma_k,\delta_k)$. For $k=s$, we have $C_s=(0,0,2,-2)$. For $k<s$, we have \[
C_k \in\left\{ (0, 0, 1, -2), (0, 0, 0, -1), (0, -1, 1, -1)\right\}.
\]
In particular, $\alpha_k=0$ for all $k=1,\dots,s$. This implies that \[
\iota(u_2),\iota(u_3),\iota(u_4),\iota(u_5),\iota(v_1),\dots,\iota(v_s)\in \langle e_3,\dots,e_N\rangle \cong -\mathbb{Z}^{N-2}=-\mathbb{Z}^{s+4},
\]
i.e., the image $\iota(Q_{X(3,2)}\oplus Q_{X(5,2)}\oplus Q_{X(n,n-q_4)})$ is a full-rank sublattice of $\langle e_3,\dots,e_N\rangle$. Therefore, the orthogonal complement of the image of the embedding \eqref{eq:main_embedding} is generated by the vector $e_1+e_2$, which has square $-2\neq -30 n$, contradicting condition \ref{item:orthogonal_complement}. 

If $\mathbb{I}_{\operatorname{total}}=-2$, assume $T_{10}=\cdots=T_{N-1}=-1$ and $T_N=0$, and write $\mathbf{P}_k$ as \eqref{eq:P_k_for_I_total=-2} for $k=1,\dots,s$, and let $C_k=(\alpha_k,\beta_k,\gamma_k,\delta_k,\eta_k)$. For $k=s$, we have $C_s=(0,0,2,-2,0)$, and for $k<s$, we have \[
C_k\in \left\{ 
\begin{aligned}
    & (0, 0, 0, 0, 1), \, (0, 0, 0, -1, 0), \, (0, 0, 1, -1, 1), \, (0, -1, 1, -1, 0), \\
 & (1, 0, 1, -1, 1), \, (-1, 0, 1, -1, 1), \, (0, 0, 1, -2, 0), \, (0, 0, 2, -2, 1)
\end{aligned}
\right\}
\]
As in \textbf{\hyperlink{Case_3-2-1}{Case 2-2}}, we can show that if the set $\{C_1,\dots,C_{s-1}\}$ contains $(0,-1,1,-1,0)$, then it cannot contain $(1,0,1,-1,1)$ or $(-1,0,1,-1,1)$. It follows that either $\alpha_k=0$ for all $k=1,\dots,s$, or $\beta_k=0$ for all $k=1,\dots,s$. In the former case, we obtain a contradiction as in the preceding paragraph. In the latter case, we obtain a contradiction as in \textbf{\hyperlink{Case_3-2-1}{Case 2-2}}. 

Finally, consider the case \eqref{case:3-2-3(2)}. In this case, the vector $\mathbf{T}$ already has three zero coefficients, so we must have $\mathbb{I}_{\operatorname{total}}=-2$, and $T_{10}=\cdots=T_N=-1$. Now write $\mathbf{P}_k$ as \eqref{eq:P_k_for_I_total=-1} for $k=1,\dots,s$, and let $C_k=(\alpha_k,\beta_k,\gamma_k,\delta_k)$. For $k=s$, we have $C_s=(0,0,2,-1)$, and for $k<s$, we have \[
C_k\in \{(0,0,1,0), (1,0,1,0), (-1,0,1,0)\}.
\]
In particular, we have $\beta_k=0$ for all $k=1,\dots,s$, and we may obtain a contradiction as in \textbf{\hyperlink{Case_3-2-1}{Case 2-2}}. 

\vspace{1mm}
\textbf{\hypertarget{Case_3-2-4}{Case 2-5} \eqref{eq:K_case2-4}:} $(K_1,\dots,K_9)=(-1,+1,-1,-1,+1,-1,+1,+1,+3)$.
In this case, we have \[
(T_1,\dots,T_9)=(+1,-1,+1,0,-1,+2,-1,-1,-2),
\]
and Lemma \ref{lem:6.4} shows that we must have $\mathbb{I}_{\operatorname{total}}=-2$, so $T_{10}=\cdots=T_N=-1$. Now write $\mathbf{P}_k$ as \eqref{eq:P_k_for_I_total=-1} for $k=1,\dots,s$, and let $C_k=(\alpha_k,\beta_k,\gamma_k,\delta_k)$. For $k=s$, we have $C_s=(0,0,1,-2)$, and for $k<s$, we have \[
C_k\in \{(0,0,1,-1), (1,0,1,-1), (-1,0,1,-1)\}.
\]
In particular, we have $\beta_k=0$ for all $k=1,\dots,s$, and we may obtain a contradiction as above.

\vspace{1mm}
\textbf{\hypertarget{Case_3-2-5}{Case 2-6} \eqref{eq:K_case2-5}:} $(K_1,\dots,K_9)=(-1,+1,-1,-1,+1,-1,-1,-1,+1)$. 
In this case, we have \[
(T_1,\dots,T_9)=(+1,-1,+1,0,-1,+1,+1,0,-1).
\]
The vector $\mathbf{T}$ already has two zero coefficients, so we must have $T_{10}=-3$. 

If $\mathbb{I}_{\operatorname{total}}=-1$, write $\mathbf{P}_k$ for $k=1,\dots,s$ as \[
\mathbf{P}_k = \alpha_k(e_1+e_2)+\beta_k(e_3+e_4+e_5)+\gamma_k(e_6+e_8+e_9)+\delta_k(e_7+e_8+e_9)+\eta_k e_{10}+\sum_{i=11}^N d_{k,i}e_i,
\]
and let $C_k=(\alpha_k,\beta_k,\gamma_k,\delta_k,\eta_k)$. For $k=s$, we have $C_s=(0,0,0,0,-3)$, and for $k<s$, we have \[
C_k\in \left\{ \begin{aligned}
     & (0,-1,0,0,-1), \,(0,-1,0,0,-2),\, (0,0,-1,0,-1), \\ & (0,0,-1,0,-2),\, (0,0,0,-1,-1), \,(0,0,0,-1,-2)
\end{aligned}\right\}.
\]
In particular, we have $\alpha_k=0$ for all $k=1,\dots,s$, and we may obtain a contradiction as in \textbf{\hyperlink{Case_3-2-3}{Case 2-4}}. 

Now assume $\mathbb{I}_{\operatorname{total}}=-2$. We may assume that $T_{11}=\cdots=T_{N-1}=-1$ and $T_N=0$. Write $\mathbf{P}_k$ for $k=1,\dots,s$ as \begin{align*}
    \mathbf{P}_k & = \alpha_k(e_1+e_2)+\beta_k(e_3+e_4+e_5)+\gamma_k(e_6+e_8+e_9)+\delta_k(e_7+e_8+e_9) \\
    & \quad + \eta_ke_{10}+\sum_{i=11}^{N-1} d_{k,i}e_i +  \lambda_k e_N,
\end{align*}
and let $C_k=(\alpha_k,\beta_k,\gamma_k,\delta_k,\eta_k,\lambda_k)$. For $k=s$, we have $C_s=(0,0,0,0,-3,0)$. For $k<s$, the conditions $\mathbf{P}_k\cdot (\mathbf{S}_4+\mathbf{K})=-1$ and $\mathbf{P}_k^2+\mathbf{P}_k\cdot \mathbf{K}=-2$ imply $\beta_k+\gamma_k+\delta_k-\lambda_k=-1$ and \[
2\alpha_k^2+(3\beta_k^2-\beta_k)+(3\gamma_k^2+3\delta_k^2+4\gamma_k\delta_k-\gamma_k-\delta_k)+(\eta_k^2+3\eta_k)+(\lambda_k^2+\lambda_k)\leq 2,
\]
and therefore \[
C_k\in \left\{\begin{aligned}
    &(1, 0, 0, 0, -1, 1), (1, 0, 0, 0, -2, 1), (-1, 0, 0, 0, -1, 1), (-1, 0, 0, 0, -2, 1), (0, 0, 0, 0, 0, 1), \\
    &(0, 0, 0, 0, -1, 1), (0, 0, 0, 0, -2, 1), (0, 0, 0, 0, -3, 1), (0, -1, 0, 0, -1, 0), (0, -1, 0, 0, -2, 0), \\
    & (0, 0, -1, 0, -1, 0), (0, 0, -1, 0, -2, 0), (0, 0, 0, -1, -1, 0), (0, 0, 0, -1, -2, 0), \\ & (0, 0, -1, 1, -1, 1), (0, 0, -1, 1, -2, 1), (0, 0, 1, -1, -1, 1), (0, 0, 1, -1, -2, 1)
\end{aligned} \right\}.
\]
Using the condition $(\mathbf{P}_\ell-\mathbf{P}_k)^2+(\mathbf{P}_\ell-\mathbf{P}_k)\cdot \mathbf{K}=-2$ of Lemma \ref{lem:6.2}(e), we can check that if the set $\{C_1,\dots,C_{s-1}\}$ contains either $(0,-1,0,0,-1,0)$ or $(0,-1,0,0,-2,0)$, then it cannot contain any of $(1,0,0,0,-1,1)$, $(1,0,0,0,-2,1)$, $(-1,0,0,0,-1,1)$, $(-1,0,0,0,-2,1)$. This implies that either $\alpha_k=0$ for all $k=1,\dots,s$, or $\beta_k=0$ for all $k=1,\dots,s$. Thus, we may obtain a contradiction as in \textbf{\hyperlink{Case_3-2-3}{Case 2-4}}. This completes the proof for \textbf{\hyperlink{Case_3-2}{Case 2}}. 

\subsubsection{\bf{The case $q_3=4$}} 
By Lemma \ref{lem:6.1}, we may assume that $\iota(u_1)=e_1-e_2$, $\iota(u_2)=e_3-e_4$, and $\iota(u_3)=e_4-e_5$. For $\iota(u_4)$, there are exactly seven cases given in \eqref{eq:iota(u_4)_for_q_3=4}. 

\vspace{1mm}
\textbf{\hypertarget{Case_4-1}{Case 1}:} $\iota(u_4)=e_1+e_2+e_3+e_4+e_5$. This case is an easy consequence of the following lemma. 

\begin{lemma}\label{lem:6.6}
    For a standard basis vector $v\in \{v_1,\dots,v_s\}$ of $Q_{X(n,n-q_4)}$, if $\iota(v)=\sum_{i=1}^N m_ie_i$, then we have either:
    \begin{itemize}
        \item $|m_i|\leq 2$ for all $i=1,\dots,N$, or
        \item $|m_i|=3$ for exactly one index $i$, and $|m_j|\leq 2$ for all remaining indices $j$.
    \end{itemize}
\end{lemma}
\begin{proof} For the vector $\mathbf{K}=\epsilon_1e_1+\cdots+3\epsilon_{i_0}e_{i_0}+\cdots +\epsilon_Ne_N$, we have \[
    \mathbf{K}\cdot \iota(v)=-3\epsilon_{i_0}m_{i_0}-\sum_{i\neq i_0} \epsilon_im_i.
    \]
    On the other hand, \eqref{eq:property_of_K} implies \[
    \mathbf{K}\cdot \iota(v) \geq -v^2-2=-\iota(v)^2-2=-2+\sum_{i=1}^N m_i^2.
    \]
    Combining these gives \[
    \left(m_{i_0}^2+3\epsilon_{i_0}m_{i_0}\right) +\sum_{i\neq i_0} (m_i^2+\epsilon_im_i)\leq 2.
    \]
    This implies that $m_{i_0}\in \{0,\pm 1,\pm 2, \pm 3\}$ and $m_i\in \{0,\pm1,\pm2\}$ for $i\neq i_0$.
\end{proof}

Now, for each $v\in \{v_1,\dots,v_s\}$, since $\iota(v)$ is orthogonal to the vectors $\iota(u_1),\dots,\iota(u_4)$, it must be of the form 
\[
\iota(v)=\alpha(3e_1+3e_2-2e_3-2e_4-2e_5) +\sum_{i=6}^N d_ie_i. 
\]
But Lemma \ref{lem:6.6} shows that $\iota(v)$ cannot have two or more coefficients with absolute value $\geq 3$, so we must have $\alpha=0$. It follows that 
\[
\iota(v_1),\dots, \iota(v_s) \in \langle e_6,\dots,e_N\rangle \cong -\mathbb{Z}^{N-5}=-\mathbb{Z}^s, 
\]
i.e., the image $\iota(Q_{X(n,n-q_4)})$ is a full-rank sublattice of $\langle e_6,\dots,e_N\rangle$. Thus, any vector $\sum_{i=1}^N c_ie_i\in -\mathbb{Z}^N$ orthogonal to $\iota(Q_{X(n,n-q_4)})$ must satisfy $c_6=\cdots=c_N=0$. Consequently, the orthogonal complement of the image of the embedding \eqref{eq:main_embedding} is generated by the vector $3e_1+3e_2-2e_3-2e_4-2e_5$, which has square $-30\neq -30 n$, contradicting condition \ref{item:orthogonal_complement}. Therefore, \textbf{\hyperlink{Case_4-1}{Case 1}} cannot occur.

\vspace{1mm}
\textbf{Case 2:} $\iota(u_4)=e_1+e_2-e_3-e_4-e_5$. The argument is identical to that of \textbf{\hyperlink{Case_4-1}{Case 1}}. 

\vspace{1mm}
\textbf{\hypertarget{Case_4-3}{Case 3}:} $\iota(u_4)=e_3+e_4+e_5+e_6+e_7$. 
Note that interchanging $e_6$ and $e_7$ leaves the vectors $\iota(u_1),\dots,\iota(u_4)$ unchanged. Up to interchanging $e_6$ and $e_7$, we have \[
(K_1,\dots,K_7)=(-1,+1,-1,-1,+1,-3,-1).
\]
On the other hand, the vector $\mathbf{T}=\sum_{i=1}^NT_ie_i$ satisfies \[
\mathbf{T}=\mathbf{S}_1+\mathbf{S}_2+\mathbf{S}_3+\mathbf{S}_4=(e_1-e_2+2e_3+e_4+e_6+e_7)+\mathbf{S}_4.
\]
Since $\mathbf{S}_4$ is orthogonal to the vectors $\iota(u_1),\dots,\iota(u_4)$, calculating the intersection numbers $\mathbf{T}\cdot \iota(u_j)$ for $j=1,\dots,4$ gives \[
T_1-T_2=2, \quad T_3-T_4=T_4-T_5=1, \quad T_3+T_4+T_5+T_6+T_7=5.
\]
Note that $\mathbb{I}_{\operatorname{total}}\in \{-1,0\}$ by Lemma \ref{lem:6.3}. Since $K_6=-3$, we have $T_6\in \{+2,+3\}$. If $T_6=+2$, then $T_i\neq 0$ for all $i$ by Lemma \ref{lem:6.4}. Also, $(K_4,K_5)=(-1,+1)$ implies $T_4\in \{+1,+2\}$ and $T_5\in \{-1,-2\}$, so the condition $T_4-T_5=1$ cannot be satisfied. If $T_6=+3$, then the conditions $(K_3,K_4,K_5)=(-1,-1,+1)$ and $ T_3-T_4=T_4-T_5=1$ imply $(T_3,T_4,T_5)=(+1,0,-1)$, and the condition $K_7=-1$ implies $T_7\in \{0,+1\}$. Therefore, $T_3+T_4+T_5+T_6+T_7=3+T_7\leq 4$, i.e., the condition $T_3+T_4+T_5+T_6+T_7=5$ cannot be satisfied. We conclude that \textbf{\hyperlink{Case_4-3}{Case 3}} cannot occur.

\vspace{1mm}
\textbf{\hypertarget{Case_4-4}{Case 4}:} $\iota(u_4)=-e_3-e_4-e_5+e_6+e_7$. We proceed as in \textbf{\hyperlink{Case_4-3}{Case 3}}. Up to interchanging $e_6$ and $e_7$, we have \[
(K_1,\dots,K_7)=(-1,+1,+1,+1,+3,-1,+1).
\]
The vector $\mathbf{T}=\sum_{i=1}^N T_ie_i$ satisfies \[
\mathbf{T}=(e_1-e_2-e_4-2e_5+e_6+e_7)+\mathbf{S}_4,
\]
and calculating $\mathbf{T}\cdot \iota(u_j)$ for $j=2,3$ gives \[
T_3-T_4=T_4-T_5=1.
\]
Now, $K_5=+3$ implies $T_5\in \{-2,-3\}$. If $T_5=-2$, then the condition $T_3-T_4=T_4-T_5=1$ implies $(T_3,T_4)=(0,-1)$, which is impossible by Lemma \ref{lem:6.4}. If $T_5=-3$, then the condition $T_4-T_5=1$ implies $T_4=-2$, which is again impossible by Lemma \ref{lem:6.4}. We conclude that \textbf{\hyperlink{Case_4-4}{Case 4}} cannot occur.

\vspace{1mm}
\textbf{\hypertarget{Case_4-5}{Case 5}:} $\iota(u_4)=2e_6+e_7$. The vector $\mathbf{K}=\sum_{i=1}^N K_ie_i$ satisfies \[
(K_1,\dots,K_7)\in \{(-1,+1,-1,-1,+1,-3,+1), (-1,+1,-1,-1,+1,-1,-3)\}.
\]
We focus on the case $(K_1,\dots,K_7)=(-1,+1,-1,-1,+1,-3,+1)$; the other case is almost identical. The vector $\mathbf{T}=\sum_{i=1}^N T_ie_i$ satisfies \begin{equation}\label{eq:Case4-5}
   \mathbf{T}=(e_1-e_2+e_3-e_5+2e_6+e_7)+\mathbf{S}_4, 
\end{equation}
and calculating $\mathbf{T}\cdot \iota(u_j)$ for $j=1,\dots,4$ gives \[
T_1-T_2=2, \quad T_3-T_4=T_4-T_5=1, \quad 2T_6+T_7=5.
\]
Since $K_6=-3$, we have $T_6\in \{+2,+3\}$. If $T_6=+2$, then $T_i\neq 0$ for all $i$, and $K_7=+1$ implies $T_7\in \{-1,-2\}$, so the condition $2T_6+T_7=5$ cannot be satisfied. Thus, we must have $T_6=+3$, and the above conditions imply that \begin{equation}\label{eq:Case4-5:2}
    (T_1,\dots,T_7)=(+1,-1,+1,0,-1,+3,-1). 
\end{equation}
Moreover, by applying an automorphism of $-\mathbb{Z}^N$ which changes the sign of some $e_i$'s ($i>7$) if necessary, we may assume that $K_8=\cdots=K_N=+1$. 

Now, for each $i=1,\dots,s$, since $\iota(v_i)$ is orthogonal to the vectors $\iota(u_1),\dots,\iota(u_4)$, it must be of the form \[
\iota(v_i)=\alpha_i(e_1+e_2)+\beta_i(e_3+e_4+e_5)+\gamma_i(e_6-2e_7) +\sum_{j=8}^N d_{i,j}e_j,
\]
which gives \[
-v_i^2=-\iota(v_i)^2= 2\alpha_i^2+3\beta_i^2+5\gamma_i^2+\sum_{j=8}^N d_{i,j}^2
\]
and \begin{align*}
    \mathbf{K}\cdot \iota(v_i) & =(-e_1+e_2-e_3-e_4+e_5-3e_6+e_7+e_8+\cdots+e_N)\cdot \iota(v_i)=\beta_i +5\gamma_i -\sum_{j=8}^N d_{i,j}.
\end{align*}
Thus, \eqref{eq:property_of_K} implies \[
2\geq -(v_i^2+\mathbf{K}\cdot \iota(v_i)) = 2\alpha_i^2+(3\beta_i^2-\beta_i)+5(\gamma_i^2-\gamma_i)+\sum_{j=8}^N (d_{i,j}^2+d_{i,j}).
\]
Since $2\alpha_i^2\geq 0$, $\gamma_i^2-\gamma_i\geq 0$, and $d_{i,j}^2+d_{i,j}\geq 0$, it follows that $\beta_i\in \{0,1\}$. Now note that \[
\mathbf{S}_4=\sum_{i=1}^s \iota(v_i) = \sum_{i=1}^s \alpha_i(e_1+e_2)+\sum_{i=1}^s \beta_i(e_3+e_4+e_5)+\sum_{i=1}^s \gamma_i (e_6-2e_7)+\sum_{i=1}^s\sum_{j=8}^N d_{i,j}e_j.
\]
On the other hand, \eqref{eq:Case4-5} and \eqref{eq:Case4-5:2} imply $\mathbf{S}_4$ is of the form \[
\mathbf{S}_4=e_6-2e_7+\sum_{j=8}^N D_j e_j.
\]
Therefore, we must have $\sum_{i=1}^s \beta_i=0$. Since $\beta_i\in \{0,1\}$ for each $i$, it follows that $\beta_i=0$ for all $i$. As in \textbf{\hyperlink{Case_3-2-1}{Case 2-2}} of Section \ref{sec:q_3=3}, this implies that the orthogonal complement of the image of the embedding \eqref{eq:main_embedding} is generated by the vector $e_3+e_4+e_5$, which has square $-3\neq -30n$, contradicting condition \ref{item:orthogonal_complement}. Hence \textbf{\hyperlink{Case_4-5}{Case 5}} cannot occur.

\vspace{1mm}
\textbf{\hypertarget{Case_4-6}{Case 6}:} $\iota(u_4)=e_1+e_2+e_6+e_7+e_8$. Up to an automorphism of $-\mathbb{Z}^N$ permuting $\{e_6,e_7,e_8\}$ (which leaves $\iota(u_4)$ invariant), the vector $\mathbf{K}=\sum_{i=1}^N K_ie_i$ satisfies \begin{numcases}{(K_1,\dots,K_8)=}   (-3,-1,-1,-1,+1,-1,-1,+1), \label{eq:case6-1} \\
(-1,+1,-1,-1,+1,-1,-1,-3),    \label{eq:case6-2} 
\end{numcases}
and we may assume that $K_9=\cdots=K_N=+1$. The vector $\mathbf{T}=\sum_{i=1}^N T_ie_i$ satisfies \[
   \mathbf{T}=(2e_1+e_3-e_5+e_6+e_7+e_8)+\mathbf{S}_4, 
\]
and calculating $\mathbf{T}\cdot \iota(u_j)$ for $j=2,3$ gives \[
T_3-T_4=T_4-T_5=1,
\]
so we must have \[
(T_3,T_4,T_5)=(+1,0,-1). 
\]
For each $v\in \{v_1,\dots,v_s\}$, since $\iota(v)$ is orthogonal to the vectors $\iota(u_1),\dots,\iota(u_4)$, it must be of the form \[
\iota(v)=\alpha(e_1+e_2)+\beta(e_3+e_4+e_5)+\gamma_6e_6+\gamma_7e_7+\gamma_8e_8+\sum_{i=9}^N d_ie_i, \quad 2\alpha+\gamma_6+\gamma_7+\gamma_8=0.
\]
We will show that $\beta\in \{0,1\}$; then we can obtain a contradiction exactly as in \textbf{\hyperlink{Case_4-5}{Case 5}}. Letting $\mathbf{w}=\alpha(e_1+e_2)+\gamma_6e_6+\gamma_7e_7+\gamma_8e_8$, observe that \[
-(\mathbf{w}^2+\mathbf{K}\cdot \mathbf{w})=\begin{cases}
    (2\alpha^2-4\alpha)+(\gamma_6^2-\gamma_6)+(\gamma_7^2-\gamma_7)+(\gamma_8^2+\gamma_8) & \text{in case \eqref{eq:case6-1}}, \\
    2\alpha^2+(\gamma_6^2-\gamma_6)+(\gamma_7^2-\gamma_7)+(\gamma_8^2-3\gamma_8) & \text{in case \eqref{eq:case6-2}}.
\end{cases}
\]
Using the condition $2\alpha+\gamma_6+\gamma_7+\gamma_8=0$, it is straightforward to verify that $-(\mathbf{w}^2+\mathbf{K}\cdot \mathbf{w})\geq 0$ in both cases. Now \eqref{eq:property_of_K} implies \[
2\geq -(v^2+\mathbf{K}\cdot \iota(v))=-(\mathbf{w}^2+\mathbf{K}\cdot \mathbf{w})+(3\beta^2-\beta)+\sum_{i=9}^N (d_i^2+d_i).
\]
Since $-(\mathbf{w}^2+\mathbf{K}\cdot \mathbf{w})\geq 0$ and $d_i^2+d_i\geq 0$, we must have $3\beta^2-\beta \leq 2$, which implies $\beta \in \{0,1\}$, as desired. This completes the proof for \textbf{\hyperlink{Case_4-6}{Case 6}}. 

\vspace{1mm}
\textbf{\hypertarget{Case_4-7}{Case 7}:} $\iota(u_4)=e_6+e_7+e_8+e_9+e_{10}$. In this case, an automorphism of $-\mathbb{Z}^N$ permuting the basis vectors $e_6,\dots,e_{10}$ leaves $\iota(u_1),\dots,\iota(u_4)$ invariant. Up to permuting $\{e_6,\dots,e_{10}\}$, the vector $\mathbf{K}=\sum_{i=1}^N K_ie_i$ satisfies \begin{numcases}
    {(K_1,\dots,K_{10})=}
    (-1,+1,+1,+1,+3,-1,-1,-1,-1,-1), \label{case:4-7-5}\\
    (-3,-1,-1,-1,+1,-1,-1,-1,-1,-1), \label{case:4-7-1} \\
    (+1,+3,-1,-1,+1,-1,-1,-1,-1,-1), \label{case:4-7-2} \\
    (-1,+1,-1,-1,+1,-3,-1,-1,-1,+1), \label{case:4-7-3} \\
    (-1,+1,-1,-1,+1,-1,-1,-1,-1,-1), \label{case:4-7-4}
\end{numcases}
and we may assume that $K_{11}=\cdots=K_N=+1$, except for the case \eqref{case:4-7-4}; in the case \eqref{case:4-7-4}, we assume that $K_{11}=+3$ and $K_{12}=\cdots=K_N=+1$. On the other hand, the vector $\mathbf{T}=\sum_{i=1}^N T_ie_i$ satisfies \[
\mathbf{T}=(e_1-e_2+e_3-e_5+e_6+e_7+e_8+e_9+e_{10})+\mathbf{S}_4,
\]
and calculating $\mathbf{T}\cdot \iota(u_j)$ for $j=1,\dots,4$ gives \[
T_1-T_2=2, \quad T_3-T_4=T_4-T_5=1, \quad T_6+T_7+T_8+T_9+T_{10}=5.
\]

Now, we handle the cases from \eqref{case:4-7-5} to \eqref{case:4-7-4} separately. As in \textbf{\hyperlink{Case_3-2}{Case 2}} of Section \ref{sec:q_3=3}, we shall invoke the partial sum vectors $\mathbf{P}_k$. 

\vspace{1mm}
\textbf{\hypertarget{Case_4-7-0}{Case 7-1} \eqref{case:4-7-5}:} $(K_1,\dots,K_{10})=(-1,+1,+1,+1,+3,-1,-1,-1,-1,-1)$. This case can be easily eliminated. Since $K_5=+3$, we have $T_5\in \{-2,-3\}$. If $T_5=-2$, then the condition $T_3-T_4=T_4-T_5=1$ implies $(T_3,T_4,T_5)=(0,-1,-2)$, which is impossible because $\mathbf{T}$ cannot have coefficients $0$ and $-2$ simultaneously unless $\mathbb{I}_{\operatorname{total}}=-2$ (Lemma \ref{lem:6.4}), but in our case we have $\mathbb{I}_{\operatorname{total}}\in \{-1,0\}$ by Lemma \ref{lem:6.3}. If $T_5=-3$, then $T_4-T_5=1$ implies $T_4=-2$, which is again impossible since $\mathbf{T}$ cannot have coefficients $-2$ and $-3$ simultaneously. 

\vspace{1mm}
\textbf{\hypertarget{Case_4-7-1}{Case 7-2} \eqref{case:4-7-1}:} $(K_1,\dots,K_{10})=(-3,-1,-1,-1,+1,-1,-1,-1,-1,-1)$. In this case, since $K_1=-3$, we have $T_1\in \{+2,+3\}$. If $T_1=+2$, then $T_1-T_2=2$ implies $T_2=0$, which is impossible, so we must have $T_1=+3$, and it is easily seen that \[
(T_1,\dots,T_{10})=(+3,+1,+1,0,-1,+1,+1,+1,+1,+1).
\]

Now, if $\mathbb{I}_{\operatorname{total}}=0$, then $T_{11}=\cdots=T_N=-1$. Thus, $\mathbf{T}+\mathbf{K}=-e_4$, and therefore \[
\mathbf{S}_4+\mathbf{K}=\mathbf{T}+\mathbf{K}-(\mathbf{S}_1+\mathbf{S}_2+\mathbf{S}_3)=-e_1+e_2-e_3-e_4+e_5-(e_6+\cdots+e_{10}).
\]
On the other hand, since the vector $\mathbf{P}_k=\sum_{i=1}^k \iota(v_i)$ ($k=1,\dots,s$) is orthogonal to the vectors $\iota(u_1),\dots,\iota(u_4)$, it must be of the form \[
\mathbf{P}_k=\alpha_k(e_1+e_2)+\beta_k(e_3+e_4+e_5)+\sum_{i=6}^{10} \gamma_{k,i} e_i + \sum_{i=11}^N d_{k,i}e_i, \quad \sum_{i=6}^{10} \gamma_{k,i}=0.
\]
Let us compute $C_k:=(\alpha_k,\beta_k,\vec{\gamma_k})$, where $\vec{\gamma_k}=\sum_{i=6}^{10} \gamma_{k,i} e_i$. For $k=s$, we have \[
\mathbf{P}_s=\mathbf{S}_4=\mathbf{T}-(\mathbf{S}_1+\mathbf{S}_2+\mathbf{S}_3)=2(e_1+e_2)-(e_{11}+\cdots+e_N),
\]
so $C_s=(2,0,\vec{0})$. For $k<s$, Lemma \ref{lem:6.2}(f) implies \begin{equation}\label{eq:4-7-1(1)}
    -1=\mathbf{P}_k\cdot (\mathbf{S}_4+\mathbf{K})= \beta_k. 
\end{equation}
Also, Lemma \ref{lem:6.2}(d) implies \[
2=-\left(\mathbf{P}_k^2+\mathbf{P}_k\cdot \mathbf{K}\right)= (2\alpha_k^2-4\alpha_k)+(3\beta_k^2-\beta_k) +\sum_{i=6}^{10} \gamma_{k,i}^2 +\sum_{i=11}^N (d_{k,i}^2+d_{k,i}).
\]
Since $d_{k,i}^2+d_{k,i}\geq 0$, it follows that \begin{equation}\label{eq:4-7-1(2)}
(2\alpha_k^2-4\alpha_k)+(3\beta_k^2-\beta_k) +\sum_{i=6}^{10} \gamma_{k,i}^2 \leq 2.    
\end{equation}
By \eqref{eq:4-7-1(1)} and \eqref{eq:4-7-1(2)}, we must have $C_k=(1,-1,\vec{0})$ for all $k=1,\dots,s-1$. Now, since $\iota(v_k)=\mathbf{P}_k-\mathbf{P}_{k-1}$ for each $k=1,\dots,s$ (with the convention that $\mathbf{P}_0=0$), it follows that \[
\iota(u_1),\iota(u_2),\iota(u_3),\iota(v_1),\dots,\iota(v_s)\in \langle e_1,\dots,e_5,e_{11},\dots,e_N\rangle \cong -\mathbb{Z}^{N-5}=-\mathbb{Z}^s,
\]
i.e., the image $\iota(Q_{X(2,1)}\oplus Q_{X(3,2)}\oplus Q_{X(n,n-q_4)})$ is a sublattice of $ \langle e_1,\dots,e_5,e_{11},\dots,e_N\rangle$, but this is impossible since $\operatorname{rank}(Q_{X(2,1)}\oplus Q_{X(3,2)}\oplus Q_{X(n,n-q_4)})=s+3$. This completes the proof for \textbf{\hyperlink{Case_4-7-1}{Case 7-2}} in the case $\mathbb{I}_{\operatorname{total}}=0$.

In the case $\mathbb{I}_{\operatorname{total}}=-1$, we may assume $T_{11}=\cdots=T_{N-1}=-1$ and $T_N=0$. For each $k=1,\dots,s$, write $\mathbf{P}_k$ as \[
\mathbf{P}_k=\alpha_k(e_1+e_2)+\beta_k(e_3+e_4+e_5)+\sum_{i=6}^{10} \gamma_{k,i} e_i + \sum_{i=11}^{N-1}d_{k,i}e_i+\delta_ke_N, \quad \sum_{i=6}^{10} \gamma_{k,i}=0,
\]
and let $C_k=(\alpha_k,\beta_k,\vec{\gamma_k},\delta_k)$, where $\vec{\gamma_k}=\sum_{i=6}^{10} \gamma_{k,i} e_i$. For $k=s$, we have \[
\mathbf{P}_s=\mathbf{S}_4=2(e_1+e_2)-(e_{11}+\cdots+e_{N-1}),
\]
so $C_s=(2,0,\vec{0},0)$. For $k<s$, the conditions $\mathbf{P}_k\cdot (\mathbf{S}_4+\mathbf{K})=-1$ and $\mathbf{P}_k^2+\mathbf{P}_k\cdot \mathbf{K}=-2$ give $\beta_k-\delta_k=-1$ and \[
(2\alpha_k^2-4\alpha_k)+(3\beta_k^2-\beta_k)+\sum_{i=6}^{10}\gamma_{k,i}^2 +(\delta_k^2+\delta_k)\leq 2.
\]
Thus, taking into account the condition $\sum_{i=6}^{10}\gamma_{k,i}=0$, so either \[
C_k\in \{(0,0,\vec{0},1), (1,0,\vec{0},1),(2,0,\vec{0},1),(1,-1,\vec{0},0)\}
\]
or $C_k=(1,0,\vec{\gamma_k},1)$ with $-\vec{\gamma_k}^2=\sum_{i=6}^{10}\gamma_{k,i}^2=2$. We now claim that if the set $C:=\{C_1,\dots,C_{s-1}\}$ contains $(1,-1,\vec{0},0)$, then it cannot contain an element of the form $(1,0,\vec{\gamma_k},1)$ with $-\vec{\gamma_k}^2=2$. If this were not the case, there would exist $k,\ell<s$ such that $C_k=(1,-1,\vec{0},0)$ and $C_\ell=(1,0,\vec{\gamma_\ell},1)$ with $-\vec{\gamma_\ell}^2=2$. Assume $k<\ell$; the other case is similar. Then $\mathbf{P}_\ell-\mathbf{P}_k$ is given by \[
\mathbf{P}_\ell-\mathbf{P}_k=(e_3+e_4+e_5)+\vec{\gamma_\ell} +\sum_{i=11}^{N-1} d_ie_i+e_N, \quad d_i=d_{\ell,i}-d_{k,i}.
\]
By Lemma \ref{lem:6.2}(e), it follows that \[
2=-\left((\mathbf{P}_\ell-\mathbf{P}_k)^2+(\mathbf{P}_\ell-\mathbf{P}_k)\cdot \mathbf{K} \right)=6+\sum_{i=11}^{N-1}(d_i^2+d_i) \geq 6,
\]
which is a contradiction. This proves our claim.

Now, if $C$ contains $(1,-1,\vec{0},0)$, then $C_k$ is not of the form $(1,0,\vec{\gamma_k},1)$ with $-\vec{\gamma_k}^2=2$ for all $k<s$, so $\vec{\gamma_k}=\vec{0}$ for all $k=1,\dots,s$, and we obtain a contradiction as in the case $\mathbb{I}_{\operatorname{total}}=0$. If $C$ does not contain $(1,-1,\vec{0},0)$, then $\beta_k=0$ for all $k=1,\dots,s$, and we may obtain a contradiction as in \textbf{\hyperlink{Case_4-5}{Case 5}}. This completes the proof for \textbf{\hyperlink{Case_4-7-1}{Case 7-2}}.

\vspace{2mm}
The remaining cases \eqref{case:4-7-2}, \eqref{case:4-7-3}, and \eqref{case:4-7-4} are similar to \textbf{\hyperlink{Case_4-7-1}{Case 7-2}}, so we summarize the arguments briefly.

\vspace{1mm}
\textbf{\hypertarget{Case_4-7-2}{Case 7-3} \eqref{case:4-7-2}:} $(K_1,\dots,K_{10})=(+1,+3,-1,-1,+1,-1,-1,-1,-1,-1)$. 
In this case, we have \[
(T_1,\dots,T_{10})=(-1,-3,+1,0,-1,+1,+1,+1,+1,+1).
\]
If $\mathbb{I}_{\operatorname{total}}=0$, then for $k=1,\dots,s$, write \begin{equation}\label{eq:P_k_for_case_7}
    \mathbf{P}_k=\alpha_k(e_1+e_2)+\beta_k(e_3+e_4+e_5)+\sum_{i=6}^{10} \gamma_{k,i} e_i + \sum_{i=11}^N d_{k,i}e_i, \quad \sum_{i=6}^{10} \gamma_{k,i}=0,
\end{equation} 
and let $C_k=(\alpha_k,\beta_k,\vec{\gamma_k})$, where $\vec{\gamma_k}=\sum_{i=6}^{10} \gamma_{k,i} e_i$. For $k=s$, we have $C_s=(-2,0,\vec{0})$, and for $k<s$, the conditions $\mathbf{P}_k\cdot (\mathbf{S}_4+\mathbf{K})=-1$ and $\mathbf{P}_k^2+\mathbf{P}_k\cdot \mathbf{K}=-2$ imply $C_k=(-1,-1,\vec{0})$. Thus, $\vec{\gamma_k}=\vec{0}$ for all $k=1,\dots,s$, and we may proceed exactly as in \textbf{\hyperlink{Case_4-7-1}{Case 7-2}} to obtain a contradiction.

If $\mathbb{I}_{\operatorname{total}}=-1$, we may assume that $T_{11}=\cdots=T_{N-1}=-1$ and $T_N=0$. For $k=1,\dots,s$, write \begin{equation}\label{eq:P_k_for_case_7(2)}
    \mathbf{P}_k=\alpha_k(e_1+e_2)+\beta_k(e_3+e_4+e_5)+\sum_{i=6}^{10} \gamma_{k,i} e_i + \sum_{i=11}^{N-1}d_{k,i}e_i+\delta_ke_N, \quad \sum_{i=6}^{10} \gamma_{k,i}=0,
\end{equation}
and let $C_k=(\alpha_k,\beta_k,\vec{\gamma_k},\delta_k)$, where $\vec{\gamma_k}=\sum_{i=6}^{10} \gamma_{k,i} e_i$. For $k=s$, we have $C_s=(-2,0,\vec{0},0)$, and for $k<s$, we have either \[
C_k\in \{(0,0,\vec{0},1), (-1,0,\vec{0},1),(-2,0,\vec{0},1),(-1,-1,\vec{0},0)\}
\]
or $C_k=(-1,0,\vec{\gamma_k},1)$ with $-\vec{\gamma_k}^2=\sum_{i=6}^{10}\gamma_{k,i}^2=2$. Now we proceed exactly as in \textbf{\hyperlink{Case_4-7-1}{Case 7-2}} to show that if the set $\{C_1,\dots,C_{s-1}\}$ contains $(-1,-1,\vec{0},0)$, then it cannot contain an element of the form $(-1,0,\vec{\gamma_k},1)$ with $-\vec{\gamma_k}^2=2$, thereby obtaining a contradiction.

\vspace{1mm}
\textbf{\hypertarget{Case_4-7-3}{Case 7-4} \eqref{case:4-7-3}:} $(K_1,\dots,K_{10})=(-1,+1,-1,-1,+1,-3,-1,-1,-1,+1)$. In this case, we have \[
(T_1,\dots,T_{10})=(+1,-1,+1,0,-1,+3,+1,+1,+1,-1).
\]
If $\mathbb{I}_{\operatorname{total}}=0$, write $\mathbf{P}_k$ as \eqref{eq:P_k_for_case_7} for $k=1,\dots,s$, and let $C_k=(\alpha_k,\beta_k,\vec{\gamma_k})$, where $\vec{\gamma_k}=\sum_{i=6}^{10} \gamma_{k,i} e_i$. For $k=s$, we have $C_s=(0,0,\vec{\gamma_s})$ with $\vec{\gamma_s}=2(e_6-e_{10})$. For $k<s$, the conditions $\mathbf{P}_k\cdot (\mathbf{S}_4+\mathbf{K})=-1$ and $\mathbf{P}_k^2+\mathbf{P}_k\cdot \mathbf{K}=-2$ imply $\beta_k=-1$ and 
\[
    2\alpha_k^2 + (3\beta_k^2-\beta_k) +\sum_{i=6}^{10}\gamma_{k,i}^2 -2\gamma_{k,6}+2\gamma_{k,10}\leq 2,
\]
which gives $C_k=(0,-1, \vec{\gamma_k})$ with $\vec{\gamma_k}=e_6-e_{10}$. Thus, $\alpha_k=0$ for all $k=1,\dots,s$. This implies that the orthogonal complement of the image of the embedding \eqref{eq:main_embedding} is generated by the vector $e_1+e_2$, which has square $-2\neq -30 n$, contradicting condition \ref{item:orthogonal_complement}. 

If $\mathbb{I}_{\operatorname{total}}=-1$, we may assume that $T_{11}=\cdots=T_{N-1}=-1$ and $T_N=0$. For $k=1,\dots,s$, write $\mathbf{P}_k$ as \eqref{eq:P_k_for_case_7(2)}, and let $C_k=(\alpha_k,\beta_k,\vec{\gamma_k},\delta_k)$, where $\vec{\gamma_k}=\sum_{i=6}^{10} \gamma_{k,i} e_i$. For $k=s$, we have $C_s=(0,0,\vec{\gamma_s},0)$ with $\vec{\gamma_s}=2(e_6-e_{10})$, and for $k<s$,  either \[
C_k\in \{(1,0, \vec{\gamma_k},1), (-1,0, \vec{\gamma_k},1), (0,-1,  \vec{\gamma_k},0)\}, \quad  \vec{\gamma_k}=e_6-e_{10},
\]
or $C_k$ is of the form \[
C_k=(0,0, \vec{\gamma_k},1);
\]
the precise possibilities for \(\vec{\gamma_k}\) will not be needed. Now we proceed as in \textbf{\hyperlink{Case_4-7-1}{Case 7-2}} to show that if the set $C=\{C_1,\dots,C_{s-1}\}$ contains $(0,-1,e_6-e_{10},0)$, then it cannot contain $(1,0,e_6-e_{10},1)$ or $(-1,0,e_6-e_{10},1)$. It follows that either $\alpha_k=0$ for all $k=1,\dots,s$, or $\beta_k=0$ for all $k=1,\dots,s$. In the former case, we can obtain a contradiction as in the preceding paragraph, and in the latter case, we can obtain a contradiction as in \textbf{\hyperlink{Case_4-5}{Case 5}}.

\vspace{1mm}
\textbf{\hypertarget{Case_4-7-4}{Case 7-5} \eqref{case:4-7-4}:} $(K_1,\dots,K_{10})=(-1,+1,-1,-1,+1,-1,-1,-1,-1,-1)$. 
In this case, we have \[
(T_1,\dots,T_{10})=(+1,-1,+1,0,-1,+1,+1,+1,+1,+1).
\]
If $\mathbb{I}_{\operatorname{total}}=0$, then $T_{11}=-3$ and $T_{12}=\cdots=T_N=-1$. For each $k=1,\dots,s$, write \[
\mathbf{P}_k=\alpha_k(e_1+e_2)+\beta_k(e_3+e_4+e_5)+\sum_{i=6}^{10} \gamma_{k,i} e_i +\delta_ke_{11}+ \sum_{i=12}^N d_{k,i}e_i, \quad \sum_{i=6}^{10} \gamma_{k,i}=0,
\]
and let $C_k=(\alpha_k,\beta_k,\vec{\gamma_k},\delta_k)$, where $\vec{\gamma_k}=\sum_{i=6}^{10} \gamma_{k,i} e_i$. For $k=s$, we have $C_s=(0,0,\vec{0},-3)$, and for $k<s$, we have \[
C_k\in \{(0,-1,\vec{0},-1), (0,-1,\vec{0},-2)\}.
\]
Thus, $\alpha_k=0$ for all $k=1,\dots,s$, and we can obtain a contradiction as in \textbf{\hyperlink{Case_4-7-3}{Case 7-4}}.

If $\mathbb{I}_{\operatorname{total}}=-1$, then we may assume $T_{11}=-3$, $T_{12}=\cdots=T_{N-1}=-1$, and $T_N=0$. For each $k=1,\dots,s$, write \[
\mathbf{P}_k=\alpha_k(e_1+e_2)+\beta_k(e_3+e_4+e_5)+\sum_{i=6}^{10} \gamma_{k,i} e_i +\delta_ke_{11}+ \sum_{i=12}^{N-1} d_{k,i}e_i+\eta_ke_N, \quad \sum_{i=6}^{10} \gamma_{k,i}=0,
\]
and let $C_k=(\alpha_k,\beta_k,\vec{\gamma_k},\delta_k,\eta_k)$, where $\vec{\gamma_k}=\sum_{i=6}^{10} \gamma_{k,i} e_i$. For $k=s$, we have $C_s=(0,0,\vec{0},-3,0)$, and for $k<s$, we have either \[
C_k\in \left\{\begin{aligned}
    & (0,0,\vec{0},-3,1),\, (0,0,\vec{0},-2,1), \, (0,0,\vec{0},-1,1), \, (0,0,\vec{0},0,1), \,(0,-1,\vec{0},-2,0), \\
    &(0,-1, \vec{0},-1,0), \, (1,0,\vec{0},-2,1), \, (1,0,\vec{0},-1,1),\, (-1,0,\vec{0},-2,1), (-1,0,\vec{0},-1,1)
\end{aligned} \right\}
\]
or \[
C_k\in \{(0,0,\vec{\gamma_k},-2,1), \, (0,0,\vec{\gamma_k},-1,1)\}, \quad -\vec{\gamma_k}^2=2.
\]
Using the condition $(\mathbf{P}_\ell-\mathbf{P}_k)^2+(\mathbf{P}_\ell-\mathbf{P}_k)\cdot \mathbf{K} =-2$, we can proceed as in \textbf{\hyperlink{Case_4-7-1}{Case 7-2}} to show that if the set $\{C_1,\dots,C_{s-1}\}$ contains $(0,-1,\vec{0},-2,0)$ or $(0,-1,\vec{0},-1,0)$, then it cannot contain $(\pm1 ,0,\vec{0},-2,1)$ or $(\pm1,0,\vec{0},-1,1)$. It follows that either $\alpha_k=0$ for all $k=1,\dots,s$, or $\beta_k=0$ for all $k=1,\dots,s$, so we can obtain a contradiction as in \textbf{\hyperlink{Case_4-7-3}{Case 7-4}}. This completes the proof for \textbf{\hyperlink{Case_4-7}{Case 7}}, and hence the proof of Theorem \ref{thm:aMY}.

\clearpage

\end{document}